\documentclass[11pt]{article}

\usepackage[english]{babel}
\usepackage{amsmath,amssymb,amsthm,amsfonts}
\usepackage{xcolor}

\usepackage[shortlabels]{enumitem}
\usepackage[numbers,sort&compress]{natbib}
\usepackage[a4paper,total={6in,9in}]{geometry}
\usepackage[colorlinks=true,linkcolor=black,citecolor=black,urlcolor=blue]{hyperref}
\usepackage{mathrsfs} 
\numberwithin{equation}{section}
\usepackage{comment}
\newtheorem{theorem}{Theorem}[section]
\newtheorem{lemma}[theorem]{Lemma}
\newtheorem{proposition}[theorem]{Proposition}
\newtheorem{corollary}[theorem]{Corollary}
\usepackage{authblk}
\theoremstyle{definition}
\newtheorem{definition}[theorem]{Definition}
\newtheorem{remark}[theorem]{Remark}
\newtheorem{ass}[theorem]{Assumption}
\newcommand{\C}{\mathbb C}

\newcommand{\Z}{\mathbb Z}
\newcommand{\T}{\mathbb T}

\newcommand{\cL}{\mathcal L}

\newcommand{\cO}{\mathcal O}
\newcommand{\cX}{\mathcal X}

\newcommand{\capc}{\operatorname{cap}}

\hypersetup{
  pdftitle={Extremal Functions and Widom Factors on Compact Riemann Surfaces},
  pdfauthor={Sampad Lahiry}
}

\usepackage[colorlinks=true,linkcolor=black,citecolor=black,urlcolor=blue]{hyperref}

\begin{document}

\title{Extremal Functions and Widom Factors on Compact Riemann Surfaces}
\author{Sampad Lahiry}
\affil{
Department of Mathematics, KU Leuven, Celestijnenlaan
200 B bus 2400, 3001 Leuven,  Belgium}
\affil{School of Mathematics and Statistics, The University of Melbourne, Victoria 3010, Australia}
\affil{ sampad.lahiry@kuleuven.be}
\date{}
\maketitle
\begin{abstract}
We study the Chebyshev extremal problem on a compact Riemann surface
$X$ of genus $g>0$. As an analog to monic polynomials, we consider
admissible meromorphic  functions having no poles away from a marked
point $P_{\infty}$ (with adequate normalisation).  For a nonpolar compact set
$E\subset X\setminus\{P_{\infty}\}$, we show that the $n$-th root of
the Chebyshev constant $t_n(E)$ converges to the capacity of $E$, and
we establish the corresponding Bernstein--Walsh inequality.

In the second part of the paper, using a Cauchy kernel adapted to Riemann surfaces, we study the refined
Szeg\H{o}--Widom asymptotics for the extremals as well as the Widom
factors
\[
    W_n(E)=\frac{t_n(E)}{\operatorname{cap}(E)^n},
\]
assuming that $E$ is a finite union of $p$ closed discs with analytic
boundaries. The geometry of the Schottky double, which has genus
$2g+p-1$, enters explicitly into these asymptotics. We find that there is no analogue of Faber-type asymptotics even for one single boundary curve. We conclude by constructing  explicit examples in genus $1$ using the Weierstrass-$\wp$ function.
\end{abstract}
\tableofcontents
\section{Introduction}

\subsection{The Chebyshev problem in the plane}

In the memoir of 1854, Chebyshev
\cite{Chebyshev1854} was led to ask which real polynomial of degree $n$ with leading
coefficient $1$ deviates least from zero on $[-1,1]$. The answer is
\[
  2^{1-n}\cos\bigl(n\arccos x\bigr),
\]
and the minimal deviation is $2^{1-n}$. The same  question can be posed for a  compact set $E$.
For compact $E\subset\mathbb C$ write $\|f\|_E=\sup_E|f|$, let $\mathcal P_{n-1}$ denote the
polynomials of degree at most $n-1$ for $n\geq 1$, and set
\begin{equation}\label{eq:tn-plane}
  t_n(E)\;:=\;\min_{p\in\mathcal P_{n-1}}\bigl\|z^n+p\bigr\|_E .
\end{equation}
If $E$ is compact and non polar the minimum is attained by a unique monic polynomial $T_n$. It was Faber
\cite{Faber1919} who posed the problem in this generality and who attached Chebyshev's name to
the extremal polynomials of a general set.

    Closed formulas are rare \cite{MDR}. For a closed disc of radius $r$, centered at $a$, one has $T_n(z)=(z-a)^n$ and
$t_n=r^n$; for a lemniscate $E=\{|P|\le\rho\}$ with $P$ monic of degree $k$, Faber observed that
$T_{km}=P^m$ for every $m\ge 1$; and an ellipse has the same Chebyshev polynomials as the segment
joining its foci. Outside such algebraic coincidences one has to be content with asymptotics, and
this is where potential theory enters.

\subsection{Capacity and the theorem of Fekete and Szeg\H{o}}

Fekete \cite{Fekete1923} introduced the transfinite diameter: with
$\delta_n(E)$ the $n$-th diameter computed from an $n$-point Fekete set, the sequence
$(\delta_n(E))_n$ is non-increasing, and has a  limit $d(E)$ . Szeg\H{o} \cite{Szego1924}, building on \cite{Faber1919,Fekete1923}, identified $d(E)$ with
the Chebyshev constant $\lim_n t_n(E)^{1/n}$ and with the logarithmic capacity
$\operatorname{cap}(E)=e^{-V_E}$, $V_E$ the Robin constant. The resulting three-fold identity,
which Saff \cite{SaffSurvey} calls the fundamental theorem of classical potential theory, contains
in particular the root asymptotic
\begin{equation}\label{eq:root-plane}
  \lim_{n\to\infty}t_n(E)^{1/n}=\operatorname{cap}(E) ,
\end{equation}
valid for every non-polar compact $E\subset\mathbb C$; see \cite[Ch.~5]{Ransford} or
\cite[Ch.~I, III]{SaffTotik}.

The two halves of \eqref{eq:root-plane} are of quite different character. The lower bound is soft \cite{CSZ3}.
The set $E$ is contained in the lemniscate $\{|T_n|\le t_n(E)\}$, whose capacity is
$t_n(E)^{1/n}$, and monotonicity  gives Szeg\H{o}'s inequality
\begin{equation}\label{eq:szego}
  t_n(E)\;\ge\;\operatorname{cap}(E)^n ,\qquad n\ge 1;
\end{equation}
equivalently, one may argue from the Bernstein--Walsh inequality
$|P(z)|\le\|P\|_E\,e^{n g_E(z)}$. The upper bound is more geometric, and the classical
route runs through Fekete points: if $q_1,\dots,q_n$ is an $n$-point Fekete set for $E$, then the
monic polynomial $\prod_{j}(z-q_j)$ already has the correct exponential size on $E$, because the
Fekete property compares $\delta_{n+1}(E)$ with $\delta_n(E)$. It is this proof, that we shall transplant to a Riemann surface.

\subsection{Widom factors}

In view of \eqref{eq:szego} it is natural to normalise and to study the \emph{Widom factors}
\begin{equation}\label{eq:widomfactor-plane}
  W_n(E)\;:=\;\frac{t_n(E)}{\operatorname{cap}(E)^n}\;\ge\;1 ,
\end{equation}
a name introduced by Goncharov and Hatino\u{g}lu \cite{GoncharovHatinoglu}. For a disc $W_n\equiv 1$; for a lemniscate of a degree $k$ polynomial $W_{km}=1$; since
$\operatorname{cap}([-1,1])=\tfrac12$ and $t_n=2^{1-n}$, the interval has $W_n\equiv 2$. At the
other extreme the Widom factors may be unbounded: this happens for the
Julia set of $(z-\lambda)^2$ with $\lambda>2$ \cite{BGH}, and Goncharov and Hatino\u{g}lu
\cite{GoncharovHatinoglu} construct Cantor-type sets along which $W_n$ grows subexponentially at
any prescribed rate. Whether the Widom factors of a compact set with finitely many components can
be unbounded is, at present, not known; see \cite{CSZ4} and the discussion in \cite{Rubin}.

 Christiansen, Simon and
Zinchenko \cite{CSZ3} showed that $W_{n_0}(E)=1$ for a single index $n_0$ precisely when the outer
boundary of $E$ is the lemniscate $\{|P|=1\}$ of a polynomial of degree $n_0$. On the real line the
true lower bound is twice as large. Schiefermayr \cite{Schiefermayr} proved
\begin{equation}\label{eq:schiefermayr}
  t_n(E)\;\ge\;2\operatorname{cap}(E)^n ,\qquad E\subset\mathbb R,
\end{equation}
and here equality holds for a given $n$ exactly when $E=P^{-1}([-1,1])$ for a polynomial $P$ of
degree $n$ \cite{Totik2014,CSZ3}, while $\lim_n W_n(E)=2$ forces $E$ to be an interval
\cite{Totik2014}.

\subsection{Widom's theory}

For one Jordan region with analytic boundary the asymptotics were settled by Faber
\cite{Faber1919}. Let $\Omega=(\mathbb C\cup\{\infty\})\setminus E$ and let $B$ be the conformal
map of $\Omega$ onto the unit disc normalised by $B(z)=\operatorname{cap}(E)/z+O(z^{-2})$ at
infinity. Then $W_n(E)\to1$ and
\begin{equation}\label{eq:faber}
  \frac{T_n(z)\,B(z)^n}{\operatorname{cap}(E)^n}\longrightarrow 1
\end{equation}
uniformly on  $\overline\Omega$. As soon as $E$ has more than one component
this breaks down, and for a structural reason: the map $B$ no longer exists as a function.

Widom's memoir \cite{Widom69} is the definitive treatment of the multiply connected case, and it
is the model for the second half of the present paper. Let $E$ be a finite union of $p$ disjoint
Jordan regions with $C^{2p+1+}$ boundary. In place of the Riemann map one takes the multivalued
$B$ determined by $|B|=e^{-g_E}$ and by the normalisation at $\infty$; continuation of $B$ around
a closed curve multiplies it by a unimodular constant, so $B$ is a section of a flat unitary line
bundle rather than a function, and the associated character $\chi_E\in\pi_1(\Omega)^{*}$ records
the harmonic measures of the components. For each character $\chi$ and $H^{\infty}(\Omega,\chi)$ defined to be the space of  holomorphic sections on $\Omega$, Widom introduces the extremal
problem
\begin{equation}\label{eq:widom-min}
  \mu(\chi)\;=\;\inf\bigl\{\|F\|_\infty:\;F\in H^\infty(\Omega,\chi),\;F(\infty)=1\bigr\},
\end{equation}
solved by a unique \emph{Widom minimiser} $F_\chi$, and proves \cite[Thm.~8.3]{Widom69} that
\begin{equation}\label{eq:SW-plane}
  W_n(E)-\mu\bigl(\chi_E^{\,n}\bigr)\longrightarrow 0,
  \qquad
  \frac{T_n(z)B(z)^n}{\operatorname{cap}(E)^n}-F_{\chi_E^{\,n}}(z)\longrightarrow 0 ,
\end{equation}
together with the corresponding statements for a positive 
weight on $\partial E$ in the class of $C^{2p-1+}$. This is
what is now called Szeg\H{o}--Widom asymptotics. The Widom factors therefore need not converge;
they follow the almost periodic sequence $n\mapsto\mu(\chi_E^{\,n})$, whose set of limit points is
a finite union of closed intervals, reducing to finitely many points exactly when $\chi_E$ has
finite order, that is, when the harmonic measures of the components at $\infty$ are rational.
Akhiezer \cite{Achieser1932} had already found this dichotomy for two intervals. Widom's proof
rests on a duality between the $H^\infty$ minimum \eqref{eq:widom-min} and an $H^1$ residue maximum
\cite[\S5]{Widom69}, on a product formula for $F_\chi$ in terms of Green's functions, and on two
systems of transcendental equations locating the $p-1$ zeros entering that formula. Each of these
has an intrinsic counterpart below.

The natural finiteness condition here is due to Parreau \cite{Parreau}, who met it on Riemann
surfaces, and was later named after Parreau and Widom by Hasumi \cite{Hasumi}: a regular compact
set is a \emph{Parreau--Widom set} when
$\mathrm{PW}(E)=\sum_{\nabla g_E(c)=0}g_E(c)<\infty$. Widom \cite{Widom71a} proved that this holds
if and only if $H^\infty(\Omega,\chi)\ne 0$ for every character $\chi$, which is what makes
\eqref{eq:widom-min} a non-degenerate problem in the first place. Christiansen, Simon and Zinchenko
\cite{CSZ1} proved the Totik--Widom bound $W_n(E)\le 2\exp[\mathrm{PW}(E)]$ for Parreau--Widom
subsets of $\mathbb R$, and in the same paper settled the conjecture of \cite{Widom69} that finite
gap subsets of $\mathbb R$ enjoy Szeg\H{o}--Widom asymptotics; with Yuditskii \cite{CSYZ2} they
extended this to Parreau--Widom sets. Widom \cite{Widom71b} had
asked whether $F_\chi$ depends continuously on $\chi$, and it is addressed by works of Hayashi and Hasumi \cite{Hayashi,Hasumi}. In the plane,
Totik--Widom bounds are known for quasidiscs and quasiconformal arcs
\cite{TotikVarga,Andrievskii,AndrievskiiNazarov}, and with no boundary regularity whatsoever one
still has $W_n(E)=O(\log n)$ for sets with finitely many components \cite{Andrievskii}. A survey of
the whole circle of ideas is \cite{CSZreview}.

\subsection{Arcs}

The case of arcs is different, and instructive. Widom treated it in \cite[\S11]{Widom69}, obtained
the asymptotic upper bound $\limsup_n W_n(E)\le 2\exp[\mathrm{PW}(E)]$ when a component of $E$ is
an arc,  sharp for finite unions of intervals, where he also proved the asymptotics of the norms
\cite[Thm.~11.5]{Widom69} --- and conjectured that the corresponding Szeg\H{o}--Widom asymptotics
holds, with a single arc behaving like an interval. Both expectations turned out to be wrong.
Thiran and Detaille \cite{ThiranDetaille} computed the circular arc
$\Gamma_\alpha=\{e^{i\theta}:|\theta|\le\alpha\}$ and found
\begin{equation}\label{eq:arc}
  \lim_{n\to\infty}W_n(\Gamma_\alpha)\;=\;2\cos^2(\alpha/4)\;=\;1+\cos(\alpha/2),
\end{equation}
which decreases from $2$ to $1$ as the arc closes up into a circle, so that every value in
$(1,2)$ occurs. Totik and Yuditskii \cite{TotikYuditskii} disproved the conjecture in its general
form. Arcs are not treated in this paper: the sets we consider in the second
half are finite unions of closed discs, which is the surface analogue of Widom's system of curves. We refer to \cite{Alpan,CSZreview,Eichinger,Totik2014} for more results in this direction.

\section{Statement of results}
\subsection{The Green kernel and the capacity on a compact surface}\label{ss:potential}
 
All of the above lives on the Riemann sphere. The same questions  in \eqref{eq:tn-plane} or
\eqref{eq:root-plane} can be posed in a compact Riemann surface, and the object of this paper is to say what both
sides of \eqref{eq:root-plane} mean on a compact Riemann surface and to prove that they still agree.

Throughout, $X$ is a compact Riemann surface of genus $g$, the point $P_\infty\in X$ is fixed, and
$\zeta$ is a fixed local coordinate centred at $P_\infty$.  If $u$ is harmonic on $X$ away from finitely many points and
$u=-c_j\log|z_j|+O(1)$ at $p_j$, then $\Delta u=-2\pi\sum_j c_j\delta_{p_j}$, while
$\int_X\Delta u=0$ by Stokes; hence $\sum_j c_j=0$. A Green function with a single logarithmic
pole therefore does not exist, and every kernel we might use has to carry a compensating logarithmic singularity
somewhere. We choose this to be $P_\infty$.
 
\begin{definition}\cite{BGK23}\label{def:green}
The \emph{normalised bipolar Green kernel with pole at $P_\infty$} is the unique symmetric function
\[
  G=G_{P_\infty}:\ \bigl(X\setminus\{P_\infty\}\bigr)^2\setminus\{p=q\}\longrightarrow\mathbb R
\]
with the following properties:
\begin{enumerate}
  \item[(i)]   for fixed $q\ne P_\infty$, the function $p\mapsto G(p,q)$ is harmonic on
               $X\setminus\{q,P_\infty\}$;
  \item[(ii)]  if $z$ is a local coordinate at $q$, then
               $G(p,q)=-\log|z(p)|+O(1)$ as $p\to q$;
  \item[(iii)] at $P_\infty$, \ $G(p,q)=\log|\zeta(p)|+o(1)$ as $p\to P_\infty$;
  \item[(iv)]  $G(p,q)=G(q,p)$.
\end{enumerate}
\end{definition}
 
The charges are $+1$ at $q$ and $-1$ at $P_\infty$, as they must be. Existence is classical:
integrate the third-kind differential with residues $+1$ at $q$ and $-1$ at $P_\infty$ and purely
imaginary periods, or read $G$ with theta functions and  prime form \cite{FarkasKra,Fay}. For
$X=\widehat{\mathbb C}$, $P_\infty=\infty$ and $\zeta=1/z$ one recovers the logarithmic kernel
\[
  G(z,w)=\log\frac1{|z-w|}.
\] At this point, note (iii).
We point out the Bipolar Green's kernel is unique, and $o(1)$ in item (iii) is possible only due the fact that the local coordinate $\zeta$ at $P_\infty$ is fixed, see \cite[Lemma 4.1]{Kui25}
 
Let now $E\subset X\setminus\{P_\infty\}$ be compact and non-polar, the  classical
logarithmic potential theory (that we require for what follows) transfers verbatim \cite{BGK23}. Writing $\mathcal M(E)$ for the probability measures
on $E$ and
\[
  U^{\mu}(p):=\int_E G(p,q)\,d\mu(q),
  \qquad
  I(\mu):=\iint_{E\times E} G(p,q)\,d\mu(p)\,d\mu(q),
\]
the Robin constant, the capacity and the equilibrium measure are
\begin{equation}\label{eq:cap}
  V_E:=\inf_{\mu\in\mathcal M(E)}I(\mu),
  \qquad
  \operatorname{cap}(E):=e^{-V_E},
  \qquad
  \mu_E:=\text{the minimiser} .
\end{equation}
Frostman's theorem for the kernel $G$ then produces the equilibrium Green function
$g_E:=V_E-U^{\mu_E}$: it is non-negative, vanishes quasi-everywhere on $E$, and is harmonic on
$X\setminus(E\cup\{P_\infty\})$. Since $\mu_E$ has total mass one, (iii) integrates to
$U^{\mu_E}(p)=\log|\zeta(p)|+o(1)$, and therefore
\begin{equation}\label{eq:gE-at-Pinfty}
  g_E(p)=-\log|\zeta(p)|-\log\operatorname{cap}(E)+o(1),\qquad p\to P_\infty.
\end{equation}
So $g_E$ is the  Green function of $\Omega:=X\setminus E$
with pole at $P_\infty$, and \eqref{eq:gE-at-Pinfty} is the exact analogue of
$g_E(z)=\log|z|-\log\operatorname{cap}(E)+o(1)$ in the plane; it is the expansion that makes
$\operatorname{cap}(E)$ appear in the Bernstein--Walsh inequality on $X$.
Balayage gives in the same way a Green function $g_E(\cdot,a)$ with pole at an arbitrary
$a\in\Omega\setminus\{P_\infty\}$, Let \(a\in X\setminus E\), \(a\ne P_\infty\).  Let \(\omega_a\) be the
balayage of the point mass \(\delta_a\) onto \(E\), characterized by
\[
        U^{\omega_a}(p)=G(p,a)
        \qquad\text{for q.e. }p\in E .
\]
Set
\begin{equation}\label{eq:gp}
g_E(p,a):=G(p,a)-U^{\omega_a}(p),
        \qquad p\in X\setminus E .
\end{equation}
Then \(g_E(\cdot,a)\ge0\), \(g_E(\cdot,a)=0\) quasi-everywhere on \(E\), and
\(g_E(\cdot,a)\) is harmonic on
$X\setminus(E\cup\{a\}).$
Moreover, if \(z\) is a local coordinate centered at \(a\), then
\[
        g_E(p,a)=-\log|z(p)|+O(1),
        \qquad p\to a.
\]
At \(P_\infty\), the singularities cancel, and \(g_E(\cdot,a)\) is harmonic.
Thus \(g_E(\cdot,a)\) is the Dirichlet Green function of
$ \Omega:=X\setminus E$
with a logarithmic singularity \(a\).

\subsection{The extremal problem on a compact Riemann surface}
 
A monic polynomial of degree $n$ is a meromorphic function on $\widehat{\mathbb C}$ whose only pole
is at $\infty$, of order exactly $n$, normalised so that its leading Laurent coefficient in the
coordinate $\zeta=1/z$ equals one. We can use this notion on a compact Riemann surface.  As before fix $P_\infty\in X$ and a local coordinate
$\zeta$ centered at $P_\infty$, and let $L(NP_\infty)$ denote the space of meromorphic functions on
$X$ with no pole outside $P_\infty$ and pole order at most $N$ there.  In recent literature \cite{AMF26, Ber23, DL26} they have been referred to as polynomials on a Riemann surface. Call $f\in L(NP_\infty)$
\emph{admissible of order $N$} if
\[
  f(p)=\zeta(p)^{-N}+O\bigl(\zeta(p)^{-N+1}\bigr),\qquad p\to P_\infty ,
\]
and for compact non-polar $E\subset X\setminus\{P_\infty\}$ put
\begin{equation}\label{eq:def1}
   \hspace{-0.3cm}  t_N(E,P_\infty):=\inf\bigl\{\|f\|_E:\ f\ \text{admissible of order}\ N\bigr\},
  \qquad
  W_N(E,P_\infty):=\frac{t_N(E,P_\infty)}{\operatorname{cap}(E)^N} ,
\end{equation}

the capacity being computed from the bipolar Green kernel normalised at $P_\infty$
(Definition~\ref{def:green}). For $X=\widehat{\mathbb C}$, $P_\infty=\infty$ and $\zeta=1/z$ this
is exactly \eqref{eq:tn-plane} and \eqref{eq:widomfactor-plane}.

Two features of the surface problem are visible already in the definition. First, admissible
functions need not exist. Riemann--Roch \cite{FarkasKra} gives
$\dim L(NP_\infty)=N-g+1$ for $N\ge 2g$, so admissible functions of every order $N\ge2g$ do exist
and the asymptotic statements below are unaffected. Second, and more consequentially, there is no
factorisation into linear factors. On $\widehat{\mathbb C}$ the function $z-q$ has divisor
$q-\infty$ and modulus $e^{-G(z,q)}$; on $X$ a meromorphic function with divisor $q-P_\infty$
would exhibit $X$ as a one-sheeted cover of the sphere, so for $g\ge1$ and $q\ne P_\infty$ no such
function exists. What survives is a section of a flat bundle, and the whole difficulty of the proof
is that the bundle depends on $q$. Our first result is that the root asymptotic is nevertheless
untouched.

\begin{theorem}\label{thm:main-root}
Let $X$ be a compact Riemann surface, let $P_\infty\in X$, and let $E\subset X\setminus\{P_\infty\}$
be compact and non-polar. With the normalisation of Definition~\ref{def:green},
\[
  \lim_{N\to\infty}t_N(E,P_\infty)^{1/N}=\operatorname{cap}(E),
  \qquad\text{equivalently}\qquad
  \lim_{N\to\infty}W_N(E,P_\infty)^{1/N}=1 .
\]
\end{theorem}

The proof follows the Fekete route, and the two obstructions above are what has to be dealt with.
The analogue of the factor $z-q$ is a multivalued function  $B_q$ with divisor $q-P_\infty$ and
$|B_q|=e^{-G(\cdot,q)}$; it exists, but only as a meromorphic section of a flat unitary line
bundle $L_{\chi_q}$ (Lemma~\ref{lem:compact-green-factor}). A Fekete product $B_{D_n}$ over an $n$-point
Fekete divisor $D_n$ is therefore a section of $L_{\chi_{D_n}}$, and the character
$\chi_{D_n}$ moves with $n$ in a way we do not control. The device that repairs this is a
Riemann--Roch correction which is uniform in the character: fixing $m\ge 2g$, we produce for every
$\chi\in\mathcal X(X)$ a section
\[
  s_\chi\in H^0\bigl(X,L_\chi\otimes\mathcal O(mP_\infty)\bigr),
  \qquad
  s_\chi(p)=\zeta(p)^{-m}+O\bigl(\zeta(p)^{-m+1}\bigr),
  \qquad
  \sup_E|s_\chi|\le C_E ,
\]
with $C_E$ independent of $\chi$ (Lemma~\ref{lem:uniform-correction}) (here $H^0\bigl(X,L_\chi\otimes\mathcal O(mP_\infty)\bigr)$ is the space of holomorphic sections of the line bundle $L_\chi\otimes\mathcal O(mP_\infty)$). Multiplying $B_{D_n}$ by the correction section of the inverse character produces an
honest meromorphic function on $X$, admissible of order $n+m$, at the cost of a bounded factor and
a bounded shift of the pole order --- neither of which is visible after extracting
$(n+m)$-th roots. The matching lower bound is a Bernstein--Walsh inequality on $X$
(Theorem~\ref{thm:surface-BW}), obtained from the domination principle once one checks that the singularity
of $\log|f|-Ng_E$ at $P_\infty$ is removable.

One remark on the dependence on $\zeta$, Replacing $\zeta$ by $\tilde\zeta=c\zeta+O(\zeta^2)$ changes (iii) by
$\log|c|$, hence shifts $G$ and $V_E$ by $\log|c|$, so that
\[
  \operatorname{cap}(E)\;\longmapsto\;|c|^{-1}\operatorname{cap}(E) .
\]
 The
normalisation of the admissible functions below rescales in step, by $|c|^{-N}$, and the two
cancel. It is the Widom factor, and neither $t_N(E,P_\infty)$ nor $\operatorname{cap}(E)$ on its
own, that is an invariant of the pointed surface and the set; Theorem~\ref{thm:main-root} is
independent of $\zeta$ for the same reason.
\subsection{Widom theory on a bordered surface}
 
The second half of the paper takes up the refined asymptotics, and there the hypotheses on $E$ have
to be strengthened. 
\begin{ass}\label{ass1}
From now on, we make the following assumptions.
\begin{itemize}
    \item Let
    \[
      E=\Omega_0:=\Omega_1\cup\cdots\cup\Omega_p
      \subset X\setminus\{P_\infty\}
    \]
    be a finite union of pairwise disjoint embedded closed discs
    with real-analytic boundary. Put
    \[
      Y
      :=X\setminus E^\circ
      =X\setminus\bigcup_{j=1}^{p}\Omega_j^\circ,
      \qquad
      \Gamma
      :=\partial Y
      =\partial E
      =\Gamma_1\cup\cdots\cup\Gamma_p .
    \]
    Thus \(Y\) is a connected compact bordered Riemann surface
    of genus \(g\) with \(p\) boundary curves, and
    \[
      Y^\circ=X\setminus E.
    \]

    \item The function \(\rho\) is a strictly positive real-analytic
    weight on \(\Gamma\).
\end{itemize}
\end{ass}
This is the surface analogue of Widom's system of curves, and by
\eqref{eq:gE-at-Pinfty} the equilibrium Green function $g_E$ is the Dirichlet Green function of
$Y^\circ$ with pole at $P_\infty$.

Cut $Y^\circ\setminus\{P_\infty\}$ to a simply connected surface, choose a harmonic conjugate
$\widetilde g_E$, and set
\[
  \Phi_E:=\exp\bigl(g_E+i\widetilde g_E\bigr).
\]
Then $|\Phi_E|=e^{g_E}$ is single-valued, while continuation of $\Phi_E$ around a closed curve
multiplies it by a unimodular constant. This produces a character $\gamma_E\in\mathcal X(Y)$ and a
flat unitary line bundle $L_E\to Y^\circ$, of which $\Phi_E$ is a section: holomorphic and
nowhere vanishing on $Y^\circ\setminus\{P_\infty\}$, with a simple pole at $P_\infty$ normalised by
\begin{equation}\label{eq:PhiE}
  \Phi_E(p)=\operatorname{cap}(E)^{-1}\zeta(p)^{-1}\bigl(1+o(1)\bigr),\qquad p\to P_\infty .
\end{equation}
This is Widom's character $\Gamma(\Phi)$ in line-bundle notation. In the same way, for $a\in Y^\circ$
we write \begin{equation}\label{eq:gfac}\Phi(\cdot,a):=\exp(g_E(\cdot,a)+i\widetilde g_E(\cdot,a)),\end{equation} a  holomorphic section of the flat line bundle
 with unitary character, boundary modulus one, and a simple pole at $a$, normalised by
$\Phi(P_\infty,a)>0$; for an effective divisor $D=a_1+\dots+a_m$ in $Y^\circ$ we put
$\Phi_D:=\prod_j\Phi(\cdot,a_j)$, so that $\operatorname{div}\Phi_D=-D$.

There are now two extremal problems, and the theorems below say that they have the same answer.
On the bundle side, fix a flat unitary line
bundle $L\to Y$ and a local unitary frame for $L$ near $P_\infty$, and put $\mathcal{H}^{\infty}(Y,L)$ as bounded holomorphic sections.
\[
  \|F\|_{\rho,\infty}:=\operatorname*{ess\,sup}_{q\in\Gamma}|F(q)|\,\rho(q),
  \]
the moduli being taken in the fixed metric of $L$. The
\emph{Widom extremal constant}  is
\begin{equation}\label{eq:widomconst}
  \mu(L,\rho):=\inf\bigl\{\|F\|_{\rho,\infty}\;:\;F\in \mathcal{H}^{\infty}(Y,L),\ F(P_\infty)=1\bigr\},
\end{equation}
written $\mu(\rho,\chi)$ when $L=L_\chi$, and $\mu(L)$ when $\rho\equiv1$. On the Chebyshev side,
put
\[
  M_{n,\rho}:=\inf\bigl\{\|T\rho\|_\Gamma\;:\;T\in H^0\bigl(X,\mathcal O(nP_\infty)\bigr)
  \ \text{admissible of order}\ n\bigr\},
\]
with extremal $T_{n,\rho}$. Using the maximum-modulus principle the unweighted problem without $\rho$ is exactly \eqref{eq:def1}. The two sides are joined at once. If
$T_N$ is admissible of order $N$, then \eqref{eq:PhiE} makes
$F_N:=\operatorname{cap}(E)^{-N}T_N\Phi_E^{-N}$ a bounded holomorphic section of $L_E^{-N}$ with
$F_N(P_\infty)=1$ and $\|F_N\|_\Gamma=\operatorname{cap}(E)^{-N}\|T_N\|_E$, whence
\begin{equation}\label{eq:comparison}
  \mu\bigl(L_E^{-N}\bigr)\;\le\;W_N(E,P_\infty).
\end{equation}
 Everything that follows can be read as the statement that
\eqref{eq:comparison} is asymptotically an equality, together with a description of what the
left-hand side is.

The next theorem builds a duality for the minimiser. It is essentially taken from Widom's \cite{Widom71a} with rank 1 and also follows from \cite{Read1958}.
Let \(\mathcal E\to Y\) be a holomorphic line bundle that extends to a
neighbourhood of \(Y\).  Using the nowhere-vanishing holomorphic section
\(\tau_{\mathcal E}\) (see \eqref{eq:nvlb}), write
\[
        s=f\,\tau_{\mathcal E}.
\]
Following the usual definition on a bordered Riemann surface, we put
\[
\mathcal H^1(Y,\mathcal E)
:=
\left\{
s=f\,\tau_{\mathcal E}:
|f|\text{ has a harmonic majorant on }Y^\circ
\right\}.
\]
This definition does not depend on the choice of
\(\tau_{\mathcal E}\), since the quotient of two such sections, as well
as its inverse, is bounded on \(Y\).   \(\mathcal{H}^1\) theory gives
nontangential boundary values in \(L^1\); see
\cite[Ch.~IV]{Heins1969} and
\cite[pp.~304--305 and Sec.~3]{Widom71a}.

For a flat unitary line bundle \(L\to Y\), we write $L^{-1}(P_\infty)=L^{-1}\otimes\mathcal O(P_\infty)$
\[
\mathcal H^1(L^{-1};P_\infty)
:=
\mathcal H^1\!\left(
Y,K_Y\otimes L^{-1}(P_\infty)
\right).
\]
Its elements are \(L^{-1}\)-valued holomorphic differentials on
\(Y^\circ\setminus\{P_\infty\}\), with at most a simple pole at
\(P_\infty\), satisfying the preceding integrable non-tangential boundary values.  For
\(\eta\) in this space, set
\[
        \|\eta\|_{\rho^{-1},1}
        :=
        \frac1{2\pi}
        \int_\Gamma |\eta|\,\rho^{-1}.
\]

\begin{theorem}
\label{thm:duality}
Let \(L\to Y\) be a flat unitary line bundle, and let
\(\rho\in C^\alpha(\Gamma)\) be strictly positive.  Then
\[
  \mu(L,\rho)
  =
  \sup_{\substack{
      \eta\in\mathcal H^1(L^{-1};P_\infty)\\
      \eta\ne0}}
  \frac{
      \left|\operatorname{Res}_{P_\infty}\eta\right|
  }{
      \|\eta\|_{\rho^{-1},1}
  }.
\]
\end{theorem}

\noindent Proved in Section \ref{ss:proof-duality}; we remark that $\left|\operatorname{Res}_{P_\infty}\eta\right|$ is well defined \cite{Widom71a}. The
extremal section is then unique and satisfies a product formula.

\begin{theorem}\label{thm:extremal}
Let $\chi\in\mathcal X(Y)$ and let $\rho>0$ be real analytic on $\Gamma$. Then the infimum
\eqref{eq:widomconst} is attained by a unique $F_{\rho,\chi}\in \mathcal{H}^{\infty}(Y,L_\chi)$ with
$F_{\rho,\chi}(P_\infty)=1$, the supremum in Theorem~\ref{thm:duality} is attained, and
\[
  \bigl|F_{\rho,\chi}\bigr|\,\rho=\mu(\rho,\chi)\qquad\text{a.e. on }\Gamma .
\]
The zeros of $F_{\rho,\chi}$ form an effective divisor $D_{\rho,\chi}$ in
$Y^\circ\setminus\{P_\infty\}$ of degree at most
\[
  \hat g:=2g+p-1=\operatorname{rank}H_1(Y,\mathbb Z),
\]
and, writing $R_\rho$ as a section of a flat unitary line bundle, with $|R_\rho|=\rho$ on $\Gamma$
and $R_\rho(P_\infty)>0$,
\begin{equation}\label{eq:product}
  F_{\rho,\chi}=\mu(\rho,\chi)\,R_\rho^{-1}\,\Phi_{D_{\rho,\chi}}^{-1}
 \end{equation}
so that, evaluating at $P_\infty$,
\begin{equation}\label{eq:mu-formula}
  \mu(\rho,\chi)=R_\rho(P_\infty)\prod_{z\in D_{\rho,\chi}}\Phi(P_\infty,z)
  =R_\rho(P_\infty)\,\exp\Bigl(\sum_{z\in D_{\rho,\chi}}g_E(z)\Bigr).
\end{equation}
The positions of the zeros are determined by two systems of $\hat g$ transcendental equations,
which replace Widom's equations in harmonic-measure coordinates.
\end{theorem}

\noindent Proved in \S\ref{ss:extremal}. Formula \eqref{eq:mu-formula} is the one to keep in mind: the Widom
constant is the outer factor at $P_\infty$ multiplied by the exponential of a Green sum over the
zeros of the extremal. Since $g_E>0$ on $Y^\circ$, it also shows that $\mu(\rho,\chi)$ is largest
when the divisor is as large as it can be, and this is what the next theorem makes precise.

The counting is governed by the Green differential $dG:=2\partial g_E(\cdot,P_\infty)$. It has a
simple pole at $P_\infty$, no zeros on $\Gamma$ by Hopf's lemma, and, by Poincar\'e--Hopf together
with $\chi(Y)=2-2g-p$, exactly $\hat g=2g+p-1$ zeros
\[
  B:=(dG)_0=b_1+\dots+b_{\hat g}
\]
in $Y^\circ\setminus\{P_\infty\}$ (~\eqref{eq:green-differential-zero-count}). These are the critical points of the
Green function $g_E$, so that $\sum_\ell g_E(b_\ell)$ is precisely the Parreau--Widom sum of $E$;
it now has $2g+p-1$ terms rather than $p-1$. 
\begin{remark}We remark at this point that $\hat g$ is the genus of the compact Riemann surface $\hat Y$ which is the Schottky double of $Y$. This also
gives another interpretation of the zero count above. Indeed, if
\(\sigma:\widehat Y\to\widehat Y\) denotes the antiholomorphic
involution, then the Dirichlet Green
function admits the  harmonic continuation
\[
 \widehat g_E(p)=
 \begin{cases}
   g_E(p),&p\in Y,\\
   -g_E\bigl(\sigma(p)\bigr),&p\in\sigma(Y).
 \end{cases}
\]
Consequently, \(dG=2\partial\widehat g_E\) is an Abelian differential
of the third kind on \(\widehat Y\), with simple poles of residues
\(-1\) and \(1\) at \(P_\infty\) and \(\sigma(P_\infty)\), respectively.
It therefore has \(2\widehat g\) zeros on \(\widehat Y\). Since it has
no zeros on \(\Gamma\), reflection places exactly \(\widehat g\) of
them in \(Y^\circ\). See
\cite[Section~2.1]{GustafssonSebbar}.
\end{remark}

\begin{theorem}[Range of the Widom constants]\label{thm:range}
Let $\rho>0$ be real analytic on $\Gamma$. As $\chi$ ranges over the character torus
$\mathcal X(Y)\cong\mathbb T^{\hat g}$, the constants $\mu(\rho,\chi)$ fill the closed interval
\begin{equation}\label{eq:range}
  \Bigl[\;R_\rho(P_\infty),\;\;R_\rho(P_\infty)\prod_{\ell=1}^{\hat g}\Phi(P_\infty,b_\ell)\;\Bigr]
  \;=\;\Bigl[\;R_\rho(P_\infty),\;\;R_\rho(P_\infty)\exp\Bigl(\sum_{\ell=1}^{\hat g}g_E(b_\ell)\Bigr)\;\Bigr].
\end{equation}
The minimum is attained at $\chi=-\gamma_{R_\rho}$, the maximum at
$\chi=-\gamma_{R_\rho}-\gamma_B$, where $\gamma_{R_\rho}=\operatorname{char}R_\rho$ and
$\gamma_B=\sum_\ell\operatorname{char}\Phi(\cdot,b_\ell)$.
\end{theorem}

\noindent Proved in \S\ref{ss:range}. The extremes are the two extremes of \eqref{eq:mu-formula}: the empty
divisor, and the full zero divisor of $dG$.

The main asymptotic theorem is the analogue of Widom's Theorem 8.3. However one key ingredient is missing, which is the Cauchy kernel. Its replacement is $C_\mathcal{D}(q,p)$, \cite[\S2, 2.1]{Bertola2021},\cite{Fay} which Bertola uses to define the Weyl-Stieltjes function. It is a meromorphic differential in $q$ variable and a meromorphic function in $p$ variable, except it has some auxiliary poles and zeros assigned by a divisor $\mathcal{D}$. Let \(Y^{-}\supset Y\) be a slightly enlarged bordered surface obtained by adjoining to \(Y\) a sufficiently thin analytic collar across \(\Gamma\). Choose distinct points \(d_1,\ldots,d_g\in X\setminus\overline{Y_-}\)
in general, so that the divisor
\(\mathcal D:=d_1+\cdots+d_g\) is nonspecial. Write
\begin{equation}\label{eq:CD1}
        C_{\mathcal D}(q,p)
        :=C^{(1,0)}(q,p;\mathcal D,P_\infty).
\end{equation}
This is a meromorphic differential in \(q\), a meromorphic function in
\(p\), and its defining properties are
\begin{equation}\label{eq:CD2}
\begin{aligned}
        \operatorname{div}_q C_{\mathcal D}(q,p)
        &\ge \mathcal D-p-P_\infty,\\
        \operatorname{div}_p C_{\mathcal D}(q,p)
        &\ge-\mathcal D-q+P_\infty,
\end{aligned}
\qquad
        \operatorname{Res}_{q=p}C_{\mathcal D}(q,p)=1.
\end{equation}
The residue theorem in the \(q\)-variable then gives
\(\operatorname{Res}_{q=P_\infty}C_{\mathcal D}(q,p)=-1\). An explicit formula is also available in terms of Riemann theta functions \cite[\S3]{Bertola2021}, however we omit it as it will not be needed. We will use this kernel to construct an admissible function of order $n$ that locally uniformly approximates the minimiser. However in this construction we run into the poles at the divisor $\mathcal{D}$. These additional poles  must be removed from the resulting projections, and this is not immediate and requires extra care.

\begin{theorem}[Szeg\H{o}--Widom asymptotics]\label{thm:sw}
Assume that  the Assumptions \ref{ass1} holds. Put
\[
  \chi_n:=-n\gamma_E\in\mathcal X(Y),
  \qquad
  \mu_n:=\mu\bigl(L_E^{-n},\rho\bigr)=\mu(\rho,\chi_n),
\]
and let $F_n:=F_{\rho,\chi_n}\in \mathcal{H}^{\infty}(Y,L_E^{-n})$, $F_n(P_\infty)=1$, be the Widom extremal in
that class. Then
\begin{equation}\label{eq:sw-norm}
  M_{n,\rho}=\operatorname{cap}(E)^{n}\bigl(\mu_n+o(1)\bigr),\qquad n\to\infty ,
\end{equation}
and
\begin{equation}\label{eq:sw-fn}
  \frac{T_{n,\rho}}{\operatorname{cap}(E)^{n}\,\Phi_E^{\,n}}-F_n\;\longrightarrow\;0
\end{equation}
locally uniformly on $Y^\circ$.
If moreover the divisors $D_{\rho,\chi_n}$ stay away from $\Gamma$, then \eqref{eq:sw-fn} holds
uniformly on $Y$.
\end{theorem}

\noindent Proved in \S\ref{sec:proof-sw-part-ii}. The zero-separation hypothesis in the last clause is the one Widom
imposes in the final part of his proof of Theorem 8.3, and it is used here for the same purpose.

Taking $\rho\equiv1$ gives $R_\rho\equiv1$, $R_\rho(P_\infty)=1$ and $M_{n,1}=t_n(E,P_\infty)$, so
that \eqref{eq:sw-norm} reads $W_n(E,P_\infty)-\mu(L_E^{-n})\to0$. Combining this with
Theorem~\ref{thm:range} and with the elementary bound \eqref{eq:comparison} yields a
Totik--Widom bound on a compact Riemann surface.
\begin{remark}
The real-analyticity assumptions in Assumption \ref{ass1} are primarily
technical and quantitative. The analyticity of \(\Gamma\) allows the
Green function \(g_E\) to be reflected harmonically across the boundary,
and the Green factor \(\Phi_E\) to be continued holomorphically to a
fixed two-sided collar of \(\Gamma\); see, for example,
\cite[Section~7.5.2]{KrantzGuide}. The same applies to the outer factors
associated with a real-analytic weight \(\rho\). If the reflected
continuation is still denoted by \(g_E\), one may displace the boundary
to a fixed level curve
\[
        \Gamma_-=\{g_E=-\varepsilon\}
\]
on the reflected side. There
\[
        |\Phi_E|=e^{-\varepsilon}=:r<1,
\]
and the resulting approximation errors are \(O(r^n)\) after
normalization by the natural scale \(\operatorname{cap}(E)^n\).
The availability of a fixed collar and this geometric decay substantially
simplify the proof of the Szeg\H{o}--Widom asymptotics.

We do not expect real analyticity to be essential for the qualitative
conclusions. In the planar setting, Widom proved the corresponding
Szeg\H{o}--Widom asymptotics under the weaker \(C^{2+}\) boundary
hypothesis, using more delicate boundary approximation in place of
fixed-collar analytic continuation; see
\cite[Section~8, especially Lemma~8.2 and Theorem~8.3,
pp.~180--190]{Widom69}, and
\cite[Section~3]{CSZreview} for a modern account. This suggests that
the present results should persist under weaker boundary regularity,
although the geometric \(O(r^n)\) estimates would then have to be
replaced by qualitative \(o(1)\) estimates. We do not pursue this
extension here.
\end{remark}
\begin{corollary}\label{cor:TW}
Let $E$ be a finite union of pairwise disjoint closed discs with real analytic boundary in
$X\setminus\{P_\infty\}$, and set
\[
  \mathrm{PW}(E):=\sum_{\ell=1}^{\hat g}g_E(b_\ell)
  =\sum_{\nabla g_E(b)=0}g_E(b),\qquad \hat g=2g+p-1 .
\]
Then $W_N(E,P_\infty)\ge1$ for every admissible $N$, and
\[
  \limsup_{N\to\infty}W_N(E,P_\infty)\;\le\;\exp\bigl[\mathrm{PW}(E)\bigr] .
\]
In particular the Widom factors of such a set are bounded.
\end{corollary}

\noindent For $g=0$ this is Widom's bound for a system of $p$ curves, with $p-1$ critical points. On
a surface of genus $g$ the same sum runs over $2g+p-1$ critical points, and the extra $2g$ of them
are contributed by the topology of $X$.

Finally, the sequence $n\mapsto\mu_n$ need not converge, and its behaviour is read off the orbit of
$\gamma_E$ on the character torus.

\begin{corollary}\label{cor:limitpoints}
Choose a basis $c_1,\dots,c_{\hat g}$ of $H_1(Y,\mathbb Z)$ and write
$\gamma_E(c_j)=e^{2\pi i\theta_j}$. Under the hypotheses of Theorem~\ref{thm:sw}, the set of limit
points of $M_{n,\rho}/\operatorname{cap}(E)^n$ is a finite union of pairwise disjoint closed
subintervals of the interval \eqref{eq:range}. It is the whole of \eqref{eq:range} when
$1,\theta_1,\dots,\theta_{\hat g}$ are linearly independent over $\mathbb Q$, and it is a set of at
most $q$ points when every $\theta_j$ lies in $\tfrac1q\mathbb Z$.
\end{corollary}

This is where the surface departs from the plane, and the departure is not cosmetic. In the plane
$\hat g=p-1$, and for a single Jordan region $p=1$ the character torus is trivial: there is only one
character $\mu$, and one is back at Faber's theorem \eqref{eq:faber}, where
$W_n\to1$ and the extremal polynomials converge after normalisation. On a surface of genus $g\ge1$ a
single disc already gives $\hat g=2g$, so that even the simplest possible $E$ carries a
$2g$-dimensional torus of characters and a nondegenerate interval \eqref{eq:range}. There is no
analogue of Faber's theorem, and the Widom factors of a single analytic disc can oscillate almost
periodically.  We conclude the paper in \S\ref{sec:genus-one-examples} with some worked out examples in genus $1$, with the aid of an explicit basis of admissible functions given by Weierstrass $\wp$-function.
\begin{remark}
We make no claim of novelty for the underlying \(\mathcal{H}^{\infty}\) extremal
theory. Theorem \ref{thm:duality} is essentially the rank-one case of Widom's
bundle-valued duality theorem
\cite[Section~3]{Widom71a}, while the principal ingredients of
Theorem \ref{thm:extremal} and the mechanism behind Theorem \ref{thm:range} are already present
in Widom's product formula and zero equations
\cite[Section~5]{Widom69}. Thus novelty of the present work does not lie in the abstract
$H^\infty$ duality theory by itself, but in its connection with the
scalar meromorphic Chebyshev problem on a compact Riemann surface.

The proof of Theorem~\ref{thm:sw} follows the strategy of
\cite[Section~8]{Widom69}, but its extension to positive genus is not
immediate. In particular, the  Cauchy kernel $C_\mathcal{D}$ carries an auxiliary
divisor whose unwanted poles must be removed, and the deletion of zeros
approaching the boundary must be accompanied by a correction of the
character.

We restrict here to contractible closed boundary curves. Arcs require
a two-sided analysis and additional control at their endpoints; this is
already substantially more delicate in the plane
\cite[Section~11]{Widom69}, as shown by
\cite{TotikYuditskii,Eichinger,Alpan}. Homologically nontrivial curves
present a different obstruction: after cutting the surface, the two
boundary traces must satisfy global matching and period conditions in
order to descend to a meromorphic function on \(X\). Finally, one may consider the corresponding $L^2$ extremal problem
with respect to a prescribed measure $d\nu$ on $\Gamma$, replacing
the $L^\infty$ norm. We leave these
extensions for future work.
\end{remark}

\section{Characters, Green factors, and correction sections}
\label{sec:characters}

Let
\[
        \cX(X):=\operatorname{Hom}(H_1(X,\Z),\T),
        \qquad
        \T:=\{w\in\C:|w|=1\}.
\]
For \(\chi\in\cX(X)\), let \(\cL_\chi\) denote the associated flat unitary line
bundle.  A section of \(\cL_\chi\) may be viewed on the universal cover as a
meromorphic function \(f\) satisfying
\[
        f(\gamma p)=\chi(\gamma)f(p)
\]
for deck transformations \(\gamma\).  Since the character is unitary, \(|f|\)
is single-valued.

Choose a simply connected coordinate disc \(U_\infty\) about \(P_\infty\), and
choose one lift of \(U_\infty\) to the universal cover of \(X\).  This gives
compatible unitary flat frames for all \(\cL_\chi\) over \(U_\infty\).  All
Laurent expansions below use these frames, and all absolute values use the flat
unitary metrics.  The choices are compatible with tensor products and dual
bundles.
In particular,
\[
        \cL_\chi\otimes\cL_\psi\simeq\cL_{\chi\psi},
        \qquad
        \cL_\chi^\vee\simeq\cL_{\chi^{-1}},
\]
with the corresponding identifications of the chosen frames.

On the sphere, the factor \(z-q\) has divisor \(q-\infty\), and its absolute
value is determined by the logarithmic kernel.  The next lemma gives the
corresponding factor on \(X\).  In general it is a section of a flat bundle,
not a single-valued function; compare \cite[Thm.~III.9.10]{FarkasKra}.

\begin{lemma}
\label{lem:compact-green-factor}
For each \(q\in X\setminus\{P_\infty\}\), there are a character
\(\chi_q\in\cX(X)\) and a meromorphic section \(B_q\) of \(\cL_{\chi_q}\) such
that
\[
        \operatorname{div}(B_q)=q-P_\infty,
\]
\[
        B_q(p)=\zeta(p)^{-1}+O(1),\qquad p\to P_\infty,
\]
and
\[
        |B_q(p)|=e^{-G(p,q)},\qquad p\in X\setminus\{q,P_\infty\}.
\]
\end{lemma}

\begin{proof}
Put \(X_q^\circ=X\setminus\{q,P_\infty\}\), and lift
\(G(\,\cdot\,,q)\) to the universal cover of \(X_q^\circ\).  Choose a
harmonic conjugate \(G^*(\,\cdot\,,q)\) so that \(G+iG^*\) is holomorphic,
and set
\[
        \widetilde B_q(p):=\exp\{-G(p,q)-iG^*(p,q)\}.
\]
This is holomorphic and non-vanishing on the covering surface, and
\[
        |\widetilde B_q(p)|=e^{-G(p,q)}.
\]

The local expansions also determine its behaviour at the punctures.  If \(z\)
is a coordinate centred at \(q\), then, after choosing a branch of \(\log z\),
\[
        G(p,q)+iG^*(p,q)=-\log z(p)+h_q(p),
\]
where \(h_q\) is holomorphic near \(q\).  Hence
\[
        \widetilde B_q(p)=z(p)e^{-h_q(p)},
\]
so \(\widetilde B_q\) has a simple zero at \(q\).  Similarly, near
\(P_\infty\),
\[
        G(p,q)+iG^*(p,q)=\log\zeta(p)+h_\infty(p),
\]
with \(h_\infty\) holomorphic and \(\operatorname{Re}h_\infty(P_\infty)=0\).
Thus
\[
        \widetilde B_q(p)=\zeta(p)^{-1}e^{-h_\infty(p)}.
\]
Its leading coefficient has modulus one.  Multiplying by a constant in
\(\T\), we make this coefficient equal to \(1\).

For a deck transformation \(\gamma\), the quotient
\[
        \frac{\widetilde B_q(\gamma p)}{\widetilde B_q(p)}
\]
is holomorphic and has modulus one, so it is a constant \(\chi_q(\gamma)\in
\T\).  These constants form a character of \(\pi_1(X_q^\circ)\).  The two
local formulas above show that small loops around \(q\) and \(P_\infty\) have
multiplier one.  The character therefore factors through \(\pi_1(X)\), and,
since \(\T\) is abelian, through \(H_1(X,\Z)\).  The function
\(\widetilde B_q\) consequently descends to a meromorphic section of
\(\cL_{\chi_q}\).  Its only zero and pole are the ones just found, which proves
all three assertions.
\end{proof}
We remark that, under the Abel map, the divisor $q - P_\infty$ (modulo linear equivalence) has degree zero and hence defines a unique point of $\mathrm{Pic}^0(X)$.
For an effective divisor \(D=q_1+\cdots+q_n\) in
\(X\setminus\{P_\infty\}\), define
\[
        B_D:=\prod_{j=1}^n B_{q_j},\qquad
        \chi_D:=\prod_{j=1}^n\chi_{q_j}.
\]
Then \(B_D\) is a section of \(\cL_{\chi_D}\),
\[
        \operatorname{div}(B_D)=D-nP_\infty,
\]
\[
        B_D(p)=\zeta(p)^{-n}+O(\zeta(p)^{-n+1}),
        \qquad p\to P_\infty,
\]
and
\[
        |B_D(p)|=\exp\left(-\sum_{j=1}^n G(p,q_j)\right).
\]

Here a section of \(\cL_\chi\otimes\cO(mP_\infty)\) is viewed as a
meromorphic section of \(\cL_\chi\) with no pole away from \(P_\infty\) and
with pole order at most \(m\) there.

\begin{lemma}
\label{lem:uniform-correction}
Fix an integer \(m\ge2g\).  There is a constant \(C_E>0\) such that for every
\(\chi\in\cX(X)\) there exists
\[
        s_\chi\in H^0\bigl(X,\cL_\chi\otimes\cO(mP_\infty)\bigr)
\]
with
\[
        s_\chi(p)=\zeta(p)^{-m}+O(\zeta(p)^{-m+1}),
        \qquad p\to P_\infty,
\]
and
\[
        \sup_E |s_\chi|\le C_E .
\]
\end{lemma}

The fixed data \(X\), \(P_\infty\), \(\zeta\), the local frames, and \(m\) are
suppressed from the notation \(C_E\); the point is that the constant is
independent of \(\chi\).

\begin{proof}
If \(g=m=0\), the character group is trivial, and the constant section \(1\)
proves the result. Assume henceforth that \((g,m)\ne(0,0)\).

We first fix a complex character \(\chi\). For \(k=m-1,m\), Riemann-Roch
gives
\[
 H^1\bigl(X,\cL_\chi\otimes\cO(kP_\infty)\bigr)
 \simeq H^0\bigl(X,K_X\otimes\cL_\chi^{-1}
                       \otimes\cO(-kP_\infty)\bigr)^*=0,
\]
because the bundle in the last term has degree \(2g-2-k<0\). Riemann--Roch
therefore yields
\[
 \dim H^0\bigl(X,\cL_\chi\otimes\cO(mP_\infty)\bigr)=m-g+1,
 \qquad
 \dim H^0\bigl(X,\cL_\chi\otimes\cO((m-1)P_\infty)\bigr)=m-g.
\]
Let \(\ell\) be the coefficient of \(\zeta^{-m}\) at \(P_\infty\), computed
in the fixed local frame. Its kernel on
\(H^0(X,\cL_\chi\otimes\cO(mP_\infty))\) is precisely
\(H^0(X,\cL_\chi\otimes\cO((m-1)P_\infty))\), and hence has codimension
one. Thus \(\ell\) is onto, so there is a section
\(s\in H^0(X,\cL_\chi\otimes\cO(mP_\infty))\) with \(\ell(s)=1\).

It remains to make this choice locally continuous in \(\chi\). Choose a
basis of \(H_1(X,\Z)\), identify the space of complex characters with
\(T:=(\C^*)^{2g}\), and write \(\chi_\lambda\) for the character
corresponding to \(\lambda\in T\). The unitary characters correspond to
the compact subtorus \(\mathbb T^{2g}:=(S^1)^{2g}\).

 The bundles
\(\cL_{\chi_\lambda}\otimes\cO(mP_\infty)\)  form a holomorphic
family over \(T\). Since the projection \(X\times T\to T\) is proper and
the dimension of
\(H^0(X,\cL_{\chi_\lambda}\otimes\cO(mP_\infty))\) is constant,
\(r:=m-g+1\),  \cite[Theorem, p.~211]{GR84} shows that, near every
\(\lambda^0\in T\), one can choose sections
\(S_1(\lambda),\ldots,S_r(\lambda)\), depending holomorphically on
\(\lambda\), which form a basis of this space for every nearby \(\lambda\).

Fix \(\lambda^0\in\mathbb T^{2g}\), and choose
\(s^0\in H^0(X,\cL_{\chi_{\lambda^0}}\otimes\cO(mP_\infty))\) with
\(\ell(s^0)=1\). On a neighbourhood \(U\) of \(\lambda^0\), write
\[
 s^0=\sum_{j=1}^r c_jS_j(\lambda^0),
 \qquad
 G_\lambda:=\sum_{j=1}^r c_jS_j(\lambda).
\]
The map 
\(\lambda\mapsto\ell(G_\lambda)\) is holomorphic, and it takes the value
\(1\) at \(\lambda^0\). After shrinking \(U\), we may assume that it does
not vanish there. Hence
\[
 s_\lambda:=\frac{G_\lambda}{\ell(G_\lambda)},
 \qquad \lambda\in U,
\]
depends holomorphically on \(\lambda\), belongs to
\(H^0(X,\cL_{\chi_\lambda}\otimes\cO(mP_\infty))\), and satisfies
\(\ell(s_\lambda)=1\).

For \(\lambda\in\mathbb T^{2g}\), the character \(\chi_\lambda\) is
unitary, so \(|s_\lambda|\) is single-valued. Shrinking the unitary part of
\(U\) once more, we obtain a neighbourhood \(U_{\lambda^0}\) of
\(\lambda^0\) in \(\mathbb T^{2g}\), with closure contained in \(U\), on
which \((p,\lambda)\mapsto|s_\lambda(p)|\) is continuous. Compactness of
\(E\times\overline{U_{\lambda^0}}\) gives a finite bound for this function.

The sets \(U_{\lambda^0}\) cover \(\mathbb T^{2g}\). Choose a finite
subcover \(U_1,\ldots,U_N\), and denote the corresponding normalized local
families by \(s^{(1)}_\lambda,\ldots,s^{(N)}_\lambda\). Then
\[
 C_E:=\max\left\{1,
   \max_{1\le j\le N}\ \max_{(p,\lambda)\in
       E\times\overline{U_j}}|s^{(j)}_\lambda(p)|\right\}<\infty.
\]
Given \(\chi=\chi_\lambda\in\cX(X)\), choose \(j\) with
\(\lambda\in U_j\) and set \(s_\chi:=s^{(j)}_\lambda\). Since
\(\ell(s_\chi)=1\),
\[
 s_\chi(p)=\zeta(p)^{-m}+O\bigl(\zeta(p)^{-m+1}\bigr)
 \quad (p\to P_\infty),
 \qquad
 \sup_E|s_\chi|\le C_E.
\]
The constant \(C_E\) is independent of \(\chi\), as required.
\end{proof}
\section{Proof of Theorem \ref{thm:main-root}}
\label{sec:chebyshev}

Let \(L(NP_\infty)\) denote the space of meromorphic functions on \(X\) with no
poles except possibly at \(P_\infty\), and with pole order at most \(N\).
Admissible functions of order \(N\) were defined in the introduction.  For
\(N\ge2g\), Serre duality and Riemann--Roch give
\[
        \dim L(NP_\infty)=N-g+1,
        \qquad
        \dim L((N-1)P_\infty)=N-g.
\]
Thus the leading coefficient at \(P_\infty\) can be prescribed, and admissible
functions exist for every \(N\ge2g\).

\begin{theorem}[Bernstein--Walsh on \(X\)]
\label{thm:surface-BW}
If \(f\in L(NP_\infty)\) is admissible of order \(N\), then
\[
        |f(p)|\le \|f\|_E e^{N g_E(p)},
        \qquad p\in X\setminus(E\cup\{P_\infty\}).
\]
In particular
\[
        t_N(E,P_\infty)\ge \capc(E)^N .
\]
\end{theorem}

\begin{proof}
The function
\[
        u(p):=\log|f(p)|-Ng_E(p)
\]
is subharmonic on \(X\setminus(E\cup\{P_\infty\})\).  The two expansions at
\(P_\infty\) give
\[
\begin{split}
        u(p)
        &=\bigl(-N\log|\zeta(p)|+o(1)\bigr)\\
        &\quad
        -N\bigl(-\log|\zeta(p)|-\log\capc(E)+o(1)\bigr)\\
        &=N\log\capc(E)+o(1).
\end{split}
\]
Thus \(u\) extends across \(P_\infty\).  On \(X\setminus E\),
Frotman conditions gives \(g_E\ge0\), and hence
\[
        u(p)\le\log|f(p)|.
\]

The maximum principle, applied on each component of \(X\setminus E\), now
gives
\[
        u\le \log\|f\|_E
\]
on \(X\setminus E\).  This is the asserted inequality away from
\(P_\infty\). Letting \(p\to P_\infty\) gives
\[
        \capc(E)^N\le \|f\|_E.
\]
Taking the infimum over admissible \(f\) proves the lower bound.
\end{proof}
 \begin{theorem}
\label{thm:fekete-diameter}
For \(n\ge2\), define
\[
        \mathcal E_n(E):=
        \min_{q_1,\ldots,q_n\in E}
        \sum_{1\le i<j\le n}G(q_i,q_j),
\]
and
\[
        \delta_n(E):=
        \exp\left(-\frac{2\mathcal E_n(E)}{n(n-1)}\right).
\]
Any minimising $n$ points $\{q_1,\dots, q_n\}$ is called the $n$-th Fekete set. Then \(\delta_n(E)\) is non-increasing and
\[
        \lim_{n\to\infty}\delta_n(E)=\capc(E).
\]
\end{theorem}

The standard Fekete argument applies verbatim. Indeed, its proof uses
only that the kernel is symmetric, lower semicontinuous, bounded from
below on $E\times E$, and equal to $+\infty$ on the diagonal; all
these properties hold for the bipolar Green kernel restricted to
$E\times E$. We therefore omit the proof and refer to
\cite{SaffTotik}.

\begin{proof}[Proof of Theorem~\ref{thm:main-root}]
Theorem~\ref{thm:surface-BW} gives
\[
        \liminf_{N\to\infty}t_N(E,P_\infty)^{1/N}\ge\capc(E).
\]
We prove the reverse inequality.  Fix \(m\ge2g\).  For each \(n\ge2\), choose
an \(n\)-point Fekete divisor
\[
        D_n=q_1^{(n)}+\cdots+q_n^{(n)}
\]
for the kernel \(G\) on \(E\), and form the Green product \(B_{D_n}\).  Let
\(\chi_{D_n}\) be its character.  By Lemma~\ref{lem:uniform-correction}, choose
a section
\[
        s_n\in H^0\bigl(X,\cL_{\chi_{D_n}^{-1}}
              \otimes\cO(mP_\infty)\bigr)
\]
such that
\[
        s_n(p)=\zeta(p)^{-m}+O(\zeta(p)^{-m+1}),
        \qquad
        \|s_n\|_E\le C_E.
\]
Then
\[
        F_n:=s_n B_{D_n}
\]
has trivial character, and hence is a single-valued meromorphic function on
\(X\).  It has no pole away from \(P_\infty\), and
\[
        F_n(p)=\zeta(p)^{-(n+m)}
               +O(\zeta(p)^{-(n+m)+1}),
        \qquad p\to P_\infty.
\]
Thus \(F_n\) is admissible of order \(n+m\).

First let \(p\in E\setminus\operatorname{supp}D_n\).  Compare the
\((n+1)\)-point configuration
\(\{p,q_1^{(n)},\ldots,q_n^{(n)}\}\) with an \((n+1)\)-point Fekete
configuration.  Since \(D_n\) itself is Fekete, we obtain
\[
\begin{split}
        |B_{D_n}(p)|\,\delta_n(E)^{n(n-1)/2}
        &=
        \exp\left(
          -\sum_{j=1}^n G(p,q_j^{(n)})
          -\sum_{1\le i<j\le n}G(q_i^{(n)},q_j^{(n)})
        \right)\\
        &\le \delta_{n+1}(E)^{n(n+1)/2}.
\end{split}
\]
The same inequality is immediate if \(p\) is one of the points of \(D_n\),
because then \(B_{D_n}(p)=0\).  Therefore
\[
        \|F_n\|_E
        \le
        C_E
        \frac{\delta_{n+1}(E)^{n(n+1)/2}}
             {\delta_n(E)^{n(n-1)/2}} .
\]
This is the point at which monotonicity of the Fekete diameters is needed.  By
Theorem~\ref{thm:fekete-diameter}, \(\delta_{n+1}(E)\le\delta_n(E)\), and so
\[
\begin{split}
        \frac{\delta_{n+1}(E)^{n(n+1)/2}}
             {\delta_n(E)^{n(n-1)/2}}
        &=
        \delta_{n+1}(E)^n
        \left(\frac{\delta_{n+1}(E)}{\delta_n(E)}\right)^{n(n-1)/2}\\
        &\le \delta_{n+1}(E)^n.
\end{split}
\]
Since \(F_n\) is admissible of order \(n+m\),
\[
        t_{n+m}(E,P_\infty)
        \le \|F_n\|_E
        \le C_E\delta_{n+1}(E)^n.
\]
Consequently,
\[
        t_{n+m}(E,P_\infty)^{1/(n+m)}
        \le C_E^{1/(n+m)}
             \delta_{n+1}(E)^{n/(n+m)}.
\]
Now \(C_E^{1/(n+m)}\to1\), \(n/(n+m)\to1\), and
\(\delta_{n+1}(E)\to\capc(E)\).  Hence
\[
        \limsup_{n\to\infty}
        t_{n+m}(E,P_\infty)^{1/(n+m)}
        \le \capc(E).
\]
As \(N=n+m\) runs through every integer \(N\ge m+2\), this is the full upper
bound
\[
        \limsup_{N\to\infty}t_N(E,P_\infty)^{1/N}\le\capc(E).
\]
Together with the lower bound, this proves the theorem.
\end{proof}

\begin{remark}
When \(X=\widehat\C\), \(P_\infty=\infty\), and \(\zeta=1/z\), \(B_q(z)=z-q\), and one may take \(m=0\) and \(s\equiv1\).
The construction is therefore exactly the classical Fekete-polynomial proof.
\end{remark}
\section{Preliminary constructions}
\subsection{Characters and line bundles on \texorpdfstring{$Y$}{Y}}
For all the following sections we will assume Assumption \ref{ass1}. That is, $Y$ is a connected compact bordered Riemann surface of genus $g$ with $p$ boundary curves which are analytic.

A unitary character
of \(Y\) is a homomorphism
\[
        \chi:\pi_1(Y,P_\infty)\longrightarrow\mathbb T,
        \qquad
        \mathbb T:=\{\zeta\in\mathbb C:|\zeta|=1\}.
\]
We denote the group of unitary characters by
\[
        \mathfrak X(Y)
        :=
        \operatorname{Hom}\bigl(\pi_1(Y,P_\infty),\mathbb T\bigr).
\]
Since every character factors through the abelianisation of the fundamental
group and
\[
        H_1(Y;\mathbb Z)\simeq\mathbb Z^{\hat{g}},
        \qquad \hat{g}:=2g+p-1,
\]
the character group is naturally identified with the torus
\[
        \mathfrak X(Y)
        \simeq
        \operatorname{Hom}\bigl(H_1(Y;\mathbb Z),\mathbb T\bigr)
        \simeq\mathbb T^{\hat{g}}.
\]
Let \(\pi:\widetilde Y\to Y^\circ\) be the universal covering map, and
identify its group of deck transformation with \(\pi_1(Y,P_\infty)\).  With a character
\(\chi\in\mathfrak X(Y)\) we associate the flat holomorphic line bundle
\[
        L_\chi
        :=
        (\widetilde Y\times\mathbb C)/{\sim},
        \qquad
        (\widetilde q,v)
        \sim
        \bigl(\gamma\cdot\widetilde q,\chi(\gamma)v\bigr).
\]
With this convention, a holomorphic section of \(L_\chi\) is represented
on the universal cover by a holomorphic function
\(\psi:\widetilde Y\to\mathbb C\) satisfying
\[
        \psi(\gamma\cdot\widetilde q)
        =
        \chi(\gamma)\psi(\widetilde q).
\]
Thus multiplicative holomorphic functions of character \(\chi\) and
holomorphic sections of \(L_\chi\) are two descriptions of the same
object.

The existence of a nowhere-vanishing section is a standard consequence
of the theory of line bundles on open Riemann surfaces.  Indeed, choose a
connected open neighbourhood \(Y'\) of \(Y\) such that
\(Y\hookrightarrow Y'\) is a deformation retract.  The character \(\chi\),
and hence the bundle \(L_\chi\), extends naturally to \(Y'\). \cite[Lemma~30.2]{Forster}
 shows that a nontrivial meromorphic section of a holomorphic
bundle on a non-compact Riemann surface may be multiplied by a meromorphic
function so that all of its zeros and poles are removed.  The resulting
section is holomorphic and nowhere vanishing, this yields the
holomorphic triviality of every line bundle on a non-compact Riemann
surface \cite[Lemma~30.2 and Theorem~30.3]{Forster}.  Applied to
\(L_\chi\), it gives a nowhere-vanishing holomorphic section, or
equivalently a function
\[
        \psi_\chi\in\mathcal O(\widetilde Y)^\ast,
        \qquad
        \psi_\chi(\gamma\cdot\widetilde q)
        =
        \chi(\gamma)\psi_\chi(\widetilde q).
\]
After fixing a lift \(\widetilde P_\infty\) of \(P_\infty\), the section
may be normalized by requiring
\[
        \psi_\chi(\widetilde P_\infty)=1.
\]
Equivalently, the existence of \(\psi_\chi\) is the rank-one case of
Forster's theorem on prescribed multiplicative factors of automorphy
\cite[Theorem~31.2]{Forster}.

For later purposes we need slightly more than pointwise existence: the
section must vary continuously with the character. We do this via
the Behnke--Stein construction by prescribed periods.  Choose loops
\(c_1,\ldots,c_{\hat{g}}\) whose homology classes form a basis of
\(H_1(Y;\mathbb Z)\).  By the Behnke--Stein period theorem
\cite[Theorem~28.6]{Forster}, there are holomorphic differentials
\(\eta_1,\ldots,\eta_{\hat{g}}\), defined on the fixed neighbourhood \(Y'\), such
that
\[
        \int_{c_k}\eta_j=\delta_{jk}.
\]
Fix \(\chi_0\in\mathfrak X(Y)\).  In a sufficiently small neighbourhood
\(U\) of \(\chi_0\), one can choose the arguments of the character
continuously; that is, there are continuous real-valued functions
\(\theta_j\) on \(U\) satisfying
\[
        \chi(c_j)=e^{i\theta_j(\chi)},
        \qquad \chi\in U.
\]
For \(\chi\in U\), set
\[
        \omega_\chi
        :=
        i\sum_{j=1}^N\theta_j(\chi)\eta_j
\]
and define
\begin{equation}\label{eq:nvlb}
    \psi_\chi(\widetilde q)
        :=
        \exp\left(
            \int_{\widetilde P_\infty}^{\widetilde q}
            \pi^\ast\omega_\chi
        \right).
\end{equation}
The periods of \(\omega_\chi\) give
\[
        \psi_\chi(\gamma\cdot\widetilde q)
        =
        \chi(\gamma)\psi_\chi(\widetilde q),
\]
so that \(\psi_\chi\) represents a nowhere-vanishing holomorphic section
of \(L_\chi\), normalized at \(P_\infty\).

This formula also makes the dependence on \(\chi\)  vary continuously.  It follows that
\(\chi\mapsto\psi_\chi\) is continuous, in fact real analytic.

Unitarity gives a corresponding statement downstairs.  Indeed,
\[
        \lvert\psi_\chi(\gamma\cdot\widetilde q)\rvert
        =
        \lvert\psi_\chi(\widetilde q)\rvert,
\]
so \(\lvert\psi_\chi\rvert\) descends to a positive function on \(Y\).

The construction is necessarily local in \(\chi\), since the arguments
\(\theta_j(\chi)\) cannot in general be chosen continuously on the whole
character torus.  On the overlap of two such character neighbourhoods,
the arguments differ by integer multiples of \(2\pi\), and the
corresponding sections differ by multiplication by a nowhere-vanishing
single-valued holomorphic function.  We record $\psi_\chi$ is the analog of $V_\Gamma$ of \cite{Widom69}.

\subsection{Boundary forms and harmonic measure}
\label{sec:boundary-harmonic-forms}

Following the potential-theoretic setup of Widom
\cite[Sec.~4]{Widom69}, we record its intrinsic form on a bordered
surface.  For corresponding Hodge star on one forms  we use the convention
\[
        dv=*du
\]
whenever \(u+iv\) is holomorphic.  On one-forms, \(d^\ast=-*d*\) denotes the
codifferential.  We write \(n_{\rm out}\) for the outward unit normal,
\(n_{\rm in}=-n_{\rm out}\) for the inward unit normal, and \(ds\) for
arclength. For the corresponding 
description in terms of the Schottky double, see
\cite[Secs.~2 and 3.1]{GustafssonSebbar}.

Let \(g_E(p,q)\) be the Dirichlet Green function of \(Y\) (see \eqref{eq:gp}), normalised by
\[
        g_E(p,q)=-\log|z(q)|+O(1),\qquad q\to p,
\]
in a local coordinate \(z\) centred at \(p\), and by
\(g_E(p,q)=0\) for \(q\in\Gamma\).

\begin{lemma}
\label{lem:poisson-boundary-form}
For each \(p\in Y^\circ\), the one-form
\[
        \Pi_p(q):=\left.-*_q d_q g_E(p,q)\right|_{\Gamma}
\]
is smooth and positive on \(\Gamma\).  In any smooth conformal metric,
\[
        \Pi_p(q)
        =
        -\frac{\partial g_E(p,q)}{\partial n_{\rm out}(q)}\,ds_q
        =
        \frac{\partial g_E(p,q)}{\partial n_{\rm in}(q)}\,ds_q .
\]

\end{lemma}

\begin{proof}
For fixed \(p\in Y^\circ\), boundary regularity for the Dirichlet problem
shows that \(g_E(p,\cdot)\) is smooth in a collar of \(\Gamma\).  Since it
vanishes on \(\Gamma\), its tangential derivative vanishes there, and hence
\[
        d_qg_E(p,q)
        =
        \frac{\partial g_E(p,q)}{\partial n_{\rm out}(q)}
        n_{\rm out}^{\flat}
        =
        \frac{\partial g_E(p,q)}{\partial n_{\rm in}(q)}
        n_{\rm in}^{\flat} .
\]
With the Stokes orientation, the restriction of
\(*n_{\rm out}^{\flat}\) to \(\Gamma\) is the positive arclength form.
The displayed formula for \(\Pi_p\) follows.  The maximum principle gives
\(g_E(p,\cdot)>0\) in \(Y^\circ\setminus\{p\}\), and the Hopf boundary lemma gives
\(\partial_{n_{\rm out}}g_E(p,\cdot)<0\); thus \(\Pi_p\) is positive.

Finally, both the Dirichlet Green function with the stated logarithmic
normalisation and the Hodge star on one-forms are conformally invariant.
Equivalently, under a conformal change of metric, the normal derivative and
the arclength element scale by reciprocal factors.  Thus \(\Pi_p\) is
intrinsic.
\end{proof}

\begin{lemma}
\label{lem:harmonic-measure-boundary-form}
For \(p\in Y^\circ\), harmonic measure at \(p\) is given by
\[
        d\omega^p(q)=\frac{1}{2\pi}\Pi_p(q).
\]
Equivalently, if \(u\in C(Y)\) is harmonic in \(Y^\circ\) and has boundary
values \(u|_\Gamma=\varphi\in C(\Gamma)\), then
\[
        u(p)=\frac{1}{2\pi}
        \int_{\Gamma}\varphi(q)\,\Pi_p(q).
\]
For \(j=1,\ldots,p\), set
\[
        \omega_j(p):=\omega^p(\Gamma_j)
        =
        \frac{1}{2\pi}\int_{\Gamma_j}\Pi_p(q).
\]
Then \(\omega_j\) is harmonic in \(Y^\circ\), extends smoothly to \(Y\), and
satisfies
\[
        \omega_j|_{\Gamma_k}=\delta_{jk},
        \qquad
        \sum_{j=1}^p\omega_j=1 .
\]
\end{lemma}

\begin{proof}
We first take \(\varphi\in C^\infty(\Gamma)\), and let \(u\) be its harmonic
extension.  By boundary regularity, \(u\) is smooth up to \(\Gamma\).  Remove
a coordinate disc \(D_\varepsilon(p)\) and put
\(Y_\varepsilon=Y\setminus D_\varepsilon(p)\).  Green's identity, with
\(\nu\) denoting the outward normal to \(Y_\varepsilon\), gives
\[
        0
        =
        \int_{\partial Y_\varepsilon}
        \left(
                u\frac{\partial g_E(p,\cdot)}{\partial\nu}
                -g_E(p,\cdot)\frac{\partial u}{\partial\nu}
        \right)ds .
\]
On \(\Gamma\), the second term vanishes.  On the inner boundary, the
expansion \(g_E(p,q)=-\log|z(q)|+O(1)\) yields
\[
        \int_{\partial D_\varepsilon(p)}
        \left(
                u\frac{\partial g_E(p,\cdot)}{\partial\nu}
                -g_E(p,\cdot)\frac{\partial u}{\partial\nu}
        \right)ds
        =2\pi u(p)+o(1).
\]
Consequently,
\[
        0
        =
        \int_\Gamma
        u(q)\frac{\partial g_E(p,q)}{\partial n_{\rm out}(q)}\,ds_q
        +2\pi u(p)+o(1).
\]
Letting \(\varepsilon\to0\) and applying
Lemma~\ref{lem:poisson-boundary-form} gives the asserted representation.

For general \(\varphi\in C(\Gamma)\), choose smooth functions converging
uniformly to \(\varphi\).  The corresponding harmonic extensions converge
uniformly on \(Y\) by the maximum principle, so the formula passes to the
limit.  Taking boundary data equal to \(1\) on \(\Gamma_j\) and to \(0\) on
the remaining components gives \(\omega_j\); these boundary data are smooth
because the components of \(\Gamma\) are disjoint.  Taking the constant
boundary value \(1\) gives \(\sum_j\omega_j=1\).
\end{proof}

\begin{lemma}
\label{lem:boundary-harmonic-forms}
For \(j=1,\ldots,p\), define
\[
        \sigma_j:=*d\omega_j .
\]
Each \(\sigma_j\) is a smooth, closed, and co-closed real one-form on \(Y\),
and
\[
        \iota_{n_{\rm in}}\sigma_j=0
        \qquad\text{on }\Gamma .
\]
Define the harmonic-measure period matrix
\[
        \mathcal P_{jk}:=\int_{\Gamma_j}\sigma_k,
        \qquad 1\le j,k\le p .
\]
Then \(\mathcal P=(\mathcal P_{jk})\) is symmetric and its row and column
sums vanish.  More precisely, for every \(c=(c_1,\ldots,c_p)\in\mathbb R^p\),
\[
        \sum_{j,k=1}^p c_j\mathcal P_{jk}c_k
        =
        \int_Y d\omega_c\wedge *d\omega_c
        =
        \int_Y |d\omega_c|^2\,dA,
        \qquad
        \omega_c:=\sum_{j=1}^p c_j\omega_j .
\]
Consequently,
\[
        \ker\mathcal P=\mathbb R(1,\ldots,1),
\]
and \(\mathcal P\) is positive definite on
\[
        (1,\ldots,1)^\perp
        =
        \left\{c\in\mathbb R^p:\sum_{j=1}^p c_j=0\right\}.
\]
\end{lemma}

\begin{proof}
Since \(\omega_j\) is harmonic,
\[
        d\sigma_j=d*d\omega_j=0.
\]
Because \(**=-1\) on one-forms,
\[
        d^\ast\sigma_j
        =
        -*d*\sigma_j
        =
        *d^2\omega_j
        =0.
\]
Thus \(\sigma_j\) is closed and co-closed.  Since \(\omega_j\) is constant
on each boundary component, the tangential part of \(d\omega_j\) vanishes on
\(\Gamma\).  Hence \(*d\omega_j\) is tangential there, which is equivalent
to \(\iota_{n_{\rm in}}\sigma_j=0\).

The identity \(\sum_k\omega_k=1\) gives \(\sum_k\sigma_k=0\), and therefore
\[
        \sum_{k=1}^p\mathcal P_{jk}=0
        \qquad (1\le j\le p).
\]
On the other hand, Stokes' theorem and \(d\sigma_k=0\) give
\[
        \sum_{j=1}^p\mathcal P_{jk}
        =
        \int_\Gamma\sigma_k
        =
        \int_Y d\sigma_k
        =0
        \qquad (1\le k\le p).
\]

Because \(\omega_j\) equals \(1\) on \(\Gamma_j\) and \(0\) on every other
boundary component, another application of Stokes' theorem yields
\[
        \mathcal P_{jk}
        =
        \int_\Gamma\omega_j\sigma_k
        =
        \int_Y d\omega_j\wedge\sigma_k
        =
        \int_Y d\omega_j\wedge *d\omega_k .
\]
This expression is symmetric in \(j\) and \(k\), and summing it against
\(c_jc_k\) gives the energy identity in the statement.

The energy vanishes precisely when \(d\omega_c=0\).  Since \(Y\) is
connected, this is equivalent to \(\omega_c\) being constant.  Its boundary
value on \(\Gamma_j\) is \(c_j\), so this occurs precisely when all the
\(c_j\) are equal.  This proves the assertions about the kernel and positive
definiteness.  Finally,
\[
        \sum_{j=1}^p c_j\sigma_j=*d\omega_c,
\]
so the same argument shows that every linear relation among the \(\sigma_j\)
is a multiple of \(\sum_j\sigma_j=0\).
\end{proof}

\begin{lemma}
\label{lem:absolute-harmonic-representatives}
Let
\[
        \mathrm{Harm}_{\hat{g}}^1(Y)
        :=
        \left\{
        \rho\in\Omega^1(Y;\mathbb R):
        d\rho=0,\ d^\ast\rho=0,
        \iota_{n_{\rm in}}\rho=0\text{ on }\Gamma
        \right\}.
\]
Then the natural map
\[
        \mathrm{Harm}_{\hat{g}}^1(Y)\longrightarrow H^1_{\mathrm{dR}}(Y;\mathbb R),
        \qquad
        \rho\longmapsto[\rho],
\]
is an isomorphism.  
\end{lemma}

\begin{proof}
This is the Hodge--Morrey--Friedrichs theorem with absolute boundary
conditions; see \cite[Theorem~2.6.1 and Ch.~3]{Schwarz1995}.  In the
standard formulation, the absolute boundary conditions for a one-form
\(\rho\) are
\(\iota_{n_{\rm out}}\rho=0\) and
\(\iota_{n_{\rm out}}d\rho=0\).  The second condition is automatic here
because \(d\rho=0\), while the first is unchanged when \(n_{\rm out}\) is
replaced by \(n_{\rm in}=-n_{\rm out}\).  The stated space is therefore
exactly the space of absolute harmonic one-forms representing
\(H^1_{\mathrm{dR}}(Y;\mathbb R)\).
\end{proof}

\subsection{A discussion on the canonical homology basis}
We assumed $X$ has genus $g$. Choose a canonical homology basis
\[
        \mathfrak a_1,\ldots,\mathfrak a_g,
        \mathfrak b_1,\ldots,\mathfrak b_g
\]
of \(X\), represented by smooth curves in \(Y^\circ\), such that
\[
        \mathfrak a_j\cdot\mathfrak a_k
        =
        \mathfrak b_j\cdot\mathfrak b_k
        =0,
        \qquad
        \mathfrak a_j\cdot\mathfrak b_k
        =\delta_{jk}.
\]
We use fraktur letters here to distinguish these cycles from points of \(Y\). For canonical homology bases and their dual harmonic differentials, see \cite[Secs.~I.4 and III.2]{FarkasKra}.

 The inclusion \(Y\hookrightarrow X\) induces a surjection
\[
        H_1(Y,\mathbb Z)\longrightarrow H_1(X,\mathbb Z),
\]
whose kernel is generated by \([\Gamma_1],\ldots,[\Gamma_p]\), subject to the single relation
\[
        [\Gamma_1]+\cdots+[\Gamma_p]=0.
\]
Consequently,
\[
        \mathfrak a_1,\ldots,\mathfrak a_g,
        \mathfrak b_1,\ldots,\mathfrak b_g,
        \Gamma_1,\ldots,\Gamma_{p-1}
\]
form an adapted basis of \(H_1(Y,\mathbb Z)\), and in particular
\[
        \operatorname{rank}H_1(Y,\mathbb Z)=2g+p-1.
\]
For a closed real one-form \(\tau\) on \(Y\), we call
\[
        \int_{\Gamma_j}\tau
\]
its boundary periods and
\[
        \int_{\mathfrak a_j}\tau,
        \qquad
        \int_{\mathfrak b_j}\tau
\]
its handle periods. Stokes' theorem gives
\[
        \sum_{j=1}^p\int_{\Gamma_j}\tau=0,
\]
so only \(p-1\) boundary periods are independent. By de Rham's theorem, \(\tau\) is exact precisely when all its boundary and handle periods vanish.
It is useful to express the same decomposition in terms of harmonic forms. Define
\[
        \mathrm{Harm}_{h}^1(Y)
        :=
        \left\{
        \rho\in\mathrm{Harm}_{\hat{g}}^1(Y):
        \int_{\Gamma_j}\rho=0,\quad j=1,\ldots,p
        \right\}.
\]
By Lemma~\ref{lem:absolute-harmonic-representatives}, the handle-period map
\[
        \mathrm{Harm}_{h}^1(Y)\longrightarrow\mathbb R^{2g},
        \qquad
        \rho\longmapsto
        \left(
        \int_{\mathfrak a_1}\rho,\ldots,\int_{\mathfrak a_g}\rho,
        \int_{\mathfrak b_1}\rho,\ldots,\int_{\mathfrak b_g}\rho
        \right)
\]
is an isomorphism. Thus there are unique forms
\[
        \xi_1,\ldots,\xi_g,\zeta_1,\ldots,\zeta_g
        \in\mathrm{Harm}_{h}^1(Y)
\]
satisfying
\[
        \int_{\mathfrak a_k}\xi_j=\delta_{jk},
        \qquad
        \int_{\mathfrak b_k}\xi_j=0,
        \qquad
        \int_{\mathfrak a_k}\zeta_j=0,
        \qquad
        \int_{\mathfrak b_k}\zeta_j=\delta_{jk}.
\]
Together with Lemma~\ref{lem:boundary-harmonic-forms}, this gives the decomposition
\[
        \mathrm{Harm}_{\hat{g}}^1(Y)
        =
        \operatorname{span}\{\sigma_1,\ldots,\sigma_p\}
        \oplus
        \mathrm{Harm}_{h}^1(Y).
\]
The first summand has dimension \(p-1\) and controls the boundary periods, whereas the second has dimension \(2g\) and controls the handle periods.
\subsection{Neumann functions}

Neumann functions on finite bordered Riemann surfaces are classical; See \cite[Sec.~4.2]{SS54} \cite{AS60}.  In the planar
setting relevant to the present work, Widom constructs the corresponding
function from Green functions and harmonic measures in
\cite[Sec.~4]{Widom69},.  We give an explicit intrinsic construction, both
to fix our signs and normalizations and to make clear how the boundary and
handle periods are removed.

\begin{definition}
\label{def:neumann-function}
Let \(Y\) be a compact bordered Riemann surface with smooth boundary
\(\Gamma=\Gamma_1\cup\cdots\cup\Gamma_p\), and let
\(a,b\in Y^\circ\), \(a\neq b\).  A Neumann function with poles at \(a\)
and \(b\) is a real-valued function
\[
        N_Y(\,\cdot\,;a,b):Y\setminus\{a,b\}\longrightarrow\mathbb R,
\]
defined up to an additive constant, which is harmonic on
\(Y^\circ\setminus\{a,b\}\), satisfies
\[
        N_Y(q;a,b)=-\log|z(q)|+O(1)\quad(q\to a),\qquad
        N_Y(q;a,b)=\log|w(q)|+O(1)\quad(q\to b),
\]
in local coordinates \(z\) and \(w\) centred at \(a\) and \(b\),
respectively, and obeys
\[
        \frac{\partial N_Y}{\partial n_{\rm in}}=0
        \qquad\text{on }\Gamma .
\]
\end{definition}

\begin{proposition}
\label{prop:construction-neumann-function}
For every pair of distinct points \(a,b\in Y^\circ\), the Neumann function
\(N_Y(\,\cdot\,;a,b)\) exists and is unique up to an additive constant.
\end{proposition}

\begin{proof}
We use the notation of Section~\ref{sec:boundary-harmonic-forms}.  Embed
\(Y\) in a compact Riemann surface \(X\), and let \(\eta_{a,b}\) be the (normalized such that all the periods are imaginary) differential of
the third kind on \(X\), with residues \(-1\) at \(a\) and \(1\) at \(b\).
Thus
\[
        \operatorname{Res}_a\eta_{a,b}=-1,\qquad
        \operatorname{Res}_b\eta_{a,b}=1,\qquad
        \operatorname{Re}\int_\gamma\eta_{a,b}=0
\]
for every closed curve \(\gamma\subset X\setminus\{a,b\}\).

Locally write
\[
        \int^q\eta_{a,b}=H(q)+iA(q).
\]
The real part \(H\) is single-valued on \(Y\setminus\{a,b\}\) and has the
required logarithmic singularities.  Moreover, \(A\) is single-valued
along each boundary component.  Indeed, \(\Gamma_j\) bounds one of the discs, on which \(\eta_{a,b}\) is holomorphic, and hence
\(\int_{\Gamma_j}\eta_{a,b}=0\).

Let \(h\) be the harmonic function on \(Y\) whose boundary values agree
with those of \(A\).  If \(\widetilde h\) is a local harmonic conjugate of
\(h\), then
\[
        H+\widetilde h+i(A-h)
\]
is holomorphic.  Its imaginary part vanishes on \(\Gamma\), so the
Cauchy--Riemann equations show that the real part has zero normal
derivative there.  Consequently,
\[
        \beta_0:=dH+*dh
\]
is a globally defined closed and co-closed real one-form on
\(Y\setminus\{a,b\}\), with
\[
        \beta_0(n_{\rm in})=0\qquad\text{on }\Gamma .
\]
Thus \(\beta_0\) already has the correct singularities and boundary
behaviour.  The only remaining obstruction to integrating it to a
single-valued function is its periods.

We first remove the boundary periods.  Let
\(\sigma_k=*d\omega_k\) and \(\mathcal P\) be as in
Lemma~\ref{lem:boundary-harmonic-forms}, and put
\[
        b_j:=\int_{\Gamma_j}\beta_0,\qquad j=1,\ldots,p.
\]
Stokes' theorem, applied to \(Y\) with small discs about \(a\) and \(b\)
removed, gives \(\sum_j b_j=0\); the integrals over the small circles
vanish because logarithmic differentials have zero periods.  By Lemma \ref{lem:boundary-harmonic-forms}
\(\mathcal P\) is invertible on the hyperplane 
\(\sum_k\lambda_k=0\), there are unique real numbers
\(\lambda_1,\ldots,\lambda_p\), subject to this normalization, such that
\[
        \sum_{k=1}^p\lambda_k\mathcal P_{jk}=-b_j,
        \qquad j=1,\ldots,p.
\]
It follows that
\[
        \beta_1:=\beta_0+\sum_{k=1}^p\lambda_k\sigma_k
\]
is closed and co-closed, has zero normal component on \(\Gamma\), and has
zero period around every boundary component.

The periods of \(\beta_1\) around small loops enclosing \(a\) and \(b\)
also vanish.  Its remaining periods therefore define a class in
\(H^1_{\mathrm{dR}}(Y;\mathbb R)\).  Let \(\rho\) be the absolute harmonic
representative of this class, as in
Lemma~\ref{lem:absolute-harmonic-representatives}.  Then
\[
        d\rho=0,\qquad d^\ast\rho=0,\qquad
        \rho(n_{\rm in})=0\quad\text{on }\Gamma,
\]
and \(\rho\) has the same handle periods as \(\beta_1\).  Hence
\[
        \alpha:=\beta_1-\rho
\]
has zero period around every closed curve in
\(Y\setminus\{a,b\}\).  Since \(\alpha\) is closed, it is exact, and we may
define
\[
        N_Y(q;a,b):=C+\int_{q_0}^{q}\alpha,
        \qquad q_0\in Y^\circ\setminus\{a,b\}.
\]
The integral is independent of the path.  Since \(d^\ast\alpha=0\), the
function \(N_Y\) is harmonic away from \(a\) and \(b\), while
\[
        \frac{\partial N_Y}{\partial n_{\rm in}}
        =
        \alpha(n_{\rm in})=0
        \qquad\text{on }\Gamma .
\]
All the correction forms are smooth near \(a\) and \(b\), so the
singularities come entirely from \(H\).  Thus
\[
        N_Y(q;a,b)=-\log|z(q)|+O(1)\quad(q\to a),\qquad
        N_Y(q;a,b)=\log|w(q)|+O(1)\quad(q\to b).
\]
This proves existence.

If \(N_1\) and \(N_2\) are two Neumann functions with the same poles, then
\(u=N_1-N_2\) is bounded and harmonic near \(a\) and \(b\), and hence
extends harmonically across both points.  Since
\(\partial u/\partial n_{\rm in}=0\) on \(\Gamma\), Green's identity gives
\[
        \int_Y|\nabla u|^2\,dA
        =
        -\int_\Gamma
        u\,\frac{\partial u}{\partial n_{\rm in}}\,ds
        =
        0.
\]
Therefore \(u\) is constant, proving uniqueness up to an additive
constant.
\end{proof}
\begin{lemma}
\label{lem:analytic-widom-41}
Let \(K\Subset Y^\circ\).  There is a two-sided analytic collar \(U_K\)
of \(\Gamma\), whose \(Y\)-side is disjoint from \(K\), with the following
properties.

\begin{enumerate}
\item For \(a\in K\), let \(g_a(p)=g_E(p,a)\).  The function \(g_a\) extends
harmonically through \(\Gamma\) to \(U_K\), and
\[
        d\mathcal G_a:=2\partial g_a
\]
extends there as a holomorphic differential.  The collar may be chosen so
that \(d\mathcal G_a\) has no zeros in \(U_K\), for every \(a\in K\).  Thus
every local branch of
\[
        \mathcal G_a=g_a+i\widetilde g_a
\]
and hence of \(\Phi_a=\exp\mathcal G_a\) continues holomorphically through
the boundary.  These local continuations have unitary monodromy and define
a non-vanishing holomorphic section of the corresponding flat unitary line
bundle.

\item Each harmonic measure \(\omega_j\) extends harmonically to \(U_K\),
and the differential
\[
        d\Omega_j:=2\partial\omega_j
\]
extends holomorphically there.

\item If \(a,b\in K\) and \(a\ne b\), then
\(N_Y(\,\cdot\,;a,b)\) extends harmonically to \(U_K\).  Consequently every
local branch of
\begin{equation}\label{eq:expneu}
       \Psi_{a,b}
        :=
        \exp\bigl(
        N_Y(\,\cdot\,;a,b)
        +i\widetilde N_Y(\,\cdot\,;a,b)
        \bigr)
\end{equation}

continues holomorphically and without zeros through \(\Gamma\), again as a
section of its flat unitary line bundle.

\item Let \(f\) be real analytic on \(\Gamma\), and let \(u_f\) be its
harmonic extension to \(Y\).  There is a two-sided collar
\(U_f\subset U_K\), possibly depending on \(f\), to which \(u_f\) extends
harmonically.  It follows that
\[
        R_f:=\exp(u_f+i\widetilde u_f)
\]
continues through \(\Gamma\) as a non-vanishing holomorphic section of a
flat unitary line bundle, with
\[
        |R_f|=e^f
        \qquad\text{on }\Gamma.
\]
\end{enumerate}
\end{lemma}

\begin{proof}
The Schwarz reflection principle for harmonic functions is the only local
input; see, for example, \cite[Sec.~7.5.2]{KrantzGuide}.  Since \(\Gamma\)
is real analytic, every point of \(\Gamma\) has a two-sided holomorphic
coordinate \(z=x+iy\) in which \(\Gamma=\{y=0\}\) and \(Y\) lies in
\(\{y\ge0\}\).  If a harmonic function \(v\) has the constant boundary
value \(c\), then it is continued to \(y<0\) by
\[
        \widehat v(x,y)=2c-v(x,-y).
\]
If instead \(\partial_yv=0\) on \(y=0\), its continuation is
\[
        \widehat v(x,y)=v(x,-y).
\]
The local extensions agree on overlaps by uniqueness and therefore patch
along each boundary component.

Choose the collar so that its \(Y\)-side is disjoint from \(K\).  For
\(a\in K\), the function \(g_a\) is harmonic in this collar and vanishes on
\(\Gamma\), so the first reflection formula applies.  Since a harmonic
function has holomorphic differential \(2\partial g_a\), the reflected
differential is holomorphic.  On the boundary,
\[
        \iota^*(2\partial g_a)=-i\Pi_a,
\]
where \(\Pi_a\) is the positive Poisson boundary form from
Lemma~\ref{lem:poisson-boundary-form}.  The Green kernel is smooth off the
diagonal, so \((q,a)\mapsto\Pi_a(q)\) is continuous on the compact set
\(\Gamma\times K\).  Hopf's lemma makes it strictly positive there.  Thus,
after shrinking the collar once, \(2\partial g_a\) has no zeros in \(U_K\)
for any \(a\in K\).  On a simply connected part of the collar, a primitive
of \(2\partial g_a\) can be normalized to have real part \(g_a\); it is then
a branch of \(\mathcal G_a\), and exponentiation gives the continuation of
\(\Phi_a\).

On \(\Gamma_\ell\), the harmonic measure \(\omega_j\) has the constant
boundary value \(\delta_{j\ell}\).  We therefore reflect
\(\omega_j-\delta_{j\ell}\) oddly, or equivalently set
\[
        \widehat\omega_j(x,y)
        =
        2\delta_{j\ell}-\omega_j(x,-y),
        \qquad y<0.
\]
This proves the second assertion.  For the Neumann function the poles stay
away from the collar, while
\(\partial N_Y/\partial n=0\) on \(\Gamma\); the even reflection therefore
gives the third assertion.

Finally, let \(f\) be real analytic.  In a boundary coordinate it has a
holomorphic extension \(A\) to a possibly smaller neighbourhood of the real
axis, with \(A=f\) on that axis.  The harmonic function
\(u_f-\operatorname{Re}A\) vanishes on the boundary and hence extends by odd
reflection.  Adding back \(\operatorname{Re}A\) gives the asserted
continuation of \(u_f\).  The size of this neighbourhood may depend on the
analyticity radius of \(f\), which is why the collar is denoted by \(U_f\).
For a family with uniformly bounded holomorphic complexifications in a
fixed collar, the same reflection formula and the maximum principle give
the stated uniform choice.

For any single-valued real harmonic function \(v\), the periods of a local
conjugate \(\widetilde v\) are real.  Analytic continuation of
\(\exp(v+i\widetilde v)\) therefore changes a branch only by a unimodular
constant.  This proves all the line-bundle statements.
\end{proof}

For the boundary argument later on, one also needs the parameter-uniform
local form of a Neumann factor which is contained in Widom's
Lemma~4.1(2).  We record it separately.

We shall need the behaviour of a Neumann factor when its zero approaches the
boundary. The Neumann function is determined only up to an additive constant.  We fix
this ambiguity by the intrinsic normalization
\[
        \int_\Gamma
        N_Y(q;a,b)\,d\omega^{P_\infty}(q)=0,
\]
where \(d\omega^{P_\infty}\) is harmonic measure at \(P_\infty\).  With this choice, \(N_Y(\,\cdot\,;a,b)\) depends continuously on
\((a,b)\), locally uniformly in \(q\) away from the  singularities.
\[
        N_Y(\,\cdot\,;a,b)
        =
        -N_Y(\,\cdot\,;b,a),
\]
and we set \(N_Y(\,\cdot\,;a,a)=0\).  Let \(K\Subset Y^\circ\),
and let \(I\) be a closed boundary arc contained in an analytic coordinate
neighbourhood \(V\), disjoint from \(K\).  Choose the coordinate
\(z=s+it\) so that
\[
        \Gamma\cap V=\{t=0\},
        \qquad
        Y\cap V=\{t\ge0\}.
\]

Suppose that \(a\in K\) and that \(b\in V\cap Y^\circ\).  Even reflection of
\(N_Y(\,\cdot\,;a,b)\) across the boundary produces a second logarithmic
singularity at the mirror point \(\overline{z(b)}\).  It follows that
\[
        N_Y^{\mathrm{ext}}(p;a,b)
        -\log|z(p)-z(b)|
        -\log|z(p)-\overline{z(b)}|
\]
is harmonic in a smaller two-sided neighbourhood of \(I\).  Consequently,
on \(I\),
\begin{equation}\label{eq:dz}
            e^{N_Y(q;a,b)}
        =
        u(q;a,b)\,|z(q)-z(b)|^2,
\end{equation}

where \(u\) is positive, real analytic in \(q\), and continuous in all its
variables as \(b\) approaches the boundary.  After shrinking \(V\), if
necessary, \(u\) is bounded above and below by positive constants, uniformly
for \(q\in I\), \(a\in K\), and \(b\in V\cap Y^\circ\).  Thus, writing
\(z(q)=s\) and \(z(b)=\alpha+i\tau\), one obtains the familiar local form
\[
        e^{N_Y(q;a,b)}
        =
        u(s;a,b)\bigl((s-\alpha)^2+\tau^2\bigr).
\]
If \(b\) remains in a compact set separated from \(I\), the Neumann factor
itself is bounded above and below on \(I\) by positive constants.

\section{Duality and Proofs of Theorem \ref{thm:duality}}
Let \(h\) be the solution of the Dirichlet problem on \(Y\) with boundary values \(\log\rho\). Locally choose a harmonic conjugate \(\widetilde h\) and put
\begin{equation}
      R_\rho:=\exp(h+i\widetilde h).
\end{equation}
      
Although \(\widetilde h\) need not be single-valued, its increments along closed curves are real constants. Thus \(R\) is a nowhere-vanishing  section of the line bundle with unitary character, and
\begin{equation}
        |R_\rho|=\rho
        \qquad\text{on }\Gamma.
\end{equation}
The harmonic function \(h\) is uniquely determined, while its harmonic conjugate is determined up to an additive real constant. Accordingly, \(R_\rho\) is determined up to a unimodular factor. We fix this factor by choosing the branch at \(P_\infty\) for which
\[
        R\rho(P_\infty)>0.
\]

\label{ss:proof-duality}

\subsection{Proof of Theorem~\ref{thm:duality}}
\label{ss:proof-duality2}
\begin{proof}
For \(\rho\equiv1\), this is the rank-one case of Widom's
bundle-valued duality lemma
\cite[Sec.~3, Lemma, pp.~308--309]{Widom71a}, which also guarantees existence of the extremal section.
We only explain how to remove the weight.

Let \(R_\rho\) be the outer section associated with \(\rho\), normalized by
\[
        |R_\rho|=\rho\quad\text{on }\Gamma,
        \qquad
        r_\rho:=R_\rho(P_\infty)>0,
\]
and let \(L_\rho\) be the flat unitary line bundle determined by its
monodromy.  The construction of \(R_\rho\) and this removal of the
weight are the same as in
\cite[p.~155 and Eq.~(5.20), p.~159]{Widom69}.  Since \(\rho\) is
positive and continuous, \(R_\rho\) and \(R_\rho^{-1}\) are bounded.

Put \(\widehat L=L\otimes L_\rho\).  The changes of variables
\[
        \widehat F=\frac{F R_\rho}{r_\rho},
        \qquad
        \widehat\eta=\frac{\eta}{R_\rho}
\]
give bijections between the corresponding weighted and unweighted
classes, and
\[
        \widehat F(P_\infty)=1,\qquad
        \|\widehat F\|_\infty
        =r_\rho^{-1}\|F\|_{\rho,\infty},
\]
while
\[
        \|\widehat\eta\|_1
        =\|\eta\|_{\rho^{-1},1},
        \qquad
        \operatorname{Res}_{P_\infty}\widehat\eta
        =r_\rho^{-1}\operatorname{Res}_{P_\infty}\eta.
\]
Consequently,
\[
        \mu(L,\rho)=r_\rho\,\mu(\widehat L,1).
\]

\end{proof}
\begin{lemma}
\label{lem:dual-attainment}
The supremum in Theorem~\ref{thm:duality} is attained.
\end{lemma}

\begin{proof}
After the changes of variables used above, this is the unweighted
existence statement.  It follows from Widom's maximizing-sequence
argument \cite[p.~151]{Widom69}.  Indeed, a maximizing
sequence is locally bounded, so a subsequence converges uniformly on
compact subsets.  The residue passes to the limit, while the
\(\mathcal{H}^1\)-norm is lower semicontinuous.  If the norm of the limit were
strictly less than one, rescaling it would give a larger value than the
supremum.  The limit therefore has norm one and attains the supremum.
Multiplying back by \(R_\rho\) gives the required differential in
\(\mathcal H^1(L^{-1};P_\infty)\).
\end{proof}

The maximizing differential need not be unique when the boundary has
more than one component; see \cite[p.~153]{Widom69}.

\section{Widom's product formula}
\label{sec:widom-product}

Let
\[
        g(p):=g_E(p,P_\infty)
\]
be the Dirichlet Green function of \(Y\) with pole at \(P_\infty\), and set
\[
        d\mathcal G:=2\partial g=dg+i*dg .
\]
Away from \(P_\infty\), the function \(g\) is harmonic, and hence
\(d\mathcal G\) is holomorphic.  In the fixed local coordinate \(\zeta\)
at \(P_\infty\), one has
\[
        g=-\log|\zeta|+h,
\]
where \(h\) is harmonic near \(P_\infty\).  Consequently,
\[
        d\mathcal G
        =
        -\frac{d\zeta}{\zeta}+2\partial h .
\]
Thus \(d\mathcal G\) is a meromorphic differential on \(Y^\circ\), with
a simple pole of residue \(-1\) at \(P_\infty\).

Let \(\iota:\Gamma\hookrightarrow Y\) denote the boundary inclusion.
Since \(g\) is constant on \(\Gamma\), one has \(\iota^*(dg)=0\).  Set
\[
        \beta_g:=-\iota^*(*dg).
\]
The boundary form \(\beta_g\) is smooth and positive by Hopf's lemma;
indeed, \(\beta_g/(2\pi)\) is harmonic measure at \(P_\infty\).
Therefore
\begin{equation}\label{eq:dG-boundary-product}
        \iota^*(d\mathcal G)=-i\beta_g,
        \qquad
        \bigl|\iota^*(d\mathcal G)\bigr|=\beta_g .
\end{equation}

In a planar exterior domain, write
\(d\mathcal G=G'(\zeta)\,d\zeta\).  Widom orients the planar boundary
curves oppositely to the Stokes orientation of the exterior domain;
with his convention, \eqref{eq:dG-boundary-product} becomes
\[
        iG'(\zeta)^{-1}\frac{|d\zeta|}{d\zeta}>0,
\]
which is \cite[(5.7)]{Widom69}.  With the Stokes orientation used
here, the left-hand side has the opposite sign.  There is also a
difference in the normalization at infinity: Widom takes
\(zf(z)\to1\), so that \(f(z)\,dz\) has residue \(-1\) at infinity,
whereas our dual differential is normalized to have positive residue.

We have the following lemma.

\begin{lemma}
\label{lem:widom-boundary-relations}
Let \(F_0\) be extremal for \(\mu(L,\rho)\), and let
\[
        \eta_0\in
        \mathcal H^1\bigl(Y,K_Y\otimes L^{-1}(P_\infty)\bigr)
\]
be a dual extremal, normalized by
\[
        \frac1{2\pi}\int_\Gamma |\eta_0|\rho^{-1}=1,
        \qquad
        \operatorname{Res}_{P_\infty}\eta_0
        =\mu(L,\rho)>0 .
\]
Then, almost everywhere on \(\Gamma\),
\[
        |F_0|\rho=\mu(L,\rho),
        \qquad
        \frac1{i}\,\iota^*(F_0\eta_0)\ge0 .
\]
Here the second inequality means that
\(\frac1{i}\,\iota^*(F_0\eta_0)\) is a non-negative boundary measure
with respect to the Stokes orientation.  Equivalently,
\[
        \frac{\iota^*(F_0\eta_0)}
             {\iota^*(d\mathcal G)}
        \le0
        \qquad\text{almost everywhere on \(\Gamma\)}.
\]
\end{lemma}

\begin{proof}
Since \(F_0(P_\infty)=1\), the residue theorem gives
\[
        \mu(L,\rho)
        =
        \operatorname{Res}_{P_\infty}(F_0\eta_0)
        =
        \frac1{2\pi i}\int_\Gamma F_0\eta_0 .
\]
Therefore
\[
\begin{aligned}
        \mu(L,\rho)
        &=
        \left|
        \frac1{2\pi i}\int_\Gamma F_0\eta_0
        \right|                                                     \\
        &\le
        \frac1{2\pi}\int_\Gamma |F_0|\,|\eta_0|                      \\
        &\le
        \|F_0\|_{\rho,\infty}
        \frac1{2\pi}\int_\Gamma |\eta_0|\rho^{-1}
        =
        \mu(L,\rho).
\end{aligned}
\]
Thus equality holds throughout.  Equality in the second inequality gives
\[
        |F_0|\rho=\mu(L,\rho)
\]
wherever the boundary value of \(\eta_0\) is nonzero.  Since
\(\operatorname{Res}_{P_\infty}\eta_0\ne0\), the differential \(\eta_0\)
is not identically zero, and we have  \(|\eta_0|>0\) almost everywhere on \(\Gamma\); see
\cite[p.~62]{Rudin1955}.  Hence the first boundary relation holds almost
everywhere.

Equality in the first inequality says that the boundary differential
\(F_0\eta_0\) has constant argument.  Since
\[
        \int_\Gamma F_0\eta_0
        =
        2\pi i\,\mu(L,\rho),
\]
this argument is \(i\), and therefore
\[
        \frac1{i}\,\iota^*(F_0\eta_0)\ge0 .
\]
Finally, using
\[
        \iota^*(d\mathcal G)=-i\beta_g,
        \qquad
        \beta_g>0,
\]
we obtain
\[
        \frac{\iota^*(F_0\eta_0)}
             {\iota^*(d\mathcal G)}
        =
        -\frac{\frac1{i}\,\iota^*(F_0\eta_0)}
               {\beta_g}
        \le0
\]
almost everywhere on \(\Gamma\).
\end{proof}
We next record the local continuation principle for the extremal pair.
It is contained in
\cite[Lemmas~4.3--4.4]{Read1958}.

\begin{lemma}
\label{lem:widom-local-regularity}
Let \(I\subset\Gamma\) be an open analytic boundary arc.  Suppose
\[
        F\in \mathcal{H}^{\infty}(Y,L),
        \qquad
        H\in
        \mathcal H^1\bigl(Y,K_Y\otimes L^{-1}(P_\infty)\bigr),
        \qquad
        H\not\equiv0.
\]
Assume that, for some positive real analytic function \(a\) on \(I\),
\[
        |F|=a
        \qquad\text{a.e. on \(I\)},
\]
and that the scalar boundary differential \(FH\) is non-negative on
\(I\), with respect to the Stokes orientation:
\[
        FH\ge0
        \qquad\text{a.e. on \(I\)}.
\]
Then \(F\) and \(H\) extend holomorphically across \(I\), in local
holomorphic frames.  Moreover, \(F\) has no zeros on \(I\), and every
zero of \(H\) on \(I\) has even order.
\end{lemma}

\begin{proof}
The assertion is local.  After shrinking the boundary chart, we may
identify \(I\) with an interval on the real axis, with \(Y\) lying in
the upper half-plane, and choose a local unitary frame of \(L\) (and the dual frame for $L^{-1}$.

Since \(a\) is positive and real analytic, it has a non-vanishing
holomorphic extension \(A\) to a neighbourhood of \(I\), which we may
choose so that
\[
        A=a>0
        \qquad\text{on \(I\)}.
\]
Replacing
\[
        F\quad\text{by}\quad \widehat F:=\frac{F}{A},
        \qquad
        H\quad\text{by}\quad \widehat H:=AH,
\]
does not change the product \(FH\), and reduces the modulus condition
to
\[
        |\widehat F|=1
        \qquad\text{a.e. on \(I\)}.
\]
It is therefore enough to treat this normalized case. Write, in the chosen coordinate and frame,
\[
        \widehat F=f,
        \qquad
        \widehat H=h(z)\,dz .
\]
Then \(f\in \mathcal{H}^{\infty}\), \(h\in \mathcal{H}^1\), and the boundary relations become
\[
        |f|=1,
        \qquad
        fh\ge0
        \qquad\text{a.e. on \(I\)}.
\]
In particular, \(fh\) is real on \(I\).  Since
\(\overline f=f^{-1}\) there, we obtain
\[
        \overline h=f^2h
        \qquad\text{a.e. on \(I\)}.
\]
Consequently,
\[
        h+f^2h=h+\overline h
\]
has real boundary values on \(I\), while
\[
        h-f^2h=h-\overline h
\]
has purely imaginary boundary values.  Both functions belong locally
to \(\mathcal{H}^1\), because \(f\in \mathcal{H}^{\infty}\) and \(h\in \mathcal{H}^1\).  The
 Schwarz reflection  therefore shows
that both extend holomorphically across \(I\).  It follows that \(h\)
and \(f^2h\) extend holomorphically across \(I\).

Since \(H\not\equiv0\), the extended function \(h\) is not identically
zero.  Hence
\[
        q:=\frac{f^2h}{h}
\]
is meromorphic in a neighbourhood of \(I\) and agrees with \(f^2\) on
the original side of the boundary.  Away from the zeros of \(h\), the
function \(q\) is continuous across \(I\) and has modulus one almost
everywhere there; hence
\[
        |q|=1
\]
throughout those subarcs.  It follows that \(q\) can have neither a
pole nor a zero at a zero of \(h\) on \(I\): either would contradict
\(|q|=1\) along the punctured interval.  Thus all such singularities
are removable, and \(q\) is holomorphic and non-vanishing in a
neighbourhood of \(I\).

After shrinking the neighbourhood to be simply connected, \(q\) has
a holomorphic square root.  Choosing its sign to agree with \(f\) on
the original side gives the holomorphic continuation of \(f\) across
\(I\).  This continuation is non-vanishing because \(q=f^2\) is
non-vanishing.  Since \(h\) has already been continued, both
\(\widehat F\) and \(\widehat H\), and hence also \(F\) and \(H\),
extend holomorphically across \(I\).

Finally, the product \(fh\) extends holomorphically across \(I\), and
its restriction to \(I\) is real and non-negative.  A nonzero real
analytic function which is non-negative on an interval can have only
zeros of even order.  Since \(f\) is non-vanishing on \(I\), the zeros
of \(h\) have the same orders as the zeros of \(fh\).  Therefore every
zero of \(H\) on \(I\) has even order.
\end{proof}

\begin{lemma}
\label{lem:widom-5-3-surface}
Let
\[
        F_1\in \mathcal{H}^{\infty}(Y,L),\qquad F_1(P_\infty)=1,
\]
and let
\[
        \eta_1\in
        \mathcal H^1\bigl(Y,K_Y\otimes L^{-1}(P_\infty)\bigr),
        \qquad
        \frac1{2\pi}\int_\Gamma |\eta_1|\rho^{-1}=1.
\]
Suppose that, for some \(M>0\),
\[
        |F_1|\rho=M
        \qquad\text{\(ds\)-a.e. on \(\Gamma\)},
\]
and that
\[
        \frac1{i}\,\iota^*(F_1\eta_1)
\]
is a non-negative boundary measure. Equivalently,
\[
        \frac{\iota^*(F_1\eta_1)}
             {-\iota^*(d\mathcal G)}
        \ge0
        \qquad\text{\(ds\)-a.e. on \(\Gamma\)}.
\]
Then \(M=\mu(L,\rho)\), the function \(F_1\) is extremal, and
\(\eta_1\) attains the maximum in the dual problem. Moreover, \(F_1\)
is the unique extremal function.
\end{lemma}

\begin{proof}
Since \(F_1(P_\infty)=1\), the residue theorem and the sign assumption give
\[
\begin{aligned}
        \operatorname{Res}_{P_\infty}\eta_1
        &=
        \frac1{2\pi i}\int_\Gamma F_1\eta_1
         =
        \frac1{2\pi}\int_\Gamma |F_1|\,|\eta_1|  \\
        &=
        M\,\frac1{2\pi}\int_\Gamma |\eta_1|\rho^{-1}
         =
        M.
\end{aligned}
\]
Thus the dual formulation gives \(M\le\mu(L,\rho)\). On the other hand,
\(F_1\) is admissible and
\[
        \|F_1\|_{\rho,\infty}=M,
\]
so \(\mu(L,\rho)\le M\). Consequently,
\[
        M=\mu(L,\rho),
\]
and equality is attained by both \(F_1\) and \(\eta_1\).

It remains to prove uniqueness. Let \(F_0\) be another extremal function.
Applying Lemma~\ref{lem:widom-boundary-relations} to \(F_0\) and
\(\eta_1\), we obtain
\[
        |F_0|\rho=\mu(L,\rho)
        \quad\text{-a.e. on \(\Gamma\)},\qquad
        \frac1{i}\,\iota^*(F_0\eta_1)\ge0.
\]
The same relations hold for \(F_1\). Since the two products have the same
boundary argument and
\[
        |F_0|\,|\eta_1|
        =
        \mu(L,\rho)\rho^{-1}|\eta_1|
        =
        |F_1|\,|\eta_1|,
\]
it follows that
\[
        \iota^*(F_0\eta_1)=\iota^*(F_1\eta_1).
\]
 Hence
\[
        F_0=F_1
        \qquad \text{a.e. on} \quad \Gamma .
\]
 Since \(L\) is flat unitary, \(|F_0-F_1|^2\) is a
globally defined subharmonic function on \(Y\). Hence, for every
\(p\in Y^\circ\),
\[
        |F_0(p)-F_1(p)|^2
        \le
        \int_\Gamma |F_0(q)-F_1(q)|^2\,d\omega^p(q)
        =
        0.
\]
 and therefore \(F_1=F_0\).
\end{proof}

\subsection{Counting zeros}
The critical points of the Green function are the zeros of its complex
differential \(d\mathcal G=2\partial g\). For a finitely connected planar
domain, their number is computed in
\cite[Sec.~2.1, p.~944]
{GustafssonSebbar} using Schottky double. We briefly record the corresponding index calculation
for the present surface. Near \(P_\infty\), in the fixed coordinate \(\zeta\),
\[
        g=-\log|\zeta|+h,
\]
where \(h\) is harmonic. Hence
\[
        d\mathcal G
        =
        -\frac{d\zeta}{\zeta}
        +\text{a holomorphic differential},
\]
so \(d\mathcal G\) has a simple pole at \(P_\infty\). On the other hand,
Hopf's lemma and \eqref{eq:dG-boundary-product} give
\[
        \iota^*(d\mathcal G)=-i\beta_g,
        \qquad
        \beta_g>0,
\]
and therefore \(d\mathcal G\) has no zeros on \(\Gamma\).

To count the interior zeros, choose a smooth conformal metric on \(Y\) and
a smooth function \(\kappa\), positive away from \(P_\infty\), such that
\(\kappa=|\zeta|^2\) near \(P_\infty\). The vector field
\[
        V:=\kappa\nabla g
\]
extends smoothly across \(P_\infty\). Its index there is \(1\). If \(a\) is
a zero of \(d\mathcal G\) of order \(m\), then in a local coordinate
centred at \(a\),
\[
        d\mathcal G=(cz^m+O(z^{m+1}))\,dz,
        \qquad c\ne0.
\]
Since the complex direction of \(\nabla g\) is a positive multiple of the
conjugate of the coefficient of \(d\mathcal G\), the index of \(V\) at
\(a\) is \(-m\). Finally, \(V\) points inward on \(\Gamma\), by Hopf's
lemma. Applying the Poincaré--Hopf theorem to \(-V\)
\cite[p.~35]{M97}, and observing that multiplication by \(-1\) does not
change the index of a vector field on an oriented surface, gives
\[
        1-\deg(d\mathcal G)_0
        =
        \chi(Y)
        =
        2-2g-p.
\]
Thus
\begin{equation}
\label{eq:green-differential-zero-count}
        \deg(d\mathcal G)_0=2g+p-1.
\end{equation}

Assume now that \(\Gamma\) and \(\rho\) are real analytic, so that
Lemma~\ref{lem:widom-local-regularity} applies to the extremal pair
\(F_0,\eta_0\). We compare the zeros of \(F_0\eta_0\) with those of
\(d\mathcal G\) by considering
\[
        Q:=\frac{F_0\eta_0}{d\mathcal G}.
\]
Both differentials have a simple pole at \(P_\infty\), and
\[
        \operatorname{Res}_{P_\infty}(F_0\eta_0)=\mu(L,\rho),
        \qquad
        \operatorname{Res}_{P_\infty}d\mathcal G=-1.
\]
The singularity of \(Q\) at \(P_\infty\) is therefore removable, with
\[
        Q(P_\infty)=-\mu(L,\rho)\ne0.
\]

Lemma \ref{lem:widom-local-regularity} shows that \(F_0\) and \(\eta_0\) extend
holomorphically across \(\Gamma\). Since neither \(F_0\) nor
\(d\mathcal G\) vanishes there, \(Q\) also extends across \(\Gamma\), and
its boundary zeros are precisely those of \(\eta_0\); in particular, they
have even order. Moreover, the boundary relation gives
\[
        Q
        =
        \frac{\iota^*(F_0\eta_0)}
             {\iota^*(d\mathcal G)}
        \le0
        \qquad\text{on \(\Gamma\)}.
\]
Thus \(Q\) has constant boundary argument away from its zeros, and the
even multiplicity allows the same branch of the argument to continue
through each boundary zero. The argument principle, applied after making
a small semicircular indentation at every such zero, gives
\[
        \deg\operatorname{div}_0(Q)=\deg\operatorname{div}_\infty(Q),
\]
where a boundary zero is counted with half its analytic multiplicity.
Since
\[
       \operatorname{div} (Q)=\operatorname{div}\bigl(F_0\eta_0\bigr)-\operatorname{div}(d\mathcal G)
\]
and the two differentials have the same pole divisor, it follows that
\[
        \deg\operatorname{div}_0\bigl(F_0\eta_0\bigr)
        =
        \deg\operatorname{div}_0(d\mathcal G).
\]
Together with \eqref{eq:green-differential-zero-count}, this yields
\begin{equation}
\label{eq:extremal-product-zero-count}
        \deg  \operatorname{div}_0\bigl(F_0\eta_0\bigr)
        =
        \deg\operatorname{div}_0(d\mathcal G)
        =
        2g+p-1,
\end{equation}
with boundary zeros counted with half their analytic multiplicity. 
Let
\[
A:=\operatorname{div}_0\bigl(F_0\eta_0\bigr),
        \qquad
        B:=\operatorname{div}_0(d\mathcal G).
\]
By the preceding zero count, these are effective divisors of the same
degree,
\[
        \deg A=\deg B=2g+p-1,
\]
where a boundary zero is counted with half its analytic multiplicity. Write
\[
        A=a_1+\cdots+a_{\widehat g},
        \qquad
        B=b_1+\cdots+b_{\widehat g},
        \qquad
        \widehat g:=2g+p-1.
\]
Thus a boundary zero of analytic order \(2m\) occurs \(m\) times among
the points \(a_j\).

Recall our convention for the Neumann factor \eqref{eq:expneu}
\[
        \Psi_{a,b}(p)
        :=
        \exp\!\bigl(
              N_Y(p;a,b)+i\widetilde N_Y(p;a,b)
        \bigr),
\]
where \(\widetilde N_Y(\cdot;a,b)\) is a local harmonic conjugate of
\(N_Y(\cdot;a,b)\). Since \(N_Y(\cdot;a,b)\) has a negative logarithmic
singularity at \(a\) and a positive logarithmic singularity at \(b\),
\(\Psi_{a,b}\) has a pole at \(a\) and a zero at \(b\); hence
\[
        \operatorname{div}\Psi_{a,b}=b-a.
\]
Its character is unitary, and its argument is constant on each component
of \(\Gamma\), because \(N_Y(\cdot;a,b)\) has vanishing normal derivative
there. If \(a\) or \(b\) lies on \(\Gamma\), the factor is understood in
the continuous boundary sense described above. By \eqref{eq:dz} analytic continuation
then has a double pole or zero at that point, which contributes one to
the divisor under the half-multiplicity convention.

Set
\[
        Q:=\frac{F_0\eta_0}{d\mathcal G}.
\]
Since
\[
        \operatorname{div}Q=A-B
\]
and
\[
        \operatorname{div}\Psi_{b_j,a_j}=a_j-b_j,
\]
the section
\begin{equation}
\label{eq:zero-free-neumann-product}
        U(p)
        :=
        Q(p)
        \left(
             \prod_{j=1}^{\widehat g}
             \Psi_{b_j,a_j}(p)
        \right)^{-1}
\end{equation}
has no zeros or poles. The local continuations established above show
that the cancellations also hold at boundary zeros, so \(U\) is smooth
and non-vanishing up to \(\Gamma\).

The boundary relation for the extremal pair shows that \(Q\) has constant
argument on every component of \(\Gamma\), and the same is true of each
Neumann factor. Consequently every local branch of \(U\) has constant
boundary argument on each \(\Gamma_k\).

It remains to show that \(U\) is constant. Let \(\widetilde U\) be its
lift to the universal cover, with character \(\chi_U\):
\[
        \widetilde U(\gamma p)
        =
        \chi_U(\gamma)\widetilde U(p),
        \qquad
        |\chi_U(\gamma)|=1.
\]
The function
\[
        u:=\log|\widetilde U|
\]
is therefore single-valued on \(Y\). Since \(U\) has no zeros or poles,
\(u\) is harmonic. On a boundary arc, choose a local logarithm
\[
        \log\widetilde U=u+iv.
\]
The boundary argument is constant, so the tangential derivative of \(v\)
vanishes. The Cauchy--Riemann equations then give
\[
        \frac{\partial u}{\partial n_{\rm out}}=0
        \qquad\text{on \(\Gamma\)}.
\]
Green's identity yields
\[
        \int_Y |du|^2
        =
        \int_\Gamma
        u\,\frac{\partial u}{\partial n_{\rm out}}\,ds
        =
        0.
\]
Thus \(u\) is constant. Hence \(\widetilde U\) has constant modulus, and
the open mapping theorem shows that \(\widetilde U\) itself is constant.
Therefore \(U\) has trivial character and is constant. We have proved
the product formula
\begin{equation}
\label{eq:widom-neumann-product}
        \frac{F_0\eta_0}{d\mathcal G}
        =
        C
        \prod_{j=1}^{\widehat g}
        \Psi_{b_j,a_j},
        \qquad
        C\in\mathbb C^\ast.
\end{equation}
\subsection{Proof of Theorem \ref{thm:extremal}}\label{ss:extremal}

We complete the proof by recovering the extremal section from its zeros. 
\begin{proof}

Let \(F_0\) be the normalized extremal in the character class \(\chi\), and write its zero divisor as
\[
        D_0:=\operatorname{div}_0(F_0)=z_1+\cdots+z_q,
\]
where the zeros are repeated according to multiplicity. By Lemma~\ref{lem:widom-local-regularity}, \(F_0\) has no zeros on \(\Gamma\), so \(D_0\) is supported in \(Y^\circ\). Moreover, every zero of \(F_0\) is a zero of \(F_0\eta_0\) of at least the same order. \eqref{eq:extremal-product-zero-count} therefore gives
\[
        q=\deg D_0\le \widehat g=2g+p-1.
\]

Normalize each Green factor \(\Phi(\,\cdot\,,z_j)\) so that
$ \Phi(P_\infty,z_j)>0,$
and set
\[
        \Phi(\,\cdot\,,D_0)
        :=\prod_{j=1}^{q}\Phi(\,\cdot\,,z_j).
\]
Recall that \(\Phi(\,\cdot\,,z_j)\) has a simple pole at \(z_j\); hence
$
        \operatorname{div}\Phi(\,\cdot\,,D_0)^{-1}=D_0.$
It follows that
\[
        S
        :=
        \frac{\mu(\rho,\chi)}{F_0}\,
        \Phi(\,\cdot\,,D_0)^{-1}
\]
is a nowhere-vanishing multiplicative holomorphic function on \(Y\): at every point of \(D_0\), the zero of \(\Phi(\,\cdot\,,D_0)^{-1}\) cancels the pole of \(F_0^{-1}\) with precisely the same multiplicity. On the boundary, using Lemma \ref{lem:widom-boundary-relations}
we obtain
\[
        |S|=\rho=|R|.
\]
The boundary relations, initially valid almost everywhere, hold everywhere here by the boundary regularity already proved.

The quotient \(S/R\) is therefore nowhere zero and multiplicative, with unitary character and boundary modulus one. Consequently,
\[
        \log\left|\frac{S}{R}\right|
\]
is a single-valued harmonic function on \(Y\) which vanishes on \(\Gamma\). The maximum principle gives
\[
        \left|\frac{S}{R}\right|\equiv1.
\]
A lift of \(S/R\) to the universal covering surface is thus a holomorphic function of constant modulus, and hence is constant by the open mapping theorem. Since both \(S(P_\infty)\) and \(R(P_\infty)\) are positive, this constant is \(1\). Therefore \(S=R\), and
\[
        F_0(p)
        =
        \mu(\rho,\chi)\,R(p)^{-1}
        \Phi(p,D_0)^{-1}
        =
        \mu(\rho,\chi)\,R(p)^{-1}
        \prod_{j=1}^{q}\Phi(p,z_j)^{-1}.
\]

Finally, evaluating at \(P_\infty\) and using \(F_0(P_\infty)=1\), we find
\[
        \mu(\rho,\chi)
        =
        R(P_\infty)\Phi(P_\infty,D_0)
        =
        R(P_\infty)\prod_{j=1}^{q}\Phi(P_\infty,z_j).
\]
Since the Green factors have been normalized to be positive at \(P_\infty\), symmetry of the Dirichlet Green function gives
\[
        \Phi(P_\infty,z_j)
        =
        \exp g_E(P_\infty,z_j)
        =
        \exp g_E(z_j,P_\infty).
\]
Hence
$$\mu(\rho,\chi)
        =
        R(P_\infty)
        \exp\left(\sum_{j=1}^{q}g_E(z_j)\right).$$
\end{proof}
\subsection{Proof of Theorem~\ref{thm:range}}\label{ss:range}
\begin{proof}
 \eqref{eq:nvlb} gives a local family of normalized sections \(\psi_\chi\), depending
continuously on \(\chi\), with  $\psi_\chi(P_\infty)=1.$
Put
\[
        q_\chi:=\frac{\psi_\chi}{\psi_{\chi_0}}.
\]
Then \(q_\chi\) is a nowhere-vanishing multiplicative holomorphic function
of character \(\chi-\chi_0\), normalized by \(q_\chi(P_\infty)=1\).
Since the construction is made using holomorphic differentials on a fixed
neighbourhood of \(Y\), one has
\[
        |q_\chi|\longrightarrow 1
        \qquad\text{uniformly on }Y
        \quad\text{as }\chi\longrightarrow\chi_0.
\]
Let
\[
        m_\chi:=\min_{\Gamma}|q_\chi|,
        \qquad
        M_\chi:=\max_{\Gamma}|q_\chi|.
\]
Multiplication by \(q_\chi\) sends a normalized section of \(L_{\chi_0}\)
to a normalized section of \(L_\chi\), while multiplication by
\(q_\chi^{-1}\) gives the inverse correspondence.  Testing the two extremal
problems against these sections gives
\begin{equation}\label{eq:range-continuity-comparison}
        m_\chi\,\mu(\rho,\chi_0)
        \le
        \mu(\rho,\chi)
        \le
        M_\chi\,\mu(\rho,\chi_0).
\end{equation}
Since \(m_\chi,M_\chi\to1\), the function
\[
        \chi\longmapsto\mu(\rho,\chi)
\]
is continuous.  The character group
\[
        \mathfrak X(Y)\simeq\mathbb T^{\widehat g},
        \qquad
        \widehat g=2g+p-1,
\]
is compact and connected; consequently, the range of
\(\mu(\rho,\chi)\) is a compact interval.

We next determine its lower endpoint.  Let
\(F\in \mathcal{H}^{\infty}(Y,L_\chi)\) satisfy \(F(P_\infty)=1\).  Since the
characters of \(F\) and \(R_\rho\) are unitary, the modulus of
\(FR_\rho\) is single-valued and subharmonic on \(Y\).  Its boundary
values satisfy
\[
        |FR_\rho|
        =
        |F|\rho
        \le
        \|F\|_{\rho,\infty}
        \qquad\text{\(ds\)-a.e. on }\Gamma.
\]
The maximum principle therefore gives
\[
        R_\rho(P_\infty)
        =
        |F(P_\infty)R_\rho(P_\infty)|
        \le
        \|F\|_{\rho,\infty}.
\]
Taking the infimum over all such \(F\), we obtain
\begin{equation}\label{eq:range-lower-bound}
        \mu(\rho,\chi)\ge R_\rho(P_\infty)
        \qquad
        \text{for every }\chi\in\mathfrak X(Y).
\end{equation}
For
\[
        \chi_-:=-\gamma_{R_\rho},
        \qquad
        F_-:=R_\rho(P_\infty)R_\rho^{-1},
\]
the section \(F_-\) has character \(\chi_-\), satisfies
\(F_-(P_\infty)=1\), and
\[
        |F_-|\rho=R_\rho(P_\infty)
        \qquad\text{on }\Gamma.
\]
Thus
\begin{equation}\label{eq:range-minimum}
        \min_{\chi\in\mathfrak X(Y)}\mu(\rho,\chi)
        =
        \mu(\rho,\chi_-)
        =
        R_\rho(P_\infty).
\end{equation}

We now turn to the upper endpoint.  Recall that
\[
        B:=\operatorname{div}_0(d\mathcal G)
        =
        b_1+\cdots+b_{\widehat g},
\]
where the zeros are repeated according to multiplicity, and normalize  so that
\[
        \Phi(P_\infty,b_\ell)>0,
        \qquad
        1\le \ell\le\widehat g.
\]
Set
\[
        \Phi_B
        :=
        \Phi(\,\cdot\,,B)
        =
        \prod_{\ell=1}^{\widehat g}\Phi(\,\cdot\,,b_\ell),
        \qquad
        \gamma_B:=\operatorname{char}\Phi_B
        =
        \sum_{\ell=1}^{\widehat g}\gamma_{b_\ell}.
\]

Fix \(\chi\in\mathfrak X(Y)\), and let \(F_0,\eta_0\) be the extremal
pair furnished by Theorem~\ref{thm:extremal}, with
\[
        \frac1{2\pi}\int_\Gamma|\eta_0|\rho^{-1}=1,
        \qquad
        \operatorname{Res}_{P_\infty}\eta_0
        =
        \mu(\rho,\chi)>0.
\]
Consider the multiplicative section
\begin{equation}\label{eq:range-H-definition}
        H_\chi
        :=
        \frac{\eta_0}{d\mathcal G}\,
        R_\rho^{-1}\Phi_B^{-1}.
\end{equation}
At a zero \(b_\ell\) of \(d\mathcal G\), the corresponding zero of
\(\Phi_B^{-1}\) cancels the pole of \(1/d\mathcal G\), with the same
multiplicity.  At \(P_\infty\), the simple poles of \(\eta_0\) and
\(d\mathcal G\) cancel.  Hence \(H_\chi\) is holomorphic on \(Y^\circ\),
including at \(P_\infty\), as a section of a flat unitary line bundle.
Moreover, Lemma~\ref{lem:widom-local-regularity} shows that
\(\eta_0/d\mathcal G\), and therefore \(H_\chi\), extends continuously
across \(\Gamma\).

On the boundary,
\[
        |R_\rho|=\rho,
        \qquad
        |\Phi_B|=1,
        \qquad
        |\iota^*(d\mathcal G)|=\beta_g.
\]
Thus, as an identity of positive boundary measures,
\begin{equation}\label{eq:range-H-boundary-measure}
        |H_\chi|\,\beta_g
        =
        |\eta_0|\rho^{-1}
        \qquad\text{on }\Gamma.
\end{equation}
Let \(u_\chi\) be the harmonic function on \(Y\) with boundary values
\(|H_\chi|\).  Since \(|H_\chi|\) is subharmonic, the subharmonic
maximum principle gives
\[
        |H_\chi|\le u_\chi
        \qquad\text{on }Y.
\]
The boundary form \(\beta_g/(2\pi)\) is harmonic measure at
\(P_\infty\).  Using \eqref{eq:range-H-boundary-measure} and the
normalization of \(\eta_0\), we obtain
\begin{equation}\label{eq:range-H-majorant}
\begin{aligned}
        u_\chi(P_\infty)
        &=
        \frac1{2\pi}\int_\Gamma |H_\chi|\,\beta_g  \\
        &=
        \frac1{2\pi}\int_\Gamma|\eta_0|\rho^{-1}
        =
        1.
\end{aligned}
\end{equation}
Consequently,
\[
        |H_\chi(P_\infty)|\le1.
\]

In the fixed local coordinate \(\zeta\) at \(P_\infty\), the residue
normalizations give
\[
        \eta_0
        =
        \left(
        \frac{\mu(\rho,\chi)}{\zeta}+O(1)
        \right)d\zeta,
        \qquad
        d\mathcal G
        =
        \left(
        -\frac1{\zeta}+O(1)
        \right)d\zeta.
\]
It follows from \eqref{eq:range-H-definition} that
\begin{equation}\label{eq:range-H-value}
        H_\chi(P_\infty)
        =
        -\frac{\mu(\rho,\chi)}
        {R_\rho(P_\infty)\Phi_B(P_\infty)}.
\end{equation}
Combining this with \(|H_\chi(P_\infty)|\le1\) yields
\begin{equation}\label{eq:range-upper-bound}
        \mu(\rho,\chi)
        \le
        R_\rho(P_\infty)\Phi_B(P_\infty)
        =
        R_\rho(P_\infty)
        \prod_{\ell=1}^{\widehat g}\Phi(P_\infty,b_\ell).
\end{equation}

It remains to show that equality is attained.  Put
\begin{equation}\label{eq:range-upper-character}
        \chi_*:=-\gamma_{R_\rho}-\gamma_B,
        \qquad
        M_*:=R_\rho(P_\infty)\Phi_B(P_\infty),
\end{equation}
and define
\begin{equation}\label{eq:range-upper-candidates}
        F_*:=M_*R_\rho^{-1}\Phi_B^{-1},
        \qquad
        \eta_*:=-R_\rho\Phi_B\,d\mathcal G.
\end{equation}
The section \(F_*\) has character
\[
        -\gamma_{R_\rho}-\gamma_B=\chi_*,
\]
and, by the definition of \(M_*\),
\[
        F_*(P_\infty)=1,
        \qquad
        |F_*|\rho=M_*
        \quad\text{on }\Gamma.
\]
The differential \(\eta_*\) has character
\[
        \gamma_{R_\rho}+\gamma_B=-\chi_*,
\]
so it takes values in \(K_Y\otimes L_{\chi_*}^{-1}(P_\infty)\).  At each
\(b_\ell\), the pole of \(\Phi_B\) is cancelled by the corresponding zero
of \(d\mathcal G\); the only remaining pole is the simple pole at
\(P_\infty\).  Thus
\[
        \eta_*
        \in
        \mathcal H^1
        \bigl(Y,K_Y\otimes L_{\chi_*}^{-1}(P_\infty)\bigr).
\]
On \(\Gamma\), using \(|R_\rho|=\rho\), \(|\Phi_B|=1\), and
\(|\iota^*(d\mathcal G)|=\beta_g\), we find
\begin{equation}\label{eq:range-dual-normalization}
        \frac1{2\pi}\int_\Gamma|\eta_*|\rho^{-1}
        =
        \frac1{2\pi}\int_\Gamma\beta_g
        =
        1.
\end{equation}
Since \(\operatorname{Res}_{P_\infty}d\mathcal G=-1\), the positivity
normalizations of \(R_\rho\) and \(\Phi_B\) give
\begin{equation}\label{eq:range-dual-residue}
        \operatorname{Res}_{P_\infty}\eta_*
        =
        R_\rho(P_\infty)\Phi_B(P_\infty)
        =
        M_*.
\end{equation}
Finally,
\[
        F_*\eta_*=-M_*\,d\mathcal G.
\]
Using \(\iota^*(d\mathcal G)=-i\beta_g\), we obtain the boundary sign
\begin{equation}\label{eq:range-boundary-certificate}
        \frac1{i}\,\iota^*(F_*\eta_*)
        =
        M_*\beta_g
        \ge0.
\end{equation}
Equations \eqref{eq:range-upper-candidates}--\eqref{eq:range-boundary-certificate}
show that \(F_*\) and \(\eta_*\) satisfy the converse boundary criterion,
Lemma~\ref{lem:widom-5-3-surface}.  Hence \(F_*\) is extremal in the
character class \(\chi_*\), and
\begin{equation}\label{eq:range-maximum}
        \max_{\chi\in\mathfrak X(Y)}\mu(\rho,\chi)
        =
        \mu(\rho,\chi_*)
        =
        M_*
        =
        R_\rho(P_\infty)\Phi_B(P_\infty).
\end{equation}

The  continuous image of the compact connected torus
\(\mathfrak X(Y)\) is a  compact interval.  By
\eqref{eq:range-minimum} and \eqref{eq:range-maximum}, its endpoints are
precisely
\[
        R_\rho(P_\infty)
        \quad\text{and}\quad
        R_\rho(P_\infty)\Phi_B(P_\infty).
\]
Finally, by the normalization of the Green factors and the symmetry of the
Dirichlet Green function,
\begin{equation}\label{eq:range-green-product}
\begin{aligned}
        \Phi_B(P_\infty)
        &=
        \prod_{\ell=1}^{\widehat g}\Phi(P_\infty,b_\ell) \\
        &=
        \exp\left(
        \sum_{\ell=1}^{\widehat g}g_E(P_\infty,b_\ell)
        \right)
        =
        \exp\left(
        \sum_{\ell=1}^{\widehat g}g_E(b_\ell)
        \right).
\end{aligned}
\end{equation}
This is exactly the interval stated in \eqref{eq:range}, and proves
Theorem~\ref{thm:range}.  
\end{proof}

\section{Szeg\H{o}--Widom asymptotics}
\label{sec:widom-asymptotics}

We retain the notation of Theorem~\ref{thm:sw} and put
\(D_n:=D_{\rho,\chi_n}\). By Theorem~\ref{thm:extremal},
\begin{equation}\label{eq:sw-extremal-product}
        F_n=\mu_nR_\rho^{-1}\Phi_{D_n}^{-1},
        \qquad
        \deg D_n\le\widehat g .
\end{equation}
Theorem~\ref{thm:range} gives, uniformly in \(n\),
\begin{equation}\label{eq:sw-mu-bounds}
        R_\rho(P_\infty)
        \le\mu_n
        \le R_\rho(P_\infty)\Phi_B(P_\infty).
\end{equation}
Only the local construction \eqref{eq:nvlb}, applied near the trivial
character, will be needed when zeros approaching the boundary are removed.

The following proposition is the analogue on \(X\) of Widom's analytic projection
\cite[Lemma~8.2, pp.~180--184]{Widom69}. The analytic collar used here permits
a fixed displacement of the contour and hence gives an error which is
exponentially small relative to the natural scale \(\capc(E)^n\).

\begin{proposition}
\label{lem:analytic-projection}
Let \(H_n\in \mathcal{H}^{\infty}(Y,L_E^{-n})\), with \(H_n(P_\infty)=1\), and suppose
that the \(H_n\) extend holomorphically through a fixed analytic collar of
\(\Gamma\), where they are uniformly bounded. Then there is a number
\(0<r<1\), independent of \(n\), such that, for all sufficiently large
\(n\), there is an admissible function
\[
        Q_n\in H^0\bigl(X,\mathcal O(nP_\infty)\bigr)
\]
with the following property: for every compact set
\(K\Subset Y^\circ\setminus\{P_\infty\}\), there is a constant \(C_K>0\)
such that
\begin{equation}\label{eq:sw-projection}
        \sup_{\Gamma\cup K}
        \left|Q_n-\capc(E)^n\Phi_E^nH_n\right|
        \le C_K\capc(E)^n r^n .
\end{equation}
Moreover,
\[
        Q_n(p)=\zeta(p)^{-n}+O\bigl(\zeta(p)^{-n+1}\bigr),
        \qquad p\to P_\infty .
\]
\end{proposition}

\begin{proof}
By Lemma~\ref{lem:analytic-widom-41}, \(g_E\) and \(\Phi_E\) extend through
an analytic collar of \(\Gamma\). We may assume that all the \(H_n\) are defined and uniformly
bounded there. After shrinking it once more, Hopf's lemma gives
\(dg_E\ne0\) throughout the collar. For sufficiently small
\(\varepsilon>0\), let \(\Gamma_-\) be the level curve
\(\{g_E=-\varepsilon\}\) on the other side of \(\Gamma\), and let
\(Y_-\) be the bordered surface bounded by \(\Gamma_-\) which contains
\(Y\). Then
\begin{equation}\label{eq:sw-inner-contour}
        |\Phi_E|=e^{-\varepsilon}=:r<1
        \qquad\text{on }\Gamma_- .
\end{equation}

Set \(S_n:=\capc(E)^n\Phi_E^nH_n\). The characters of \(\Phi_E^n\) and
\(H_n\) cancel, so \(S_n\) is a single-valued meromorphic function on
\(Y_-\), with its only pole at \(P_\infty\). The normalizations at
\(P_\infty\) give
\[
        S_n(p)=\zeta(p)^{-n}+O\bigl(\zeta(p)^{-n+1}\bigr),
        \qquad p\to P_\infty,
\]
while the collar bound and \eqref{eq:sw-inner-contour} give
\begin{equation}\label{eq:sw-source-bound}
        \sup_{\Gamma_-}|S_n|\le C\capc(E)^n r^n .
\end{equation}

Following the residue construction associated with this kernel introduced in \eqref{eq:CD1}, 
(compare
\cite[Prop.~2.4]{Bertola2021}), define
\[
        Q_n^0(p):=-\operatorname{Res}_{q=P_\infty}
        S_n(q)C_{\mathcal D}(q,p).
\]
For \(p\in Y_-\setminus\{P_\infty\}\), the residue theorem in the
\(q\)-variable gives
\begin{equation}\label{eq:sw-cauchy-identity}
        S_n(p)-Q_n^0(p)
        =\frac{1}{2\pi i}\int_{\Gamma_-}
        S_n(q)C_{\mathcal D}(q,p).
\end{equation}
The right-hand side is holomorphic in \(p\) near \(P_\infty\); hence
\(Q_n^0\) has the same principal part there as \(S_n\). Moreover, for
every compact set \(K\Subset Y_-\setminus\{P_\infty\}\),
\begin{equation}\label{eq:sw-raw-projection}
        \sup_K|S_n-Q_n^0|
        \le C_K\sup_{\Gamma_-}|S_n|.
\end{equation}
The only other possible poles of \(Q_n^0\) are the fixed points
\(d_1,\ldots,d_g\).

It remains to remove these auxiliary poles. Fix, once and for all, a local coordinate \(t_j\) centered at each
\(d_j\). Near \(p=d_j\)
we may write
\[
        C_{\mathcal D}(q,p)=\frac{b_j(q)}{t_j(p)}+O(1).
\]
Thus the coefficient of \(t_j^{-1}\) in the principal part of \(Q_n^0\) is
\[
        a_{j,n}:=-\operatorname{Res}_{q=P_\infty}S_n(q)b_j(q).
\]
The divisor properties of the kernel show that \(b_j\) has no pole in
\(Y_-\setminus\{P_\infty\}\). The residue theorem and
\eqref{eq:sw-source-bound} therefore give
\begin{equation}\label{eq:sw-auxiliary-coefficients}
        |a_{j,n}|\le C_j\sup_{\Gamma_-}|S_n|
        \le C_j\capc(E)^n r^n .
\end{equation}

What we do next is known as the Mittag-Leffler problem \cite[\S18]{Forster}.
  For
\[
        f\in L(n_0P_\infty+\mathcal D),
\]
write
\[
        f(p)=\frac{c_j(f)}{t_j(p)}+O(1),
        \qquad p\to d_j.
\]
Consider the linear map
\[
        \operatorname{pp}_{\mathcal D}:
        L(n_0P_\infty+\mathcal D)
        \longrightarrow \mathbb C^g,
        \qquad
        f\longmapsto
        \bigl(c_1(f),\ldots,c_g(f)\bigr).
\]
Its kernel is \(L(n_0P_\infty)\), since all the coefficients
\(c_j(f)\) vanish precisely when \(f\) is holomorphic at every \(d_j\).
Riemann-Roch gives
\[
        \dim L(n_0P_\infty+\mathcal D)=n_0+1,
        \qquad
        \dim L(n_0P_\infty)=n_0+1-g.
\]
Hence, by rank--nullity,
\[
        \dim\operatorname{im}(\operatorname{pp}_{\mathcal D})
        =(n_0+1)-(n_0+1-g)=g.
\]
Since   \(\mathbb C^g\) also has dimension \(g\), the map is
onto.  We may therefore choose fixed functions \(r_1,\ldots,r_g\) such
that
\[
        r_j(p)=\frac{\delta_{jk}}{t_k(p)}+O(1),
        \qquad p\to d_k,
\]
and such that each \(r_j\) has pole order at most \(n_0\) at
\(P_\infty\).
Define
\[
        \Delta_n:=\sum_{j=1}^g a_{j,n}r_j,
        \qquad
        Q_n:=Q_n^0-\Delta_n .
\]
Then \(Q_n\) has no pole at any \(d_j\), and, for \(n>n_0\),
\(Q_n\in H^0(X,\mathcal O(nP_\infty))\). Since \(\Delta_n\) has pole
order strictly smaller than \(n\) at \(P_\infty\), it does not change the
coefficient of \(\zeta^{-n}\). Finally,
\eqref{eq:sw-auxiliary-coefficients} gives
\[
        \sup_K|\Delta_n|\le C_K\sup_{\Gamma_-}|S_n|.
\]
Apply this estimate and \eqref{eq:sw-raw-projection} to the compact set
\(\Gamma\cup K\), and use \eqref{eq:sw-source-bound}; this proves
\eqref{eq:sw-projection} and the required normalization at \(P_\infty\).
\end{proof}

We shall next remove from \(D_n\) the zeros which approach \(\Gamma\). For
a planar domain, the character of a Green factor is expressed in
harmonic-measure coordinates in \cite[(4.3), pp.~140--141]{Widom69}.
Widom uses the resulting small characters to delete zeros near the boundary
and restore the prescribed character in
\cite[(8.7), pp.~185--186]{Widom69}. The corresponding fact on \(Y\) is as
follows.

\begin{lemma}
\label{lem:sw-near-character}
As \(a\to\Gamma\),
\[
        g_E(a)\longrightarrow0,
        \qquad
        \gamma_a\longrightarrow0
        \quad\text{in }\mathfrak X(Y),
\]
uniformly as \(a\) approaches \(\Gamma\).
\end{lemma}
\begin{proof}
The first assertion follows from the continuous extension of
\(g_E\) to \(\overline Y\), where it vanishes on \(\Gamma\). Since
\(\Gamma\) is compact, the convergence is uniform as \(a\) approaches
\(\Gamma\).

For the second assertion,  let \(c\) be one of the $\mathfrak a, \mathfrak b$
cycles, and choose relatively compact neighbourhoods
\[
        c\subset K\Subset U\Subset Y^\circ.
\]
For \(a\) sufficiently close to \(\Gamma\), one has
\(a\notin\overline U\), so the function
\[
        u_a(p):=g_E(p,a)
\]
is harmonic in \(p\) on \(U\). By symmetry of the Green function and
its joint continuity away from the diagonal,
\[
        \|u_a\|\longrightarrow0
\]
uniformly as \(a\to\Gamma\). The interior gradient estimate \cite[Chapter~2, Theorem~2.10]{GilbargTrudinger} for
harmonic functions therefore gives
\[
        \|d_p g_E(p,a)\|
        \le C_{K,U}\|g_E(\,\cdot\,,a)\|
        \longrightarrow0.
\]
Consequently,
\[
        \int_c *_p d_p g_E(p,a)\longrightarrow0,
\]
and hence the character of \(\Phi(\,\cdot\,,a)\) along \(c\)
tends to \(1\), uniformly as \(a\to\Gamma\).

For a boundary component \(\Gamma_k\), the corresponding period is,
up to the fixed orientation sign,
\[
        2\pi\omega_k(a),
\]
where \(\omega_k\) is the harmonic measure of \(\Gamma_k\). Each
\(\omega_k\) extends continuously to \(\overline Y\), with
\[
        \omega_k|_{\Gamma_j}=\delta_{kj}.
\]
Thus it is of the form
\[
        \exp\!\bigl(\pm2\pi i\,\omega_k(a)\bigr),
\]
extends continuously to \(\overline Y\) and is equal to \(1\) on
every boundary component. All the characters therefore tend uniformly to \(1\) as
\(a\to\Gamma\).
Since all of them tend
uniformly to \(1\), it follows that
\[
        \gamma_a\longrightarrow0
        \qquad\text{in }\mathfrak X(Y),
\]
uniformly as \(a\to\Gamma\).
\end{proof}
\begin{lemma}\
\label{lem:delete-near-boundary-zeros}
For every \(\varepsilon>0\) there is a \(\delta>0\) such that, for every
\(n\), one can find a section
\[
        G_{n,\delta}\in \mathcal{H}^{\infty}(Y,L_E^{-n}),
        \qquad G_{n,\delta}(P_\infty)=1,
\]
for which
\begin{equation}\label{eq:sw-deletion-bound}
        \|G_{n,\delta}\|_{\rho,\infty}
        \le e^\varepsilon\mu_n.
\end{equation}
Moreover, the sections \(G_{n,\delta}\) extend holomorphically through one
fixed analytic collar of \(\Gamma\), and are uniformly bounded there. The
collar and the bound may depend on \(\varepsilon\), but not on \(n\).
\end{lemma}

\begin{proof}
Write the product formula from Theorem~\ref{thm:extremal} as
\[
        F_n=\mu_nR_\rho^{-1}\Phi_{D_n}^{-1},
        \qquad \deg D_n\le\widehat g,
\]
and decompose
\[
        D_n=D_{n,\delta}^{\rm far}+D_{n,\delta}^{\rm near},
\]
where \(D_{n,\delta}^{\rm near}\) consists of the points whose distance
from \(\Gamma\) is less than \(\delta\), counted with multiplicity. Since
the degree of \(D_n\) is at most \(\widehat g\),
Lemma~\ref{lem:sw-near-character} gives
\[
        \xi_{n,\delta}
        :=-\gamma_{D_{n,\delta}^{\rm near}}
        \longrightarrow0
        \quad\text{in }\mathfrak X(Y)
\]
as \(\delta\downarrow0\), uniformly in \(n\).

For \(\delta\) sufficiently small, all the characters
\(\xi_{n,\delta}\) lie in the neighbourhood of the trivial character on
which \eqref{eq:nvlb} is defined. Let \(\psi_{n,\delta}\) be the resulting
nowhere-vanishing holomorphic section of \(L_{\xi_{n,\delta}}\), normalized
by
\[
        \psi_{n,\delta}(P_\infty)=1.
\]
We take \(\psi_{n,\delta}\equiv1\) when
\(D_{n,\delta}^{\rm near}=0\). The explicit construction in
\eqref{eq:nvlb} depends continuously on the character. Thus, after
decreasing \(\delta\) if necessary, these sections are uniformly bounded
in a fixed neighbourhood of \(Y\) and satisfy
\begin{equation}\label{eq:sw-small-character-bound}
        |\psi_{n,\delta}|\le e^\varepsilon
        \qquad\text{on }\Gamma,\quad n\ge1.
\end{equation}

Define
\begin{equation}\label{eq:sw-deleted-section}
        G_{n,\delta}
        :=
        \frac{\mu_nR_\rho^{-1}
        \Phi_{D_{n,\delta}^{\rm far}}^{-1}\psi_{n,\delta}}
        {\Phi_{D_{n,\delta}^{\rm near}}(P_\infty)}.
\end{equation}
Deleting the near factors changes the character by
\(\gamma_{D_{n,\delta}^{\rm near}}\), while
\(\psi_{n,\delta}\) has the opposite character. Hence
\[
\begin{aligned}
        \operatorname{char}G_{n,\delta}
        &=-\gamma_{R_\rho}
          -\gamma_{D_{n,\delta}^{\rm far}}
          +\xi_{n,\delta} \\
        &=-\gamma_{R_\rho}-\gamma_{D_n}
         =\chi_n=-n\gamma_E .
\end{aligned}
\]
Thus \(G_{n,\delta}\) is a holomorphic section of \(L_E^{-n}\).

The normalization follows directly from \(F_n(P_\infty)=1\). Indeed,
\[
        \mu_nR_\rho(P_\infty)^{-1}
        \Phi_{D_{n,\delta}^{\rm far}}(P_\infty)^{-1}
        =
        \Phi_{D_{n,\delta}^{\rm near}}(P_\infty),
\]
and therefore \eqref{eq:sw-deleted-section} gives
\(G_{n,\delta}(P_\infty)=1\). On \(\Gamma\), using
\(|R_\rho|=\rho\) and \(|\Phi_D|=1\), we obtain
\[
        |G_{n,\delta}|\rho
        =
        \mu_n
        \frac{|\psi_{n,\delta}|}
        {\Phi_{D_{n,\delta}^{\rm near}}(P_\infty)}.
\]
Our normalization of the Green factors gives
\[
        \Phi_{D_{n,\delta}^{\rm near}}(P_\infty)
        =
        \exp\!\left(
        \sum_{a\in D_{n,\delta}^{\rm near}}g_E(a)
        \right)
        \ge1.
\]
Consequently, \eqref{eq:sw-small-character-bound} proves
\eqref{eq:sw-deletion-bound} without any further loss in
\(\varepsilon\).

It remains only to check the collar bound. The formula for \(\mu_n\) and
Theorem~\ref{thm:range} give
\[
        \sum_{a\in D_n}g_E(a)
        =
        \log\frac{\mu_n}{R_\rho(P_\infty)}
        \le\log\Phi_B(P_\infty).
\]
Since \(g_E(a)\to+\infty\) as \(a\to P_\infty\), the points of \(D_n\)
remain uniformly away from \(P_\infty\). Once \(\delta\) is fixed, the
supports of \(D_{n,\delta}^{\rm far}\) therefore lie in a compact subset of
\(Y^\circ\) separated from \(\Gamma\).
Lemma~\ref{lem:analytic-widom-41} gives a fixed collar on which the
corresponding inverse Green factors are holomorphic and uniformly bounded.
The outer section \(R_\rho^{-1}\) is fixed and extends through an analytic
collar, the sections \(\psi_{n,\delta}\) are uniformly bounded there by
construction, and \(\mu_n\) is uniformly bounded by
Theorem~\ref{thm:range}. Finally, the denominator in
\eqref{eq:sw-deleted-section} is at least one. These observations give a
collar and a bound independent of \(n\), and complete the proof.
\end{proof}

\subsection{Proof of Theorem~\ref{thm:sw}: Part I}
\label{sec:proof-sw-part-i}

We first prove \eqref{eq:sw-norm} and the local convergence in
\eqref{eq:sw-fn}. The argument follows the first part of Widom's proof of
Theorem~8.3; see \cite[pp.~185--187]{Widom69}.

The lower bound is \eqref{eq:comparison}. For the reverse
inequality, fix \(\varepsilon>0\), and choose \(\delta>0\) and
\(G_{n,\delta}\) as in Lemma~\ref{lem:delete-near-boundary-zeros}. Thus
\(G_{n,\delta}(P_\infty)=1\), the sections \(G_{n,\delta}\) are uniformly
bounded in a fixed analytic collar of \(\Gamma\), and
\[
        \|G_{n,\delta}\|_{\rho,\infty}
        \le e^\varepsilon\mu_n .
\]
Proposition~\ref{lem:analytic-projection} then gives admissible functions
\(Q_n\in H^0(X,\cO(nP_\infty))\) and constants \(C_\varepsilon>0\) and
\(0<r_\varepsilon<1\), independent of \(n\), such that
\[
        \sup_\Gamma
        \bigl|Q_n-\capc(E)^n\Phi_E^nG_{n,\delta}\bigr|
        \le C_\varepsilon\capc(E)^n r_\varepsilon^n .
\]
Since \(|\Phi_E|=1\) on \(\Gamma\), after absorbing
\(\max_\Gamma\rho\) into \(C_\varepsilon\), we obtain
\[
        M_{n,\rho}
        \le \|Q_n\rho\|_\Gamma
        \le \capc(E)^n
        \bigl(e^\varepsilon\mu_n+C_\varepsilon r_\varepsilon^n\bigr).
\]
Together with \eqref{eq:comparison}, this gives
\begin{equation}\label{eq:sw-part-i-squeeze}
        0\le \capc(E)^{-n}M_{n,\rho}-\mu_n
        \le (e^\varepsilon-1)\mu_n+C_\varepsilon r_\varepsilon^n .
\end{equation}
Theorem~\ref{thm:range} bounds \(\mu_n\) independently of \(n\). Taking the
upper limit in \eqref{eq:sw-part-i-squeeze} and then letting
\(\varepsilon\downarrow0\), we conclude that
\begin{equation}\label{eq:sw-part-i-norm}
        \capc(E)^{-n}M_{n,\rho}-\mu_n\longrightarrow0,
\end{equation}
which is \eqref{eq:sw-norm}.

Set
\[
        \mathcal T_n
        :=\capc(E)^{-n}T_{n,\rho}\Phi_E^{-n}.
\]
The normalization of \(\Phi_E\) at \(P_\infty\), together with the monic
principal part of \(T_{n,\rho}\), shows that the singularity of
\(\mathcal T_n\) at \(P_\infty\) is removable and that
\(\mathcal T_n(P_\infty)=1\). Hence
\(\mathcal T_n\in \mathcal{H}^{\infty}(Y,L_{\chi_n})\), where
\(\chi_n=-n\gamma_E\), and \eqref{eq:sw-part-i-norm} yields
\begin{equation}\label{eq:sw-part-i-asymptotic-extremal}
        \|\mathcal T_n\|_{\rho,\infty}
        =\capc(E)^{-n}M_{n,\rho}
        =\mu_n+o(1).
\end{equation}

It remains to compare \(\mathcal T_n\) with the exact extremal \(F_n\). Let
\((n_k)\) be any subsequence. Since \(\mathfrak X(Y)\) is compact, after
passing to a further subsequence we may assume that
\(\chi_{n_k}\to\chi\). To place the sections in one fixed bundle, put
\(\alpha_k:=\chi-\chi_{n_k}\). The local construction
\eqref{eq:nvlb}, applied near the trivial character, gives nowhere-vanishing
sections \(\psi_k\) of \(L_{\alpha_k}\), normalized by
\(\psi_k(P_\infty)=1\), such that
\begin{equation}\label{eq:sw-part-i-character-correction}
        |\psi_k|^{\pm1}\longrightarrow1
        \quad\text{uniformly on }\Gamma,
        \qquad
        \psi_k^{\pm1}\longrightarrow1
        \quad\text{locally uniformly on }Y^\circ
\end{equation}
in the local trivializations furnished by that construction. Thus
\[
        \widehat F_k:=\psi_kF_{n_k},
        \qquad
        \widehat{\mathcal T}_k:=\psi_k\mathcal T_{n_k}
\]
are sections of the fixed bundle \(L_\chi\), and both take the value \(1\)
at \(P_\infty\). By \eqref{eq:range-continuity-comparison},
\(\mu_{n_k}\to\mu(\rho,\chi)\); moreover,
\eqref{eq:sw-part-i-asymptotic-extremal} and
\eqref{eq:sw-part-i-character-correction} give
\begin{equation}\label{eq:sw-part-i-corrected-norms}
        \|\widehat F_k\|_{\rho,\infty}
        \le \mu(\rho,\chi)+o(1),
        \qquad
        \|\widehat{\mathcal T}_k\|_{\rho,\infty}
        \le \mu(\rho,\chi)+o(1).
\end{equation}

The maximum principle, applied to the single-valued moduli of
\(R_\rho\widehat F_k\) and \(R_\rho\widehat{\mathcal T}_k\), now gives
uniform bounds on compact subsets of \(Y^\circ\). Montel's theorem in local
frames therefore gives a further subsequence for which
\[
        \widehat F_k\longrightarrow F,
        \qquad
        \widehat{\mathcal T}_k\longrightarrow G
\]
locally uniformly on \(Y^\circ\), where \(F\) and \(G\) are holomorphic
sections of \(L_\chi\) and \(F(P_\infty)=G(P_\infty)=1\). The boundary-norm
estimate passes to the limit without requiring convergence on \(\Gamma\):
for every \(p\in Y^\circ\), the maximum principle and
\eqref{eq:sw-part-i-corrected-norms} give
\[
        |R_\rho(p)F(p)|\le\mu(\rho,\chi),
        \qquad
        |R_\rho(p)G(p)|\le\mu(\rho,\chi).
\]
Thus \(R_\rho F\) and \(R_\rho G\) are bounded by \(\mu(\rho,\chi)\) in
\(Y^\circ\). Since \(R_\rho\) is nowhere zero, it follows that
\(F,G\in \mathcal{H}^{\infty}(Y,L_\chi)\); thus
\[
        \|F\|_{\rho,\infty}\le\mu(\rho,\chi),
        \qquad
        \|G\|_{\rho,\infty}\le\mu(\rho,\chi).
\]
The reverse inequalities follow from the definition of \(\mu(\rho,\chi)\),
since both sections are normalized at \(P_\infty\). Thus \(F\) and \(G\) are
both extremal in the character class \(\chi\), and the uniqueness assertion
in Theorem~\ref{thm:extremal} gives \(F=G\). It follows from
\eqref{eq:sw-part-i-character-correction} that, for every
\(K\Subset Y^\circ\),
\[
        \sup_K|\mathcal T_{n_k}-F_{n_k}|
        \le
        \sup_K|\psi_k|^{-1}
        \sup_K|\widehat{\mathcal T}_k-\widehat F_k|
        \longrightarrow0.
\]

Every subsequence has a further subsequence with this property. If the full
sequence failed to converge locally uniformly, there would be a compact set
\(K\Subset Y^\circ\), a number \(c>0\), and a subsequence for which
\(\sup_K|\mathcal T_n-F_n|\ge c\); the preceding argument would then give a
further subsequence on which this supremum tends to zero, a contradiction.
Therefore
\[
        \frac{T_{n,\rho}}{\capc(E)^n\Phi_E^n}-F_n
        \longrightarrow0
\]
locally uniformly on \(Y^\circ\), with the removable value at
\(P_\infty\) included. This proves \eqref{eq:sw-fn} on every relatively
compact subset of \(Y^\circ\). The stronger convergence up to \(\Gamma\)
under the zero-separation hypothesis will be proved in Part~II.

\subsection{Proof of Theorem~\ref{thm:sw}: Part II}
\label{sec:proof-sw-part-ii}

Assume now that the divisors \(D_n\) stay a positive distance from
\(\Gamma\). We follow the positive-measure argument in the final part of
Widom's proof of Theorem~8.3; see
\cite[especially (8.11)--(8.17), pp.~188--190]{Widom69}.

It is useful first to make the compactness supplied by the hypothesis
explicit. By \eqref{eq:mu-formula} and Theorem~\ref{thm:range},
\[
        \sum_{a\in D_n}g_E(a)
        =\log\frac{\mu_n}{R_\rho(P_\infty)}
\]
is bounded independently of \(n\). Since \(g_E(a)\to+\infty\) as
\(a\to P_\infty\), the zero-separation hypothesis implies that
\(\operatorname{supp}D_n\) is contained in one compact set
\(K_0\Subset Y^\circ\setminus\{P_\infty\}\). The product formula
\eqref{eq:product}, Lemma~\ref{lem:analytic-widom-41}, and the bound
\(\deg D_n\le\widehat g\) therefore give a fixed two-sided collar \(U\) of
\(\Gamma\) on which
\begin{equation}\label{eq:sw-part-ii-F-collar}
        \sup_n\sup_U
        \bigl(|F_n|+|F_n^{-1}|+|dF_n|\bigr)<\infty.
\end{equation}

We may consequently apply Proposition~\ref{lem:analytic-projection} directly to
\(F_n\). Together with \eqref{eq:comparison} and the boundary
identity \(|F_n|\rho=\mu_n\), which is pointwise here by
Lemma~\ref{lem:widom-local-regularity}, this gives constants \(C>0\) and
\(0<r<1\), independent of \(n\), such that
\begin{equation}\label{eq:sw-part-ii-sharp-norm}
        0\le \capc(E)^{-n}M_{n,\rho}-\mu_n\le Cr^n.
\end{equation}

Choose a normalized dual extremal \(\eta_n\), so that
\[
        \frac1{2\pi}\int_\Gamma|\eta_n|\rho^{-1}=1,
        \qquad
        \operatorname*{Res}_{P_\infty}\eta_n=\mu_n,
\]
and put
\begin{equation}\label{eq:sw-part-ii-dual-probability}
        d\lambda_n
        :=\frac{1}{2\pi i\,\mu_n}\,\iota^*(F_n\eta_n).
\end{equation}
By Lemma~\ref{lem:widom-boundary-relations}, this is a positive probability
measure on \(\Gamma\). Write
\[
        (F_n\eta_n)_0
        =a_{1,n}+\cdots+a_{\widehat g,n},
\]
with boundary zeros counted by half multiplicity. The Neumann product proved
above gives
\[
        \frac{F_n\eta_n}{d\mathcal G}
        =C_n\prod_{\ell=1}^{\widehat g}
        \Psi_{b_\ell,a_{\ell,n}},
\]
where a factor with \(a_{\ell,n}=b_\ell\) is understood to be one. Write
\(\beta_g=\beta\,ds\), with \(\beta>0\). Since
\(\iota^*(d\mathcal G)=-i\beta_g\) and the boundary quotient is
non-positive, taking absolute values and using
\(\lambda_n(\Gamma)=1\) gives the exact formula
\begin{equation}\label{eq:sw-part-ii-density}
        d\lambda_n(q)
        =
        \frac{w_n(q)\,ds_q}
        {\displaystyle\int_\Gamma w_n\,ds},
        \qquad
        w_n(q)
        :=
        \beta(q)\prod_{\ell=1}^{\widehat g}
        \exp N_Y(q;b_\ell,a_{\ell,n}).
\end{equation}

We shall need one consequence of this formula. There are constants
\(c>0\) and \(L_0>0\), independent of \(n\), such that every boundary arc
\(I\) with \(|I|\le L_0\) satisfies
\begin{equation}\label{eq:sw-part-ii-mass-lower}
        \lambda_n(I)\ge c|I|^{2\widehat g+1}.
\end{equation}
Indeed, choose a finite collection of nested analytic boundary charts
\(V_0\Subset V_1\Subset V\), disjoint from the fixed points \(b_\ell\), so
that every sufficiently short arc is contained in one of the smallest
charts. By \eqref{eq:dz}, a factor whose moving zero lies in \(V_1\) has,
on \(V_0\), the form
\[
        u(s)\bigl((s-\alpha)^2+\tau^2\bigr),
        \qquad \tau\ge0,
\]
with \(u\) bounded above and below by positive constants. A factor whose
moving zero lies outside \(V_1\) is bounded above and below on \(V_0\),
again uniformly in its parameter. The coordinate length is uniformly
comparable with arclength. If \(L\) is the coordinate length of \(I\) and
\(r_0\le\widehat g\) factors of the first kind occur, a change of variables
therefore reduces the required lower estimate to
\[
        L^{2r_0+1}\int_0^1|P(x)|^2\,dx,
\]
where \(P\) is a monic polynomial of degree \(r_0\). The infimum of this
integral over all monic polynomials of a fixed degree is positive: it is the
squared distance in \(L^2[0,1]\) from \(x^{r_0}\) to the polynomials of
degree at most \(r_0-1\), with this space interpreted as \(\{0\}\) when
\(r_0=0\). Since \(0\le r_0\le\widehat g\) and \(L\le1\), this proves
\[
        \int_I w_n\,ds\ge c|I|^{2\widehat g+1}.
\]
The same parameter-uniform estimates give
\(\int_\Gamma w_n\,ds\le C\), and
\eqref{eq:sw-part-ii-mass-lower} follows from
\eqref{eq:sw-part-ii-density}. This is the bordered-surface form of the
estimate used by Widom after (8.16).

Set
\[
        W_n:=\frac{\mathcal T_n}{F_n}.
\]
By \eqref{eq:sw-part-ii-F-collar}, \(W_n\) is a single-valued holomorphic
function in a fixed collar of \(\Gamma\). On the boundary,
\eqref{eq:sw-part-ii-sharp-norm}, \(|F_n|\rho=\mu_n\), and the uniform lower
bound for \(\mu_n\) give
\begin{equation}\label{eq:sw-part-ii-W-bound}
        |W_n|\le1+Cr^n.
\end{equation}
Moreover, the residue theorem and the normalization
\(\mathcal T_n(P_\infty)=1\) give
\[
        \int_\Gamma W_n\,d\lambda_n
        =
        \frac{1}{2\pi i\,\mu_n}
        \int_\Gamma\mathcal T_n\eta_n
        =1.
\]
Since \(\lambda_n\) is a probability measure, it follows that
\begin{equation}\label{eq:sw-part-ii-L2}
\begin{aligned}
        \int_\Gamma|W_n-1|^2\,d\lambda_n
        &=
        \int_\Gamma|W_n|^2\,d\lambda_n-1 \\
        &\le Cr^n.
\end{aligned}
\end{equation}

We next record the boundary derivative estimate
\begin{equation}\label{eq:sw-part-ii-derivative}
        \left|\frac{dW_n}{ds}\right|\le Cn
        \qquad\text{on }\Gamma.
\end{equation}
This requires a two-sided estimate. From
\eqref{eq:sw-part-ii-sharp-norm} and \(\min_\Gamma\rho>0\),
\[
        \sup_\Gamma|T_{n,\rho}|
        \le C\capc(E)^n.
\]
The same bound holds on the \(E\)-side by the maximum principle in each
component of \(E\), while Theorem~\ref{thm:surface-BW} gives
\[
        |T_{n,\rho}(p)|
        \le C\capc(E)^n e^{ng_E(p)}
        \qquad (p\in Y^\circ).
\]
Let \(g_E^{\rm ext}\) denote the reflected continuation of \(g_E\) through
\(\Gamma\). Since \(|\Phi_E|=e^{g_E^{\rm ext}}\), the preceding two
estimates show that
\begin{equation}\label{eq:sw-part-ii-thin-collar}
        \sup_{\{|g_E^{\rm ext}|\le c/n\}}
        |\mathcal T_n|
        \le C.
\end{equation}
Hopf's lemma gives \(|dg_E^{\rm ext}|>0\) near \(\Gamma\), so this strip has
geometric width comparable to \(1/n\). Cauchy's estimate in local flat
frames therefore gives
\[
        \left|\frac{d\mathcal T_n}{ds}\right|
        \le Cn
        \qquad\text{on }\Gamma,
\]
and \eqref{eq:sw-part-ii-derivative} follows from
\eqref{eq:sw-part-ii-F-collar}.

Suppose that \(W_n\not\to1\) uniformly on \(\Gamma\). Then, for some
\(\varepsilon>0\) and along a subsequence, there are points
\(q_n\in\Gamma\) such that
\[
        |W_n(q_n)-1|\ge\varepsilon.
\]
By \eqref{eq:sw-part-ii-derivative}, there is an arc \(I_n\) centered at
\(q_n\), of length \(c_\varepsilon/n\), on which
\[
        |W_n-1|\ge\varepsilon/2.
\]
Hence \eqref{eq:sw-part-ii-mass-lower} gives
\[
        \int_\Gamma|W_n-1|^2\,d\lambda_n
        \ge
        c_\varepsilon n^{-(2\widehat g+1)},
\]
contradicting the exponential estimate \eqref{eq:sw-part-ii-L2}. Thus
\(W_n\to1\) uniformly on \(\Gamma\). Since the sections \(F_n\) are
uniformly bounded there,
\[
        \mathcal T_n-F_n
        =
        F_n(W_n-1)
        \longrightarrow0
        \qquad\text{uniformly on }\Gamma.
\]
For each fixed \(n\), the difference \(\mathcal T_n-F_n\) is a holomorphic
section of the same flat unitary bundle \(L_E^{-n}\). Its norm is therefore
single-valued and subharmonic, and the maximum principle gives
\[
        \sup_Y|\mathcal T_n-F_n|
        \le
        \sup_\Gamma|\mathcal T_n-F_n|
        \longrightarrow0.
\]
This proves the final assertion of \eqref{eq:sw-fn} and completes the proof
of Theorem~\ref{thm:sw}.

\subsection{Proof of Corollary~\ref{cor:limitpoints}}

The zero-separation hypothesis is not needed for this consequence. Let
\[
        G_\gamma
        :=
        \overline{\{\chi_n:n\ge1\}}
        \subset\mathfrak X(Y).
\]
By \eqref{eq:sw-norm} and the continuity of
\(\chi\mapsto\mu(\rho,\chi)\), the set of limit points of
\(M_{n,\rho}/\capc(E)^n\) is exactly
\begin{equation}\label{eq:sw-limit-set}
        \{\mu(\rho,\chi):\chi\in H\}.
\end{equation}
The closure of a cyclic semigroup in a compact group is a subgroup; hence
\(H\) is a closed subgroup of the torus \(\mathfrak X(Y)\). Such a subgroup
has finitely many connected components, each a translate of a subtorus;
this is the group-theoretic observation used by Widom in the proof of
Theorem~8.4 \cite[pp.~190--191]{Widom69}. The continuous image of each
component under \(\mu(\rho,\cdot)\) is therefore a closed interval. After
combining overlapping intervals, \eqref{eq:sw-limit-set} is a finite union
of pairwise disjoint closed intervals, all contained in the interval
\eqref{eq:range}.

If \(1,\theta_1,\ldots,\theta_{\widehat g}\) are linearly independent over
\(\mathbb Q\), Kronecker's theorem makes the orbit \(\{\chi_n\}\) dense in
\(\mathfrak X(Y)\). Thus \(H=\mathfrak X(Y)\), and
Theorem~\ref{thm:range} shows that \eqref{eq:sw-limit-set} is the whole
interval \eqref{eq:range}. If instead every \(\theta_j\) belongs to
\(q^{-1}\mathbb Z\), then \(\chi_{n+q}=\chi_n\), so the limit set contains
at most \(q\) points. This proves the corollary.

\section{System of equations}
\label{subsec:widom-zero-systems}

We now extract the equations for the zeros from the product identities proved above. With \(A=a_1+\cdots+a_{\widehat g}\) and \(B=b_1+\cdots+b_{\widehat g}\) as above, put
\[
        \mathcal P_A:=\prod_{\ell=1}^{\widehat g}\Psi_{b_\ell,a_\ell},
        \qquad
        \operatorname{div}\mathcal P_A=A-B,
        \qquad
        \frac{F_0\eta_0}{d\mathcal G}
        =C\mathcal P_A,\quad C\in\mathbb C^*.
\]
 The character of every Neumann factor is trivial on the boundary cycles, so \(\operatorname{char}\mathcal P_A\) belongs to
\[
        \mathfrak X_0(Y):=
        \bigl\{
        \xi\in\mathfrak X(Y):
        \xi([\Gamma_k])=1,\ k=1,\ldots,p
        \bigr\}
        \simeq\mathbb T^{2g}.
\]
Since \(F_0\eta_0/d\mathcal G\) is single-valued and has the same boundary argument on every component of \(\Gamma\), the first system is
\begin{subequations}\label{eq:widom-first-system}
\begin{align}
        \operatorname{char}\mathcal P_A
        =\sum_{\ell=1}^{\widehat g}
          \operatorname{char}\Psi_{b_\ell,a_\ell}
        &=0
        &&\text{in }\mathfrak X_0(Y),
        \label{eq:widom-first-character}\\
        \left.
        \frac{\mathcal P_A}{|\mathcal P_A|}
        \right|_{\Gamma_k}
        &=
        \left.
        \frac{\mathcal P_A}{|\mathcal P_A|}
        \right|_{\Gamma_p},
        && k=1,\ldots,p-1.
        \label{eq:widom-first-boundary}
\end{align}
\end{subequations}
The second line is read only after the first has made \(\mathcal P_A\) single-valued; at a boundary zero its phase is continued through the zero, whose reflected order is even. The two lines contribute, respectively, \(2g\) handle conditions and \(p-1\) relative boundary-phase conditions, hence \(\widehat g=2g+p-1\) real equations.

Let \(D=(F_0)_0=z_1+\cdots+z_q\). As observed above, \(D\leq A\), the divisor \(D\) is supported in \(Y^\circ\), and \(q\leq\widehat g\). Taking characters in
\[
        F_0=\mu(\rho,\chi)R_\rho^{-1}\Phi_D^{-1}
\]
gives the second system
\begin{equation}\label{eq:widom-second-system}
        \gamma_D
        :=
        \sum_{j=1}^q\gamma_{z_j}
        =
        -\chi-\gamma_{R_\rho}
        \qquad\text{in }\mathfrak X(Y).
\end{equation}

These conditions are also sufficient. Indeed, suppose that an effective divisor \(A\) of degree \(\widehat g\) and an interior sub-divisor \(D\leq A\), both disjoint from \(P_\infty\), satisfy \eqref{eq:widom-first-system} and \eqref{eq:widom-second-system}. Choose \(\lambda_A\in\mathbb T\) so that \(\lambda_A\mathcal P_A\geq0\) on \(\Gamma\), and set
\begin{equation}\label{eq:widom-zero-reconstruction}
        M_D:=R_\rho(P_\infty)\Phi_D(P_\infty),
        \qquad
        F_D:=M_D R_\rho^{-1}\Phi_D^{-1},
        \qquad
        \eta_{A,D}:=
        -c_A\lambda_A\mathcal P_A
        \frac{d\mathcal G}{F_D},
\end{equation}
where \(c_A>0\) is chosen so that
\[
        \frac1{2\pi}
        \int_\Gamma|\eta_{A,D}|\rho^{-1}=1.
\]
Then \(F_D\) has character \(\chi\), \(F_D(P_\infty)=1\), and \(|F_D|\rho=M_D\) on \(\Gamma\), whereas \(\eta_{A,D}\) has character \(-\chi\) and
\[
        \operatorname{div}\eta_{A,D}=A-D-P_\infty,
        \qquad
        \frac{\iota^*(F_D\eta_{A,D})}
             {\iota^*(d\mathcal G)}
        =
        -c_A\lambda_A\mathcal P_A
        \leq0.
\]
In particular, \(D\leq A\) gives
\[
        \eta_{A,D}\in
        \mathcal H^1
        \bigl(Y,K_Y\otimes L_\chi^{-1}(P_\infty)\bigr).
\]
Lemma~\ref{lem:widom-5-3-surface} therefore applies and yields
\[
        F_D=F_0,
        \qquad
        \mu(\rho,\chi)
        =
        M_D
        =
        R(P_\infty)\Phi_D(P_\infty).
\]
Consequently the two systems are necessary and sufficient. 

There is an Abel--Jacobi interpretation, but not a reduction of the whole problem to ordinary Jacobi inversion. Let \(\widehat Y\) be the Schottky double of \(Y\), of genus \(\widehat g\), let \(\sigma\) be its reflection, and put
\[
        A^{\mathrm d}:=A+\sigma_*A,
        \qquad
        B^{\mathrm d}:=B+\sigma_*B,
\]
with the boundary multiplicities doubled. If \eqref{eq:widom-first-system} holds, a unimodular multiple of \(\mathcal P_A\) is non-negative on \(\Gamma\) and reflects to a meromorphic function on \(\widehat Y\) with divisor \(A^{\mathrm d}-B^{\mathrm d}\). Thus, writing \(\mathcal A_{\widehat Y}\) for the Abel--Jacobi map, Abel's theorem gives
\begin{equation}\label{eq:widom-abel-condition}
        \mathcal A_{\widehat Y}
        \bigl(A^{\mathrm d}-B^{\mathrm d}\bigr)
        =0
        \qquad\text{in }\operatorname{Jac}(\widehat Y).
\end{equation}
This is the real principal-divisor condition underlying the first system, and \eqref{eq:widom-second-system} may likewise be viewed as a real Abel equation through the Green-character map \(D\mapsto\gamma_D\). Ordinary Jacobi inversion, however, does not impose the common boundary sign, keep the divisor in the chosen half of the double, or enforce the variable degree and the sub-divisor constraint \(D\leq A\). The simultaneous problem is therefore a constrained real Jacobi inversion problem, not the classical one; compare \cite[Secs.~20--21]{Forster}. This distinction is already present in Widom: he reduces the later \(L^2\) problem, but not the \(\mathcal{H}^{\infty}\) system of Theorem~5.4, to classical Jacobi inversion \cite[pp.~139, 160, 171--172]{Widom69}.

\section{Worked out examples}
\label{sec:genus-one-examples}

This section is meant to be read as a calculation.  We choose the simplest
surface of positive genus, write down its meromorphic functions in concrete
coordinates, and solve several Chebyshev problems exactly.  No theta functions
are needed.  The only recurring device is the elementary average of a function
over the two points \(z\) and \(-z\).

The examples exhibit three different phenomena.  A complete inverse image under
\(\wp\) reduces an even-order problem to an ordinary planar Chebyshev problem.
An odd-order problem on the same set becomes a weighted planar problem.    Thus the
examples do more than illustrate the root asymptotic: they show, at finite
degree, how the topology of the torus enters through the Green function.

\subsection{The square torus and  Weierstrass \texorpdfstring{$\wp$}{wp}}

Let
\[
        X=\C/(\Z+i\Z),\qquad P_\infty=0,
\]
and use \(z\) as the local coordinate at \(0\).  We write
\[
        x(z):=\wp(z),
        \qquad
        y(z):=-\frac12\wp'(z).
\]
The factor \(-1/2\) is convenient: both \(x^m\) and \(yx^m\) then have leading
Laurent coefficient one at \(0\).

For the square lattice put
\[
        e:=\wp(1/2)=\frac{\Gamma(1/4)^4}{8\pi}>0.
\]
The standard  formulas give \(g_3=0\) and \(g_2=4e^2\); with
half-period \(\omega_1=1/2\), this is exactly
\cite[23.5.3--23.5.4]{NISTDLMF}.  The Weierstrass equation therefore reads
\begin{equation}\label{eq:lemniscatic-curve}
        y^2=x^3-e^2x=x(x-e)(x+e).
\end{equation}
The map \(x:X\to\widehat\C\) has degree two.  Its finite branch values are
\(e,0,-e\), attained at \(1/2,(1+i)/2,i/2\), respectively.  At \(0\),
\begin{equation}\label{eq:wp-local-expansions}
        x(z)=z^{-2}+\frac{g_2}{20}z^2+O(z^6),
        \qquad
        y(z)=z^{-3}-\frac{g_2}{20}z+O(z^5).
\end{equation}

Every elliptic function for this lattice is rational in \(\wp\) and \(\wp'\).
If it has no pole away from \(0\), no denominators are needed after using
\eqref{eq:lemniscatic-curve}: it has a unique form
\begin{equation}\label{eq:elliptic-reduced-form}
        A(x)+yB(x),\qquad A,B\in\C[x].
\end{equation}
Equivalently, the affine curve \(X\setminus\{0\}\) has coordinate ring
\[
        \C[x,y]/(y^2-x^3+e^2x).
\]
This is the usual algebraic description of elliptic functions; see, for
example, \cite[Ch.~III]{FarkasKra}.

For later use we record exactly which monomials occur.  For \(N\ge2\),
\begin{equation}\label{eq:elliptic-RR-basis}
\begin{split}
        L(N\cdot0)
        ={}&\operatorname{span}\{1,x,\ldots,x^{\lfloor N/2\rfloor}\}\\
        &\oplus
        y\,\operatorname{span}
        \{1,x,\ldots,x^{\lfloor(N-3)/2\rfloor}\},
\end{split}
\end{equation}
where the second line is absent when \(N<3\).  The displayed functions have the
distinct pole orders
\[
        0,2,3,4,\ldots,N.
\]
They are therefore linearly independent, and their number is
\(N=\dim L(N\cdot0)\) by Riemann--Roch.  This proves
\eqref{eq:elliptic-RR-basis}.

At every admissible pole order we now have a canonical monic function:
\begin{equation}\label{eq:canonical-elliptic-monomials}
        H_{2m}=x^m\quad(m\ge1),
        \qquad
        H_{2m+3}=yx^m\quad(m\ge0).
\end{equation}
There is no admissible function of order \(1\).

The decomposition \eqref{eq:elliptic-reduced-form} is useful because
\[
        x(-z)=x(z),\qquad y(-z)=-y(z).
\]
Thus \(A(x)\) is the even part and \(yB(x)\) is the odd part.  This simple
observation will help us in what follows.  Notice that
\eqref{eq:elliptic-RR-basis} concerns scalar meromorphic functions, the
competitors in the Chebyshev problem.  Sections of a nontrivial flat line bundle
are different objects and generally require sigma or theta functions.

\subsection{Pullbacks}

For a compact non-polar set \(K\subset\C\), put
\[
        E=x^{-1}(K).
\]
 If \(z\in E\), then \(-z\in E\) as
well.  For \(F=A(x)+yB(x)\),
\begin{equation}\label{eq:two-sheet-average}
        \frac{F(z)+F(-z)}2=A(x(z)),
        \qquad
        \frac{F(z)-F(-z)}2=y(z)B(x(z)).
\end{equation}
Both projections have norm at most \(\|F\|_E\).

Write \(t_m^{\C}(K)\) for the minimum norm of a monic polynomial of degree
\(m\) on \(K\), and let \(Q_m\) denote its Chebyshev polynomial.

\begin{proposition}
\label{prop:trace-transfer}
For every \(m\ge1\),
\begin{equation}\label{eq:trace-transfer-norm}
        t_{2m}(x^{-1}(K),0)=t_m^{\C}(K),
\end{equation}
and \(Q_m(x)\) is an -\(2m\)-th  extremal.
\end{proposition}

\begin{proof}
Let \(F\) be any admissible function of order \(2m\).  By
\eqref{eq:elliptic-RR-basis}, its even part in
\eqref{eq:two-sheet-average} is \(A(x)\), where \(A\) is monic of degree \(m\).
Hence
\[
        t_m^{\C}(K)\le \|A\|_K
        =\|A(x)\|_{x^{-1}(K)}
        \le\|F\|_{x^{-1}(K)}.
\]
Taking the infimum over \(F\) gives one inequality.  Conversely, \(Q_m(x)\) is
admissible of order \(2m\), and
\[
        \|Q_m(x)\|_{x^{-1}(K)}=\|Q_m\|_K=t_m^{\C}(K).
\]
This gives the reverse inequality.
\end{proof}

This is the degree-two instance of averaging over all points in a fibre.  
Polynomial inverse-image identities for extremal norms have a substantial
literature: see Kamo and Borodin \cite{KamoBorodin1994}, Peherstorfer and
Steinbauer \cite{PeherstorferSteinbauer2001}, and Christiansen, Simon, and
Zinchenko \cite[Theorem~6.1]{CSZ3}.  The weighted
version is Theorem~2.4 of Novello, Schiefermayr, and Zinchenko
\cite{NovelloSchiefermayrZinchenko2021}.  Proposition
\ref{prop:trace-transfer} is the same averaging idea written in the concrete
two-sheeted coordinates of this torus.

The norm identity also gives
\begin{equation}\label{eq:finite-map-capacity}
        \capc_X(x^{-1}(K),0)=\capc_{\C}(K)^{1/2},
\end{equation}
\begin{equation}\label{eq:finite-map-Widom}
        W_{2m}^X(x^{-1}(K),0)=W_m^{\C}(K).
\end{equation}
Indeed, take \(2m\)-th roots in \eqref{eq:trace-transfer-norm} and use the
planar root asymptotic together with Theorem~\ref{thm:main-root}.

When \(K\) is regular, this can be checked directly.  Let \(g_K\) be the planar
Green function with pole at infinity, extended by zero to \(K\), then
\begin{equation}\label{eq:finite-map-green}
        g_{x^{-1}(K)}(z)=\frac12 g_K(x(z)).
\end{equation}
The right-hand side is harmonic off \(x^{-1}(K)\cup\{0\}\), vanishes on the
compact set, and near \(0\) has the expansion
\[
        \frac12g_K(x(z))
        =-\log|z|-\frac12\log\capc_{\C}(K)+o(1).
\]
Uniqueness of the Green function proves \eqref{eq:finite-map-green} and again
gives \eqref{eq:finite-map-capacity}; compare the classical polynomial preimage
formula in \cite[Theorem~5.2.5]{Ransford}.

\subsection{Example 1: the inverse image of a disk}

Fix \(0<R<e\), and let
\begin{equation}\label{eq:canonical-elliptic-disk}
        E_R:=\{z\in X:|x(z)|\le R\}.
\end{equation}
The planar disk contains exactly one branch value of \(x\), namely \(0\), and
its boundary contains none.  The monodromy around \(0\) interchanges the two
sheets, so \(E_R\) and its boundary are both connected.  The 
Riemann--Hurwitz formula gives
\[
        \chi(E_R)=2\chi(\overline{\mathbb D})-1=1.
\]
Thus \(E_R\) is an embedded closed disk.  Since \(|w|=R\) contains no critical
value, its boundary is real analytic.  This example satisfies the hypotheses
of Assumption \ref{ass1} with one boundary component.

\subsubsection*{Even orders}

The monic planar extremal of degree \(m\) on \(|w|\le R\) is \(w^m\).
Proposition~\ref{prop:trace-transfer} gives
\begin{equation}\label{eq:elliptic-disk-even-extremal}
        T_{2m}^{(R)}(z)=x(z)^m=\wp(z)^m,
        \qquad
        t_{2m}(E_R,0)=R^m.
\end{equation}
The only zero is \((1+i)/2\), with multiplicity \(2m\).  Square-lattice
rotation gives
\[
        x(iz)=-x(z),
        \qquad
        T_{2m}^{(R)}(iz)=(-1)^mT_{2m}^{(R)}(z).
\]
Moreover,
\begin{equation}\label{eq:elliptic-disk-green-capacity}
        g_{E_R}(z)=\frac12\log\frac{|x(z)|}{R}
        \quad(z\in X\setminus E_R),
        \qquad
        \capc(E_R)=R^{1/2},
        \qquad
        W_{2m}(E_R,0)=1.
\end{equation}

The extremal is unique.  The exterior \(\{|x|>R\}\) is connected because it
contains the unique pole of \(x\) (see \S10.5 for more details).  If \(F\) were another minimizer, then
\(F/x^m\) would be holomorphic on the exterior, extend to \(0\) with value one,
and have boundary modulus at most one.  The maximum-modulus principle forces
\(F/x^m\equiv1\).

The same formula displays the Green function's monodromy.  On the universal cover of the
exterior take
\begin{equation}\label{eq:elliptic-disk-Phi}
        \Phi_{E_R}=\left(\frac{x}{R}\right)^{1/2}.
\end{equation}
Its character has order two, since \(\Phi_{E_R}^2=x/R\) is single-valued, but it
is not trivial.  Indeed,
\[
        (x)=2[(1+i)/2]-2[0].
\]
A global square root of \(x\) would have one simple zero and one simple pole and
would define a degree-one meromorphic map from a genus-one surface to the
sphere, which is impossible.  An exterior square root is impossible for the
same reason: \(E_R\) is a disk containing the unique double zero of \(x\), so
\(x\) has a holomorphic square root inside.  Since the boundary is connected,
an exterior root could, after changing its sign, be glued to the interior root.

Thus
\begin{equation}\label{eq:elliptic-disk-Widom-section}
\begin{split}
        \capc(E_R)^{-2m}T_{2m}^{(R)}\Phi_{E_R}^{-2m}
        &=R^{-m}x^m\left(\frac{x}{R}\right)^{-m}\\
        &=1.
\end{split}
\end{equation}
The maximum principle gives \(\mu(L)\ge1\) for every normalized Widom class, so
\begin{equation}\label{eq:elliptic-disk-mu}
        \mu(L_{E_R}^{-2m})=W_{2m}(E_R,0)=1.
\end{equation}

\subsubsection*{Odd orders}

If \(F\) is admissible of order \(2m+3\), its odd part in
\eqref{eq:two-sheet-average} is \(yq(x)\), where \(q\) is monic of degree \(m\).
This projection does not increase the norm, and every \(yq(x)\) is admissible.
Hence
\begin{equation}\label{eq:odd-weighted-reduction}
        t_{2m+3}(x^{-1}(K),0)
        =\inf_{\substack{q\text{ monic}\\ \deg q=m}}
          \max_{w\in K}|w(w^2-e^2)|^{1/2}|q(w)|.
\end{equation}
Thus odd surface orders become weighted planar problems.  The weight is not
inserted by hand: it is \(|y|\), obtained from
\eqref{eq:lemniscatic-curve}.

For the disk, the first two odd cases are elementary:
\begin{equation}\label{eq:elliptic-disk-odd-examples}
\begin{array}{c|c|c}
N&T_N^{(R)}&t_N(E_R,0)\\ \hline
3&y&\sqrt{R(R^2+e^2)}\\[2mm]
5&xy&R^{3/2}\sqrt{R^2+e^2}.
\end{array}
\end{equation}
For \(N=3\), there is no polynomial \(q\) to choose.  For \(N=5\), write
\(q(w)=w+c\).  The disk and the weight are invariant under \(w\mapsto-w\), so
averaging \(q(w)\) with \(-q(-w)\) replaces it by \(w\) without increasing the
weighted norm.

It remains to verify the norms.  The maximum-modulus principle puts the maximum
on \(|w|=R\), and there
\[
        |w^2-e^2|^2
        =R^4+e^4-2R^2e^2\cos(2\arg w)
        \le (R^2+e^2)^2.
\]
Equality holds at \(w=\pm iR\).  Therefore
\[
        \max_{|w|\le R}|w(w^2-e^2)|=R(R^2+e^2),
\]
which gives the norm of \(y\).  For \(xy\), observe that
\[
        |xy|^2=|w^3(w^2-e^2)|.
\]
The maximum-modulus principle again puts the maximum on \(|w|=R\); there the
additional factor is \(R^2\).  This gives the displayed norm of \(xy\).  Since
\(\capc(E_R)=R^{1/2}\),
\begin{equation}\label{eq:elliptic-disk-low-odd-Widom}
        W_3(E_R,0)=W_5(E_R,0)
        =\frac{\sqrt{R^2+e^2}}{R}>1.
\end{equation}
Thus every even Widom factor is one, whereas the first two odd factors are
strictly larger than one.

For completeness, the root asymptotic can be checked directly.  Let \(\tau_m\)
denote the right-hand side of \eqref{eq:odd-weighted-reduction} for the disk.
The weight is bounded by a constant \(C_R\) on the disk, so the trial polynomial
\(q(w)=w^m\) gives \(\tau_m\le C_RR^m\), and hence
\[
        \limsup_{m\to\infty}\tau_m^{1/m}\le R.
\]
Conversely, for \(0<r<R\), we have by Cauchy's theorem
\(\max_{|w|=r}|q(w)|\ge r^m\), while
\[
        \min_{|w|=r}|w(w^2-e^2)|^{1/2}
        \ge\sqrt{r(e^2-r^2)}.
\]
Thus \(\liminf\tau_m^{1/m}\ge r\).  Letting \(r\uparrow R\) gives
\(\tau_m^{1/m}\to R\), and the even and odd subsequences together yield
\[
        \lim_{N\to\infty}t_N(E_R,0)^{1/N}
        =R^{1/2}=\capc(E_R).
\]

\subsection{Example 2: the inverse image of an interval}

Let
\[
        K_I=[2e,3e],
        \qquad
        E_I:=x^{-1}(K_I).
\]
The interval misses all three finite branch values, so \(E_I\) is the disjoint
union of two analytic arcs exchanged by \(z\mapsto-z\).

The capacity of \(K_I\) is \(e/4\).  Its monic Chebyshev polynomial of degree
\(m\) is
\[
        Q_m(w)=2\left(\frac e4\right)^m
        T_m\left(\frac{2w}{e}-5\right),
\]
where \(T_m(\cos\theta)=\cos(m\theta)\).  Since the leading coefficient of
\(T_m\) is \(2^{m-1}\), the coefficient of \(w^m\) here is
\[
        2\left(\frac e4\right)^m2^{m-1}
        \left(\frac2e\right)^m=1.
\]
Proposition~\ref{prop:trace-transfer} gives
\begin{equation}\label{eq:elliptic-interval-extremal}
        T_{2m}^{I}(z)
        =2\left(\frac e4\right)^m
        T_m\left(\frac{2\wp(z)}e-5\right),
\end{equation}
\begin{equation}\label{eq:elliptic-interval-values}
        t_{2m}(E_I,0)=2\left(\frac e4\right)^m,
        \qquad
        \capc(E_I)=\sqrt{\frac e4},
        \qquad
        W_{2m}(E_I,0)=2.
\end{equation}
In particular,
\[
        t_{2m}(E_I,0)^{1/(2m)}
        =2^{1/(2m)}\sqrt{\frac e4}
        \longrightarrow\capc(E_I).
\]

The lifted function retains the classical zeros and alternating values.  Its
zeros are the two inverse images of each point
\[
        \frac{5e}{2}+\frac e2
        \cos\frac{(2j-1)\pi}{2m},
        \qquad j=1,\ldots,m,
\]
and all \(2m\) zeros are simple.  At both inverse images of
\[
        \frac{5e}{2}+\frac e2\cos\frac{j\pi}{m},
        \qquad j=0,\ldots,m,
\]
the value is \(2(e/4)^m(-1)^j\).

We also check uniqueness on the surface.  Suppose \(F\) is another minimizer
and set \(M=2(e/4)^m\).  Its even projection is the unique planar interval
extremal.  At an alternation point \(w_j\), with inverse images \(p_j,-p_j\),
\[
        \frac{F(p_j)+F(-p_j)}2=Q_m(w_j)=(-1)^jM,
        \qquad |F(\pm p_j)|\le M.
\]
Equality in the triangle inequality forces
\(F(p_j)=F(-p_j)=(-1)^jM\).  Thus \(F-T_{2m}^{I}\) has at least
\(2(m+1)\) distinct zeros.  Its leading term has cancelled, so it lies in
\(L((2m-1)\cdot0)\); a nonzero function there has at most \(2m-1\) zeros.
This contradiction proves uniqueness.

For a general real interval \(K=[A,B]\), \(A<B\), the same calculation gives
\begin{equation}\label{eq:general-elliptic-interval}
        T_{2m}(z)
        =\frac{(B-A)^m}{2^{2m-1}}
        T_m\left(\frac{2\wp(z)-A-B}{B-A}\right),
\end{equation}
\[
        t_{2m}=2\left(\frac{B-A}{4}\right)^m,
        \qquad
        \capc\bigl(\wp^{-1}([A,B])\bigr)
        =\sqrt{\frac{B-A}{4}}.
\]
If \(c=(A+B)/2\) and \(r=(B-A)/2\), the first three functions are
\[
\begin{split}
        T_2&=x-c,\\
        T_4&=(x-c)^2-\frac{r^2}{2},\\
        T_6&=(x-c)^3-\frac{3r^2}{4}(x-c).
\end{split}
\]
These formulas directly check the leading coefficient and scaling in
\eqref{eq:general-elliptic-interval}.

\subsection{Example 3: lemniscates of \texorpdfstring{\(\wp'\)}{wp'}}

We first record the elementary argument.  Let \(H\) be any admissible function
of order \(d\), and put
\[
        E_{H,R}:=\{p\in X:|H(p)|\le R\},\qquad R>0.
\]
The exterior \(\Omega_R=\{|H|>R\}\) is connected.  Indeed, the restriction of
\(H\) to each component of
\(H^{-1}(\{|w|>R\}\cup\{\infty\})\) is a proper open map and hence maps that
component onto \(\{|w|>R\}\cup\{\infty\}\).  Every component would therefore
contain a pole, whereas \(0\) is the unique pole of \(H\).

If \(F\) is admissible of order \(dm\), then \(F/H^m\) is holomorphic on
\(\Omega_R\) and extends to \(0\) with value one.  On the boundary,
\[
        \left|\frac{F}{H^m}\right|
        \le\frac{\|F\|_{E_{H,R}}}{R^m}.
\]
The maximum-modulus principle gives \(\|F\|_{E_{H,R}}\ge R^m\).  Equality is
attained by \(H^m\), and equality in the maximum principle proves uniqueness.
Moreover,
\[
        g_{E_{H,R}}=\frac1d\log\frac{|H|}{R}
        \quad\text{on }\Omega_R.
\]
Comparison at \(0\) proves the following formulas.

\begin{corollary}
\label{cor:surface-lemniscates}
For every \(m\ge1\), \(H^m\) is the unique order-\(dm\) Chebyshev extremal on
\(E_{H,R}\), and
\begin{equation}\label{eq:lemniscate-values}
        t_{dm}(E_{H,R},0)=R^m,
        \qquad
        \capc(E_{H,R})=R^{1/d},
        \qquad
        W_{dm}(E_{H,R},0)=1.
\end{equation}
\end{corollary}

This is the surface version of the classical solid-lemniscate calculation; the
planar equality cases are discussed in
\cite[Theorem~1.2 and Corollary~3.5]{CSZ3}.

Now take
\[
        H=y=-\frac12\wp'=z^{-3}+O(z).
\]
It is a degree-three meromorphic map with the unique pole \(3[0]\).  For
\[
        E_R^{(3)}:=\{z\in X:|y(z)|\le R\},
\]
Corollary~\ref{cor:surface-lemniscates} gives
\begin{equation}\label{eq:wp-prime-family}
        T_{3m}^{(3)}(z)=y(z)^m
        =\left(-\frac12\wp'(z)\right)^m,
\end{equation}
\begin{equation}\label{eq:wp-prime-values}
        t_{3m}(E_R^{(3)},0)=R^m,
        \qquad
        \capc(E_R^{(3)})=R^{1/3},
        \qquad
        W_{3m}(E_R^{(3)},0)=1.
\end{equation}

The zeros of \(y\) are the three nonzero half-periods, so the zero divisor of
\(y^m\) contains each with multiplicity \(m\).  Also
\[
        y(iz)=iy(z),
        \qquad
        T_{3m}^{(3)}(iz)=i^mT_{3m}^{(3)}(z).
\]
The zeros are simple.  Indeed,
\[
        y'=-\frac12\wp''=-3x^2+e^2,
\]
so \(y'=e^2\) at \(x=0\) and \(y'=-2e^2\) at \(x=\pm e\).  The
inverse-function theorem now shows that, for every sufficiently small \(R>0\),
\(E_R^{(3)}\) is the disjoint union of three embedded analytic closed disks.
This fixed family contains odd pole orders whenever \(m\) is odd.

There is an exact lemniscate example at every admissible order.  With \(H_N\)
as in \eqref{eq:canonical-elliptic-monomials}, set
\[
        E_{N,R}:=\{|H_N|\le R\}.
\]
Then \(H_N\) is the unique order-\(N\) extremal on \(E_{N,R}\), and
\[
        t_N(E_{N,R},0)=R,
        \qquad
        \capc(E_{N,R})=R^{1/N},
        \qquad
        W_N(E_{N,R},0)=1.
\]
If \(|w|=R\) contains no critical value of \(H_N\), then
\(\partial E_{N,R}\) is real analytic.  Here the set varies with \(N\).
\subsection*{Acknowledgements}
I thank Benedikt Buchecker and Aron Wennman for helpful discussions.
Sampad Lahiry acknowledges financial support from the International Research Training Group (IRTG) between KU Leuven and University of Melbourne and the  Melbourne Research Scholarship of University of Melbourne.

\end{document}